\documentclass[english,letter paper,12pt,leqno]{article}
\usepackage{array}
\usepackage{stmaryrd}
\usepackage{amsmath, amscd, amssymb, mathrsfs, accents, amsfonts,amsthm}
\usepackage[all]{xy}
\xyoption{rotate}
\usepackage{dsfont}
\usepackage{bussproofs}
\usepackage{tikz}
\usepackage{epigraph}
\usepackage{tikz-cd}
\usepackage{bbm}
\usepackage{csquotes}

\definecolor{Myblue}{rgb}{0,0,0.6}
\usepackage[a4paper,colorlinks,citecolor=Myblue,linkcolor=Myblue,urlcolor=Myblue,pdfpagemode=None]{hyperref}
\usepackage[final]{showlabels} 

\SelectTips{cm}{}

\newtheorem{theorem}{Theorem}[section]
\newtheorem{thm}[theorem]{Theorem}
\newtheorem{proposition}[theorem]{Proposition}
\newtheorem{lemma}[theorem]{Lemma}
\newtheorem{corollary}[theorem]{Corollary}

\newcommand{\tagarray}{\mbox{}\refstepcounter{equation}$(\theequation)$}

\newtheoremstyle{example}{\topsep}{\topsep}
	{}
	{}
	{\bfseries}
	{.}
	{2pt}
	{\thmname{#1}\thmnumber{ #2}\thmnote{ #3}}
	
	\theoremstyle{example}
	\newtheorem{definition}[theorem]{Definition}
	\newtheorem{example}[theorem]{Example}
	\newtheorem{remark}[theorem]{Remark}

\numberwithin{equation}{section}

\newcommand{\call}[1]{\mathcal{#1}}

\newcommand{\comment}[1]{}

\def\be{\begin{equation}}
\def\ee{\end{equation}}

\def\ldot{\,.\,}
\def\FV{\operatorname{FV}}

\def\typearrow{\rightarrow}
\def\imp{\supset}

\begin{document}

\def\ScoreOverhang{1pt}

\makeatletter
\DeclareRobustCommand{\rvdots}{%
  \vbox{
    \baselineskip4\p@\lineskiplimit\z@
    \kern-\p@
    \hbox{}\hbox{.}\hbox{.}\hbox{.}
  }}
\makeatother

\newcommand{\proofvdots}[1]{\overset{\displaystyle #1}{\rvdots}}
\def\Res{\res\!}
\newcommand{\ud}{\mathrm{d}}
\newcommand{\Ress}[1]{\res_{#1}\!}
\newcommand{\cat}[1]{\mathcal{#1}}
\newcommand{\lto}{\longrightarrow}
\newcommand{\xlto}[1]{\stackrel{#1}\lto}
\newcommand{\md}[1]{\mathscr{#1}}
\def\sus{\l}
\def\l{\,|\,}
\def\sgn{\textup{sgn}}
\def\samp{\zeta}
\def\Samp{Z}
\def\traff{N}
\newcommand{\bb}[1]{\mathbb{#1}}
\newcommand{\scr}[1]{\mathscr{#1}}

\title{Gentzen-Mints-Zucker duality}
\author{Daniel Murfet, William Troiani}

\maketitle

\begin{abstract}
The Curry-Howard correspondence is often described as relating proofs (in intutionistic natural deduction) to programs (terms in simply-typed lambda calculus). However this narrative is hardly a perfect fit, due to the computational content of cut-elimination and the logical origins of lambda calculus. We revisit Howard's work and interpret it as an isomorphism between a category of proofs in intuitionistic sequent calculus and a category of terms in simply-typed lambda calculus. In our telling of the story the fundamental duality is not between proofs and programs but between \emph{local} (sequent calculus) and \emph{global} (lambda calculus or natural deduction) points of view on a common logico-computational mathematical structure.

\end{abstract}

\tableofcontents

\setlength{\epigraphwidth}{0.85\textwidth}
\epigraph{There may, indeed, be other applications of the system than its use as a logic.}{A.~Church, \textsl{Postulates for the foundation of logic}}

\section{Introduction}

Sequent calculus and lambda calculus were both invented in the context of logical investigations, the former by Gentzen as a language of proofs \cite{gentzen} and the latter by Church as a language of functions \cite{church}. The computational content of these calculi emerged at different times, with the relevance of $\beta$-reduction of lambda terms to the emerging theory of computation being more quickly realised than the relevance of cut-elimination. By now it is clear that both calculi have logical and computational aspects, and that the two calculi are deeply related to one another. In this paper we revisit this relationship in the form of an isomorphism of categories (Theorem \ref{gentzen_mints_zucker})
\be\label{eq:curry_howard_duality}
\xymatrix@C+2pc{
F_\Gamma: \cat{S}_\Gamma \ar[r]^-{\cong} & \cat{L}_\Gamma
}
\ee
for each sequence $\Gamma$ of formulas (n\'ee types) where $\cat{S}_\Gamma$ is a category of proofs in intuitionistic sequent calculus (defined in Section \ref{section:sequent_calc}) and $\cat{L}_\Gamma$ is a category of simply-typed lambda terms (defined in Section \ref{section:lambda_calc}). Both categories have the same set of objects, viewed either as the formulas of intuitionistic propositional logic or simple types. The set of morphisms $\cat{S}_\Gamma(p,q)$ is the set of proofs of $\Gamma \vdash p \imp q$ up to an equivalence relation $\sim_p$ generated by cut-elimination transformations and commuting conversions together with a small number of additional natural relations, while $\cat{L}_\Gamma(p,q)$ is the set of simply-typed lambda terms of type $p \rightarrow q$ whose free variables have types taken from $\Gamma$, taken up to $\beta \eta$-equivalence. We refer to this isomorphism of categories and the normal form theorem which refines it (Theorem \ref{theorem:normal_form_prop}) as the \emph{Gentzen-Mints-Zucker duality} between sequent calculus and lambda calculus. The name reflects work by Zucker \cite{zucker} and Mints \cite{mints}, elaborated below.

A duality consists of two different points of view on the same object \cite{atiyah}. The greater the difference between the two points of view, the more informative is the duality which relates them. Such correspondences are important because two independent discoveries of the same structure is strong evidence that the structure is natural. The above duality is interesting precisely because sequent calculus proofs and lambda terms are not tautologically the same thing: for example the cut-elimination relations are fundamentally \emph{local} while the $\beta$-equivalence relation is \emph{global} (see Section \ref{section:comparison}). In this sense sequent calculus and lambda calculus are respectively local and global points of view on a common logico-computational mathematical structure.
\\

This duality is related to, but distinct from, the Curry-Howard correspondence. The precise relationship is elaborated in Section \ref{section:ch_intro} below, but broadly speaking it is captured by conceptual diagram of Figure \ref{figure:triumvirate}. The Curry-Howard correspondence gives a bijection between natural deduction proofs and lambda terms, while Gentzen-Mints-Zucker duality reveals that $\beta\eta$-normal lambda terms are \emph{normal forms} for sequent calculus proofs modulo an equivalence relation generated by pairs that are well-motivated from the point of view of the Brouwer-Heyting-Kolmogorov interpretation of intuitionistic proof.
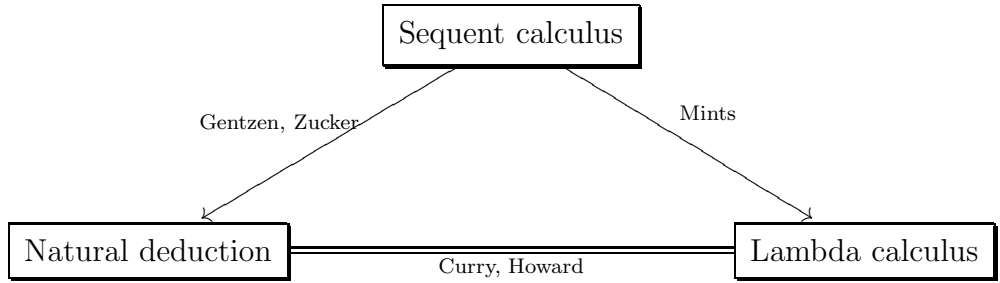
\begin{figure}
\begin{center}
\quad\xymatrix@C+1pc@R+3pc{
& *++[F-,]{\text{Sequent calculus}} \ar[dl(0.85)]_-[@!30]{\text{Gentzen, Zucker}} \ar[dr(0.85)]^-[@]{\text{Mints}}\\
*++[F-,]{\text{Natural deduction}} \ar@{=}[rr]_-[@]{\text{Curry, Howard}} & & *++[F-,]{\text{Lambda calculus}}
}
\end{center}
\caption{The relationship between three logico-computational calculi.}
\label{figure:triumvirate}
\end{figure}

\subsection{The Curry-Howard correspondence}\label{section:ch_intro}

The relationship between proofs and lambda terms (or ``programs'') has become widely known as the Curry-Howard correspondence following Howard's work \cite{howard}, although an informal understanding of the computational content of intuitionistic proofs has older roots in the Brouwer-Heyting-Kolmogorov interpretation \cite{troelstra}. The correspondence has been so influential in logic and computer science that the Curry-Howard correspondence as \emph{philosophy} now overshadows the Curry-Howard correspondence as a \emph{theorem}.

As a theorem, the Curry-Howard correspondence is the observation that the formulas of implicational propositional logic are the same as the types of simply-typed lambda calculus, and that there is a surjective map from the set of all proofs of a sequent $\Gamma \vdash \alpha$ in ``sequent calculus'' to the set of all lambda terms of type $\alpha$ with free variables in $\Gamma$ (due to Howard \cite[\S 3]{howard} building on ideas of Curry and Tait). This map is not a bijection, and as such does not represent the best possible statement about the relationship between proofs and lambda terms. The problem is that there are \emph{two} natural continuations of Howard's work, depending on how one interprets the somewhat vague notion of \emph{proof} in \cite[\S 1]{howard}. 

The vagueness is due to the fact that the system called ``sequent calculus'' by Howard is actually a hybrid of intuitionistic sequent calculus in the sense of Gentzen's LJ (there are weakening, exchange and contraction rules in Howard's system) and natural deduction in the sense of Gentzen \cite{gentzen} and Prawitz \cite{prawitz} (Howard has an elimination rule for implication rather than the left introduction rule of sequent calculus). The use of an elimination rule makes the connection to lambda terms straightforward, and the inclusion of structural rules places the correspondence in its proper context as a relationship between proofs with hypotheses and lambda terms with free variables \cite[\S 3]{howard}.

In resolving this ambiguity subsequent authors writing about the correspondence have almost universally decided that \emph{proof} means \emph{natural deduction proof}; see for example \cite[\S 4.8,\S 7.4, \S 7.6]{sorensen} and \cite{wadler}. The correspondence may then be interpreted as a bijection between lambda terms and natural deduction proofs \cite[\S 6.5]{selinger}. This bijection represents the natural conclusion of one line of development starting with \cite{howard} and henceforth we refer to this bijection as the \emph{Curry-Howard correspondence}. Despite its philosophical importance, this correspondence is not \emph{mathematically} of great interest, because natural deduction and lambda calculus are so similar that the bijection is close to tautological (in the case of closed terms it is left as an exercise in one standard text \cite[Ex. 4.8]{sorensen}). 
\\

In the present work we investigate the second natural continuation of \cite{howard}, which takes seriously the structural rules in Howard's ``sequent calculus'' and seeks to give a bijection between sequent calculus proofs and lambda terms. As soon as explicit structural rules are introduced into proofs, however, there will be multiple proofs that map to the same lambda term, and so for there to be a bijection between proofs and lambda terms, \emph{proof} must mean \emph{equivalence class of preproofs modulo some relation}. If this relation is simply ``maps to the same lambda term'' then what we have constructed is merely a surjective map from proofs to lambda terms, which is hardly more than what is in \cite{howard}. Hence in this second line of thought, the \emph{identity of proofs} becomes a central concern. 

Consequently one of the contributions of this paper is to give explicit generating relations for a relation $\sim_p$ on preproofs such that $\pi_1 \sim_p \pi_2$ if and only if $F_\Gamma(\pi_1) = F_\Gamma(\pi_2)$, and to give a logical justification of these relations independent of the translation to lambda terms. This establishes $\cat{S}_\Gamma$ as a mathematical structure in its own right, so that the comparison to $\cat{L}_\Gamma$ may be meaningfully referred to as a duality. 
\\

Given the Curry-Howard correspondence, the identity of proofs is closely related to the old problem of when two sequent calculus proofs map to the same natural deduction under the translation defined by defined by Gentzen \cite{gentzen} (see Remark \ref{remark:translation_who}). This has been studied by various authors, most notably Zucker \cite{zucker}, Pottinger \cite{pottinger}, Dyckhoff-Pinto \cite{dyckhoffpinto}, Mints \cite{mints} and Kleene \cite{kleene}. The most important results are those obtained by Zucker and Mints, and we restrict ourselves here to comments on their work; see also Section \ref{section:related_work}.

There is substantial overlap between our main results and those of Zucker and Mints, which we became aware of after this paper had been completed. In \cite{zucker} Zucker gives a set of generating relations characterising when two sequent calculus proofs map to the same natural deduction, for a calculus that does not contain weakening and exchange. The main content of Theorem \ref{gentzen_mints_zucker} also lies in identifying an explicit set of generating relations on preproofs, for the map from sequent calculus proofs to lambda terms in a system of sequent calculus that is (as far as is possible for a system that must be translated unambiguously to lambda terms) as close as possible to Gentzen's LJ. As far as we know, this paper is the first place that the generating relations have been established for Gentzen's LJ with all structural rules.

Mints, using ideas of Kleene \cite{kleene}, identifies a set of normal forms of sequent calculus proofs and studies them using the map from proofs to lambda terms. Our proof of the main theorem (Theorem \ref{gentzen_mints_zucker}) relies on the identification of normal forms, which differ slightly from those of Mints (see also Theorem \ref{theorem:normal_form_prop}). Again we treat a standard form of LJ, whereas \cite{mints} follows Kleene's system G \cite{kleene} in the form of its $(L \imp)$ rule.

Finally, since we must argue that $\cat{S}_\Gamma$ has an independent existence in order for the duality to be a relationship between equals we are committed to mounting a purely logical defense of all the generating relations of $\sim_p$. This is not a concern shared by Zucker, Mints or Kleene. The most interesting generating relations are those that we call $\lambda$-equivalences (Definition \ref{inefficiency}) which are justified on the grounds that they represent an internal Brouwer-Heyting-Kolmogorov interpretation, see the discussion preceding Definition \ref{inefficiency} and Section \ref{section:internal_bhk}.



\section{Sequent Calculus}\label{section:sequent_calc}


There is an infinite set of atomic formulas and if $p$ and $q$ are formulas then so is $p \imp q$. Let $\Psi_\imp$ denote the set of all formulas. For each formula $p$ let $Y_p$ be an infinite set of variables associated with $p$. For distinct formulas $p,q$ the sets $Y_p, Y_q$ are disjoint. We write $x : p$ for $x \in Y_p$ and say \emph{$x$ has type $p$}. Let $\call{P}^n$ be the set of all length $n$ sequences of variables with $\call{P}^0 := \lbrace \varnothing \rbrace$, and $\call{P} := \cup_{n = 0}^\infty \call{P}^n$. A \emph{sequent} is a pair $(\Gamma,p)$ where $\Gamma \in \call{P}$ and $p \in \Psi_\imp$, written $\Gamma \vdash p$. We call $\Gamma$ the \emph{antecedent} and $p$ the \emph{succedent} of the sequent. Given $\Gamma$ and a variable $x:p$ we write $\Gamma, x:p$ for the element of $\call{P}$ given by appending $x:p$ to the end of $\Gamma$. A variable $x:p$ may occur more than once in a sequent.

Our intuitionistic sequent calculus is the system LJ of \cite[\S III]{gentzen} restricted to implication, with formulas in the antecedent tagged with variables and a more liberal set of deduction rules (see Remark \ref{remark:liberal_LJ}). We follow the convention of \cite[\S 5.1]{girard} in grouping $(\operatorname{ax})$ and $(\operatorname{cut})$ together rather than including the latter in the structural rules.

\begin{definition}\label{defn:deduction_rules}
A \emph{deduction rule} results from one of the schemata below by a substitution of the following kind: replace $p,q,r$ by arbitrary formulas, $x,y$ by arbitrary variables, and $\Gamma, \Delta, \Theta$ by arbitrary (possibly empty) sequences of formulas separated by commas:
\label{sequentcalc} 
\begin{itemize}
    \item the \textbf{identity group}:
    \begin{itemize}
        \item \textbf{Axiom}:
        \begin{prooftree}
        \AxiomC{}
        \RightLabel{$({\operatorname{ax}})$}
        \UnaryInfC{$x:p \vdash p$}
        \end{prooftree}
        \item
        \textbf{Cut}:
        \begin{prooftree}
        \AxiomC{$\Gamma \vdash p$}
        \AxiomC{$\Delta, x:p,\Theta \vdash q$}
        \RightLabel{$({\operatorname{cut}})$}
        \BinaryInfC{$\Gamma, \Delta, \Theta \vdash q$}
        \end{prooftree}
    \end{itemize}
    \item the \textbf{structural rules}:
    \begin{itemize}
        \item \textbf{Contraction}:
        \begin{prooftree}
        \AxiomC{$\Gamma, x:p, y:p, \Delta \vdash q$}
        \RightLabel{$({\operatorname{ctr}})$}
        \UnaryInfC{$\Gamma, x:p, \Delta \vdash q$}
        \end{prooftree}
        \item \textbf{Weakening}:
        \begin{prooftree}
        \AxiomC{$\Gamma, \Delta \vdash q$}
        \RightLabel{$({\operatorname{weak}})$}
        \UnaryInfC{$\Gamma, x:p, \Delta \vdash q$}
        \end{prooftree}
        \item \textbf{Exchange}:
        \begin{prooftree}
        \AxiomC{$\Gamma, x:p,y:q, \Delta \vdash r$}
        \RightLabel{$({\operatorname{ex}})$}
        \UnaryInfC{$\Gamma, y:q,x:p, \Delta \vdash r$}
        \end{prooftree}
    \end{itemize}
    \item the \textbf{logical rules}:
    \begin{itemize}
        \item 
        \textbf{Right introduction}:
        \begin{prooftree}
        \AxiomC{$\Gamma, x:p, \Delta \vdash q$}
        \RightLabel{$(R\imp)$}
        \UnaryInfC{$\Gamma, \Delta \vdash p \imp q $}
        \end{prooftree}
        \item
        \textbf{Left introduction}:
        \begin{prooftree}
        \AxiomC{$\Gamma \vdash p$}
        \AxiomC{$\Delta, x:q, \Theta \vdash r$}
        \RightLabel{$(L \imp)$}
        \BinaryInfC{$y: p \imp q, \Gamma, \Delta, \Theta \vdash r$}
        \end{prooftree}
    \end{itemize}
\end{itemize}
\end{definition}

\begin{definition}
A \emph{preproof} is a finite rooted planar tree where each edge is labelled by a sequent and each node except for the root is labelled by a valid deduction rule. If the edge connected to the root is labelled by the sequent $\Gamma \vdash p$ then we call the preproof a \emph{preproof of $\Gamma \vdash p$}.
\end{definition}

Observe that the only valid label for a leaf node is an axiom rule, so a preproof reads from the leaves to the root as a deduction of $\Gamma \vdash p$ from axiom rules.

\begin{example}\label{example:church_2} Here is the Church numeral $\underline{2}$ in our sequent calculus
\begin{prooftree}
        \AxiomC{}
        \RightLabel{$({\operatorname{ax}})$}
        \UnaryInfC{$x:p \vdash p$}
        \AxiomC{}
        \RightLabel{$({\operatorname{ax}})$}
        \UnaryInfC{$x:p \vdash p$}
        \AxiomC{}
        \RightLabel{$({\operatorname{ax}})$}
        \UnaryInfC{$x:p \vdash p$}
        \RightLabel{$(L \imp)$}
        \BinaryInfC{$y': p \imp p, x:p \vdash p$}
        \RightLabel{$(L \imp)$}
        \BinaryInfC{$y: p \imp p, y': p \imp p, x:p \vdash p$}
        \RightLabel{$(\operatorname{ctr})$}
        \UnaryInfC{$y: p \imp p, x:p \vdash p$}
        \RightLabel{$(R \imp)$}
        \UnaryInfC{$y:p \imp p \vdash p \imp p$}
\end{prooftree}
\end{example}

\begin{remark} Multiple occurrences of a deduction rule are communicated in the notation with doubled horizontal lines. For example if $\Gamma = x_1: p_1, \ldots, x_n: p_n$ then the preproof
\begin{prooftree}
        \AxiomC{}
        \RightLabel{$({\operatorname{ax}})$}
        \UnaryInfC{$y:q \vdash q$}
        \doubleLine
        \RightLabel{$(\operatorname{weak})$}
        \UnaryInfC{$\Gamma, y: q \vdash q$}
\end{prooftree}
weakens in every formula in the sequence. The doubled horizontal line therefore stands for $n$ occurrences of the rule $(\operatorname{weak})$. The preproofs which perform these weakenings in a different order are, of course, not equal as preproofs, so the notation is an abuse. We will only use it below in the context of defining generating pairs of equivalence relations in cases where any reading of this notation leads to the same equivalence relation.
\end{remark}

\begin{remark}\label{remark:liberal_LJ}
A deduction rule is \emph{strict} if it is an arbitrary $(\operatorname{ax})$ or $(\operatorname{ex})$ rule, or it is one of the other rules and the occurrence of $x:p$ in the rule is leftmost in the antecedent. A \emph{strict preproof} is a preproof in which every deduction rule is strict. These are the deduction rules and preproofs of Gentzen's original sequent calculus \cite[\S III]{gentzen}. A general deduction rule is clearly derivable from the strict rules by exchange, and so we may choose to view non-strict deduction rules as derived rules; see Lemma \ref{lemma:strict_vs_liberal}.

We adopt the more liberal rules since they make the commuting conversions, cut-elimination transformations and the proof of cut-elimination easier to present. A similar calculus is adopted, for similar reasons, in \cite{cutctr} and elsewhere.
\end{remark}

\begin{remark}\label{remark:where_left} We follow Gentzen \cite[\S III]{gentzen} in putting the variable introduced by a $(L \imp)$ rule at the first position in the antecedent. This choice is correct from the point of view of the relationship between sequent calculus proofs and lambda terms, as may be seen in Lemma \ref{lemma:preimage_app} and Section \ref{section:normal_form_seq}.
\end{remark}

When should two preproofs be considered to be the \emph{same} proof? Clearly some of the structure of a preproof is logically insignificant, but it is by no means trivial to identify a precise notion of \emph{proof} as separate from \emph{preproof}. Historically, logic has concerned itself primarily with the provability of sequents $\Gamma \vdash p$ rather than the structure of the set of all preproofs, but as proof theory has developed the question of the identity of proofs has acquired increasing importance; see Ungar \cite{ungar} and Prawitz \cite[\S 4.3]{prawitz_phil}.

We say that a relation $\sim$ on the set of preproofs satisfies condition (C0) if $\pi_1 \sim \pi_2$ implies $\pi_1, \pi_2$ are preproofs of the same sequent. The relation satisfies condition (C1) if it satisfies (C0) and $\pi_1 \sim \pi_2$ implies $\pi_1' \sim \pi_2'$ where $\pi_1', \pi_2'$ are the result of applying the same deduction rule to $\pi_1, \pi_2$ respectively. For example if $\sim$ satisfies (C1) and $\pi_1 \sim \pi_2$ then
\begin{center}
\AxiomC{$\pi_1$}
\noLine
\UnaryInfC{$\vdots$}
\noLine
\UnaryInfC{$\Gamma, \Delta \vdash p$}
\RightLabel{$({\operatorname{weak}})$}
\UnaryInfC{$\Gamma,x:p, \Delta \vdash q$}
\DisplayProof
$\sim$
\AxiomC{$\pi_2$}
\noLine
\UnaryInfC{$\vdots$}
\noLine
\UnaryInfC{$\Gamma, \Delta \vdash p$}
\RightLabel{$({\operatorname{weak}})$}
\UnaryInfC{$\Gamma,x:p, \Delta \vdash q$}
\DisplayProof
\end{center}
Condition (C2) is defined using the following schematics:
    \begin{center}
    \begin{tabular}{ >{\centering}m{6cm} >{\centering}m{6cm} >{\centering}m{0.5cm}}
        \AxiomC{$\pi_i$}
        \noLine
        \UnaryInfC{$\vdots$}
        \noLine
        \UnaryInfC{$\Gamma \vdash p$}
        \AxiomC{$\rho_i$}
        \noLine
        \UnaryInfC{$\vdots$}
        \noLine
        \UnaryInfC{$\Delta, x:p,\Theta \vdash q$}
        \RightLabel{$({\operatorname{cut}})$}
        \BinaryInfC{$\Gamma, \Delta, \Theta \vdash q$}
        \DisplayProof
        &
        \AxiomC{$\pi_i$}
        \noLine
        \UnaryInfC{$\vdots$}
        \noLine
        \UnaryInfC{$\Gamma \vdash p$}
        \AxiomC{$\rho_i$}
        \noLine
        \UnaryInfC{$\vdots$}
        \noLine
        \UnaryInfC{$\Delta, x:q, \Theta \vdash r$}
        \RightLabel{$(L \imp)$}
        \BinaryInfC{$y: p \imp q, \Gamma, \Delta, \Theta \vdash r$}
        \DisplayProof
        &
        \tagarray{\label{condition_c2}}
    \end{tabular}
    \end{center}
We say that a relation $\sim$ on the set of preproofs satisfies condition (C2) if it satisfies (C0) and whenever $\pi_1 \sim \pi_2$ and $\rho_1 \sim \rho_2$ then also $\kappa_1 \sim \kappa_2$ where $\kappa_i$ for $i \in \{1,2\}$ is obtained from the pair $(\pi_i, \rho_i)$ by application of one of the deduction rules in \eqref{condition_c2}.

\begin{definition}
A relation $\sim$ on preproofs is \emph{compatible} if it satisfies (C0),(C1),(C2).
\end{definition}

An \emph{occurrence} of $x:p$ in a preproof $\pi$ is an occurrence in the antecedent $\Gamma$ of a sequent labelling some edge of $\pi$. Some occurrences are related by the flow of information in the preproof, and some are not. More precisely:

\begin{definition}[(Ancestors)]
An occurrence of $z_1:s$ in a preproof $\pi$ is an \emph{immediate strong ancestor} (resp. \emph{immediate weak ancestor}) of an occurrence $z_2:s$ if there is a deduction rule in $\pi$ where $z_1:s$ is in the numerator and $z_2:s$ is in the denominator, and one of the following holds (referring to the schemata in Definition \ref{defn:deduction_rules}):
\begin{itemize}
\item[(i)] $z_1:s, z_2:s$ are in the same position of $\Gamma, \Delta, \Theta$ in the numerator and denominator.
\item[(ii)] the rule is $(\operatorname{ctr})$, $z_1:s$ is the first of the two variables being contracted (resp. $z_1:s$ is either of the variables being contracted) and $z_2:s$ is the result of that contraction.
\item[(iii)] the rule is $(\operatorname{ex})$ and either $z_1:s = x:p, z_2:s = x:p$ or $z_1:s = y:p, z_2:s = y:p$.
\end{itemize}
One occurrence $z:s$ is a \emph{strong ancestor} (resp. \emph{weak ancestor}) of another $z':s$ if there is a sequence $z:s = z_1:s, \ldots, z_n:s = z':s$ of occurrences in $\pi$ with $z_i:s$ an immediate strong (resp. weak) ancestor of $z_{i+1}:s$ for $1 \le i < n$.
\end{definition}

Note that if $z_1:s$ is a strong ancestor of $z_2:s$ then $z_1 = z_2$ but this is not necessarily true for weak ancestors.

\begin{definition}\label{defn:approx} Let $\approx_{str}$ (resp. $\approx_{wk}$) denote the equivalence relation on the set of variable occurrences generated by the strong (resp. weak) ancestor relation.
\end{definition}

\begin{definition}[(Ancestor substitution)]\label{defn:ancestor_sub}
Let $x:p$ be an occurrence of a variable in a preproof $\pi$ and $y:p$ another variable. We denote by $\operatorname{subst}^{str}(\pi, x, y)$ the preproof obtained from $\pi$ by replacing the occurrence $x:p$ and all its strong ancestors by $y$.
\end{definition}

\begin{example}\label{example:weak_ancestor_2} In the preproof $\underline{2}$ of Example \ref{example:church_2} the partition of variable occurrences according to the equivalence relation $\approx_{str}$ is shown by colours in
\begin{prooftree}
        \AxiomC{}
        \RightLabel{$({\operatorname{ax}})$}
        \UnaryInfC{$\textcolor{red}{x:p} \vdash p$}
        \AxiomC{}
        \RightLabel{$({\operatorname{ax}})$}
        \UnaryInfC{$\textcolor{magenta}{x:p} \vdash p$}
        \AxiomC{}
        \RightLabel{$({\operatorname{ax}})$}
        \UnaryInfC{$\textcolor{green}{x:p} \vdash p$}
        \RightLabel{$(L \imp)$}
        \BinaryInfC{$\textcolor{blue}{y': p \imp p}, \textcolor{magenta}{x:p} \vdash p$}
        \RightLabel{$(L \imp)$}
        \BinaryInfC{$\textcolor{cyan}{y: p \imp p}, \textcolor{blue}{y': p \imp p}, \textcolor{red}{x:p} \vdash p$}
        \RightLabel{$(\operatorname{ctr})$}
        \UnaryInfC{$\textcolor{cyan}{y: p \imp p}, \textcolor{red}{x:p} \vdash p$}
        \RightLabel{$(R \imp)$}
        \UnaryInfC{$\textcolor{cyan}{y:p \imp p} \vdash p \imp p$}
\end{prooftree}
and the partition according to $\approx_{wk}$ in
\begin{prooftree}
        \AxiomC{}
        \RightLabel{$({\operatorname{ax}})$}
        \UnaryInfC{$\textcolor{red}{x:p} \vdash p$}
        \AxiomC{}
        \RightLabel{$({\operatorname{ax}})$}
        \UnaryInfC{$\textcolor{magenta}{x:p} \vdash p$}
        \AxiomC{}
        \RightLabel{$({\operatorname{ax}})$}
        \UnaryInfC{$\textcolor{green}{x:p} \vdash p$}
        \RightLabel{$(L \imp)$}
        \BinaryInfC{$ \textcolor{cyan}{y': p \imp p}, \textcolor{magenta}{x:p} \vdash p$}
        \RightLabel{$(L \imp)$}
        \BinaryInfC{$\textcolor{cyan}{y: p \imp p}, \textcolor{cyan}{y': p \imp p}, \textcolor{red}{x:p} \vdash p$}
        \RightLabel{$(\operatorname{ctr})$}
        \UnaryInfC{$\textcolor{cyan}{y: p \imp p}, \textcolor{red}{x:p} \vdash p$}
        \RightLabel{$(R \imp)$}
        \UnaryInfC{$\textcolor{cyan}{y:p \imp p} \vdash p \imp p$}
\end{prooftree}
\end{example}

In our preproofs we have tags, in the form of variables, for hypotheses. Since the precise nature of these tags is immaterial, if the variable is eliminated in a $(R \imp)$, $(L \imp)$,$(\operatorname{ctr})$ or $(\operatorname{cut})$ rule the identity of the proof should be independent of the tag. 

\begin{definition}[($\alpha$-equivalence)]
\label{alphaequivalence}
We define $\sim_\alpha$ to be the smallest compatible equivalence relation on preproofs such that
\begin{center}
    \begin{tabular}{ >{\centering}m{6cm} >{\centering}m{0.5cm} >{\centering}m{6cm} >{\centering}m{0.5cm}}   
        \AxiomC{$\Gamma \vdash p$}
        \AxiomC{$\pi$}
    \noLine
    \UnaryInfC{$\vdots$}
    \noLine
        \UnaryInfC{$\Delta, x:p,\Theta \vdash q$}
        \RightLabel{$({\operatorname{cut}})$}
        \BinaryInfC{$\Gamma, \Delta, \Theta \vdash q$}
        \DisplayProof
        &
        $\sim_\alpha$
        &
        \AxiomC{$\Gamma \vdash p$}
        \AxiomC{$\operatorname{subst}^{str}(\pi,x,y)$}
    \noLine
    \UnaryInfC{$\vdots$}
    \noLine
        \UnaryInfC{$\Delta, y:p,\Theta \vdash q$}
        \RightLabel{$({\operatorname{cut}})$}
        \BinaryInfC{$\Gamma, \Delta, \Theta \vdash q$}
        \DisplayProof
        &
        \tagarray{\label{alpha_cut}}
    \end{tabular}
    \end{center}
    
    \begin{center}
    \begin{tabular}{ >{\centering}m{6cm} >{\centering}m{0.5cm} >{\centering}m{6cm} >{\centering}m{0.5cm}}   
    \AxiomC{$\pi$}
    \noLine
    \UnaryInfC{$\vdots$}
    \noLine 
        \UnaryInfC{$\Gamma, x:p, y:p, \Delta \vdash q$}
        \RightLabel{$({\operatorname{ctr}})$}
        \UnaryInfC{$\Gamma, x:p, \Delta \vdash q$}
        \DisplayProof
        &
        $\sim_\alpha$
        &
        \AxiomC{$\operatorname{subst}^{str}(\pi,y,z)$}
    \noLine
    \UnaryInfC{$\vdots$}
    \noLine 
        \UnaryInfC{$\Gamma, x:p, z:p, \Delta \vdash q$}
        \RightLabel{$({\operatorname{ctr}})$}
        \UnaryInfC{$\Gamma, x:p, \Delta \vdash q$}
        \DisplayProof
        &
        \tagarray{\label{alpha_ctr}}
    \end{tabular}
    \end{center}
    
    \begin{center}
    \begin{tabular}{ >{\centering}m{6cm} >{\centering}m{0.5cm} >{\centering}m{6cm} >{\centering}m{0.5cm}}
    \AxiomC{$\pi$}
    \noLine
    \UnaryInfC{$\vdots$}
    \noLine
    \UnaryInfC{$x:p, \Gamma\vdash q$}
    \RightLabel{$(R\imp)$}
    \UnaryInfC{$\Gamma \vdash p \imp q$}
    \DisplayProof
    &
    $\sim_\alpha$
    &
    \AxiomC{$\operatorname{subst}^{str}(\pi, x, y)$}
    \noLine
    \UnaryInfC{$\vdots$}
    \noLine
    \UnaryInfC{$y:p, \Gamma \vdash q$}
    \RightLabel{$(R\imp)$}
    \UnaryInfC{$\Gamma \vdash p \imp q$}
    \DisplayProof
            &
        \tagarray{\label{alpha_R}}
    \end{tabular}
    \end{center}
    
    \begin{center}
    \begin{tabular}{ >{\centering}m{6cm} >{\centering}m{0.5cm} >{\centering}m{6cm} >{\centering}m{0.5cm}}
    
    \AxiomC{$\Gamma \vdash p$}
    \AxiomC{$\pi$}
    \noLine
    \UnaryInfC{$\vdots$}
    \noLine
    \UnaryInfC{$\Delta, x:q, \Theta \vdash r$}
    \RightLabel{$(L \imp)$}
    \BinaryInfC{$z: p \imp q, \Gamma, \Delta, \Theta \vdash r$}
    \DisplayProof
    &
    $\sim_\alpha$
    &
\AxiomC{$\Gamma \vdash p$}
    \AxiomC{$\operatorname{subst}^{str}(\pi,x,y)$}
    \noLine
    \UnaryInfC{$\vdots$}
    \noLine
    \UnaryInfC{$\Delta, y:q, \Theta \vdash r$}
    \RightLabel{$(L \imp)$}
    \BinaryInfC{$z: p \imp q, \Gamma, \Delta, \Theta \vdash r$}
    \DisplayProof
            &
        \tagarray{\label{alpha_L}}
    \end{tabular}
    \end{center}
for any proof $\pi$ and variables $x,y,z$ of the same type.
\end{definition}

\begin{remark} The generating relation \eqref{alpha_L} of Definiton \ref{alphaequivalence} is to be read as a pair of preproofs $(\psi,\psi') \in \; \sim_\alpha$ where both preproofs have final sequent $z: p \imp q, \Gamma, \Delta, \Theta \vdash r$ and the branch ending in $\Gamma \vdash p$ is any preproof (but it is the same preproof in both $\psi$ and $\psi'$). To avoid clutter we will not label branches, here or elsewhere, if it is clear how to match up the branches in the two preproofs involved in the relation.
\end{remark}

The price for our more liberal deduction rules is the inclusion of $\tau$-equivalences below, which express that two instances of the same deduction rule, operating in different places, are essentially the same.

\begin{definition}[($\tau$-equivalence)]
\label{tauequivalence} We define $\sim_{\tau}$ to be the smallest compatible equivalence relation on preproofs satisfying
    \begin{center}
    \begin{tabular}{ >{\centering}m{6cm} >{\centering}m{0.5cm} >{\centering}m{6cm} >{\centering}m{0.5cm}}
        \AxiomC{$\Gamma \vdash p$}
        \AxiomC{$\Delta, x:p,y:q,\Theta \vdash q$}
        \RightLabel{$({\operatorname{cut}})$}
        \BinaryInfC{$\Gamma, \Delta, y:q, \Theta \vdash q$}
        \DisplayProof
        &
        $\sim_\tau$
        &
        \AxiomC{$\Gamma \vdash p$}
        \AxiomC{$\Delta, x:p,y:q,\Theta \vdash q$}
        \RightLabel{$(\operatorname{ex})$}
        \UnaryInfC{$\Delta, y:q, x:p, \Theta \vdash q$}
        \RightLabel{$({\operatorname{cut}})$}
        \BinaryInfC{$\Gamma, \Delta, y:q, \Theta \vdash q$}
        \DisplayProof
        &
        \tagarray{\label{tau_cut}}
    \end{tabular}
    \end{center}
    
    \begin{center}
    \begin{tabular}{ >{\centering}m{6cm} >{\centering}m{0.5cm} >{\centering}m{6cm} >{\centering}m{0.5cm}}
        \AxiomC{$\Gamma, x:p, x':p, y:q, \Gamma' \vdash r$}
        \RightLabel{$(\operatorname{ctr})$}
        \UnaryInfC{$\Gamma, x:p, y:q, \Gamma' \vdash r$}
        \RightLabel{$(\operatorname{ex})$}
        \UnaryInfC{$\Gamma, y:q, x:p, \Gamma' \vdash r$}
        \DisplayProof
        &
        $\sim_\tau$
        &
        \AxiomC{$\Gamma, x:p, x':p, y: q, \Gamma' \vdash r$}
        \doubleLine
        \RightLabel{$(\operatorname{ex})$}
        \UnaryInfC{$\Gamma, y:q, x:p, x':p, \Gamma' \vdash r$}
        \RightLabel{$\operatorname{(ctr)}$}
        \UnaryInfC{$\Gamma, y:q, x:p, \Gamma' \vdash r$}
        \DisplayProof
        &
        \tagarray{\label{tau_ctr_ex}}
    \end{tabular}
    \end{center}

\begin{center}
\begin{tabular}{ >{\centering}m{6cm} >{\centering}m{0.5cm} >{\centering}m{6cm} >{\centering}m{0.5cm}} 
         \AxiomC{$\Gamma, x:p, \Gamma' \vdash q$}
         \RightLabel{$(\operatorname{weak})$}
         \UnaryInfC{$\Gamma, x:p, y:r, \Gamma' \vdash q$}
         \RightLabel{$(\operatorname{ex})$}
         \UnaryInfC{$\Gamma, y:r, x:p, \Gamma' \vdash q$}
         \DisplayProof
         & $\sim_\tau$ &
         \AxiomC{$\Gamma, x:p, \Gamma' \vdash q$}
         \RightLabel{$(\operatorname{weak})$}
         \UnaryInfC{$\Gamma, y:r, x:p, \Gamma' \vdash q$}
         \DisplayProof
         &
         \tagarray{\label{tau_weak_ex}}
\end{tabular}
\end{center}

    \begin{center}
    \begin{tabular}{>{\centering}m{6cm} >{\centering}m{0.5cm} >{\centering}m{6cm} >{\centering}m{0.5cm}}    
            \AxiomC{$\Gamma, x:p, z:r,\Gamma' \vdash s$}
            \RightLabel{$(R \imp)$}
            \UnaryInfC{$\Gamma, x:p, \Gamma' \vdash r \imp s$}
            \DisplayProof
            &
            $\sim_\tau$
            &
            \AxiomC{$\Gamma, x:p, z:r, \Gamma' \vdash s$}
            \RightLabel{$(\operatorname{ex})$}
            \UnaryInfC{$\Gamma, z:r, x:p, \Gamma' \vdash s$}
            \RightLabel{$(R \imp)$}
            \UnaryInfC{$\Gamma, x:p, \Gamma' \vdash r \imp s$}
            \DisplayProof
            &
            \tagarray{\label{tau_R_ex}}
    \end{tabular}
    \end{center}

    \begin{center}
    \begin{tabular}{>{\centering}m{10cm} >{\centering}m{0.5cm}}    
            \AxiomC{$\Gamma \vdash p$}
        \AxiomC{$\Delta, x:q, z:r, \Delta' \vdash r$}
        \RightLabel{$(L \imp)$}
        \BinaryInfC{$y: p\imp q, \Gamma,\Delta, z:r, \Delta' \vdash r$}
            \DisplayProof\\\vspace{0.5cm}
            $\sim_\tau$\\\vspace{0.5cm}
            \AxiomC{$\Gamma \vdash p$}
        \AxiomC{$\Delta, x:q, z:r, \Delta' \vdash r$}
        \RightLabel{$(\operatorname{ex})$}
        \UnaryInfC{$\Delta, z:r, x:q, \Delta' \vdash r$}
        \RightLabel{$(L \imp)$}
        \BinaryInfC{$y: p\imp q, \Gamma,\Delta, z:r, \Delta' \vdash r$}
            \DisplayProof
            &
            \tagarray{\label{tau_L_ex2}}
    \end{tabular}
    \end{center}
    
\begin{center}
\begin{tabular}{ >{\centering}m{10cm} >{\centering}m{0.5cm}} 
\AxiomC{$x_1: p_1, \ldots, x_n : p_n \vdash q$}
        \doubleLine
        \RightLabel{$(\operatorname{ex}, \sigma_1, \ldots, \sigma_r)$}
        \UnaryInfC{$x_{\tau 1}: p_{\tau 1}, \ldots, x_{\tau n}: p_{\tau n} \vdash q$}
        \DisplayProof\\\vspace{0.5cm}
        $\sim_\tau$\\\vspace{0.5cm}
        \AxiomC{$x_1: p_1, \ldots, x_n : p_n \vdash q$}
        \doubleLine
        \RightLabel{$(\operatorname{ex}, \rho_1, \ldots, \rho_s)$}
        \UnaryInfC{$x_{\tau 1}: p_{\tau 1}, \ldots, x_{\tau n}: p_{\tau n} \vdash q$}
        \DisplayProof
        &
        \tagarray{\label{tau_ex_ex}}
\end{tabular}
\end{center}
where $\tau$ is a permutation and $\sigma_1, \ldots, \sigma_r$ and $\rho_1,\ldots,\rho_s$ are sequences of transpositions of consecutive positions (one or both lists may be empty) with the property that $\sigma_1 \cdots \sigma_r = \tau = \rho_1 \cdots \rho_s$ in the permutation group. The two preproofs in \eqref{tau_ex_ex} are respectively the sequences of exchanges corresponding to the $\sigma_i$ and $\rho_j$.
\end{definition}

\begin{remark} Note that we only include the $\tau$-equivalences with ``right to left'' exchanges, since the other possible relation follows from these and \eqref{tau_ex_ex}, for example:
    \begin{center}
    \begin{tabular}{ >{\centering}m{6cm} >{\centering}m{0.5cm} >{\centering}m{6cm}}
        \AxiomC{$\Gamma, y:q, x:p, x':p, \Gamma' \vdash r$}
        \RightLabel{$(\operatorname{ctr})$}
        \UnaryInfC{$\Gamma, y:q, x:p, \Gamma' \vdash r$}
        \RightLabel{$(\operatorname{ex})$}
        \UnaryInfC{$\Gamma, x:p, y:q, \Gamma' \vdash r$}
        \DisplayProof
        &
        $\sim_\tau$
        &
        \AxiomC{$\Gamma, y:q, x:p, x':p, \Gamma' \vdash r$}
        \RightLabel{$(\operatorname{ex})$}
        \doubleLine
        \UnaryInfC{$\Gamma, x:p, x':p, y:q, \Gamma' \vdash r$}
        \RightLabel{$(\operatorname{ex})$}
        \doubleLine
        \UnaryInfC{$\Gamma, y:q, x:p, x':p, \Gamma' \vdash r$}
        \RightLabel{$(\operatorname{ctr})$}
        \UnaryInfC{$\Gamma, y:q, x:p, \Gamma' \vdash r$}
        \RightLabel{$(\operatorname{ex})$}
        \UnaryInfC{$\Gamma, x:p, y:q, \Gamma' \vdash r$}
        \DisplayProof
    \end{tabular}
    \end{center}
    \begin{center}
    \begin{tabular}{ >{\centering}m{6cm} >{\centering}m{0.5cm} >{\centering}m{6cm}}
        &
        $\sim_\tau$
        &
        \AxiomC{$\Gamma, y:q, x:p, x':p, \Gamma' \vdash r$}
        \RightLabel{$(\operatorname{ex})$}
        \doubleLine
        \UnaryInfC{$\Gamma, x:p, x':p, y:q, \Gamma' \vdash r$}
        \RightLabel{$(\operatorname{ctr})$}
        \UnaryInfC{$\Gamma, x:p, y:q, \Gamma' \vdash r$}
        \RightLabel{$(\operatorname{ex})$}
        \UnaryInfC{$\Gamma, y:q, x:p, \Gamma' \vdash r$}
        \RightLabel{$(\operatorname{ex})$}
        \UnaryInfC{$\Gamma, x:p, y:q, \Gamma' \vdash r$}
        \DisplayProof
    \end{tabular}
    \end{center}
    \begin{center}
    \begin{tabular}{ >{\centering}m{6cm} >{\centering}m{0.5cm} >{\centering}m{6cm}}
        &
        $\sim_\tau$
        &
        \AxiomC{$\Gamma, y:q, x:p, x':p, \Gamma' \vdash r$}
        \RightLabel{$(\operatorname{ex})$}
        \doubleLine
        \UnaryInfC{$\Gamma, x:p, x':p, y:q, \Gamma' \vdash r$}
        \RightLabel{$(\operatorname{ctr})$}
        \UnaryInfC{$\Gamma, x:q, y:q, \Gamma' \vdash r$}
        \DisplayProof
    \end{tabular}
    \end{center}
\end{remark}

\begin{lemma}\label{lemma:strict_vs_liberal} Every preproof is equivalent under $\sim_\tau$ to a strict preproof.
\end{lemma}
\begin{proof}
Left to the reader.
\end{proof}

We have a strong intuition about the structure of logical arguments which leads to the expectation that the antecedent of a sequent is an extended ``space'' disjoint subsets of which may be the locus of independent operations. This independence is formalised by commuting conversions, which identify preproofs that differ only by ``insignificant'' rearranging of deduction rules. 

\begin{definition}[(Commuting conversions)]
\label{commutingequivalence}
We define $\sim_c$ to be the smallest compatible equivalence relation on preproofs generated by the following pairs. We begin with pairs involving two structural rules:
    \begin{center}
    \begin{tabular}{ >{\centering}m{7.1cm} >{\centering}m{0.5cm} >{\centering}m{7.1cm} >{\centering}m{0.5cm}}
        \AxiomC{$\Gamma,\Gamma',\Gamma'' \vdash q$}
        \RightLabel{$(\operatorname{weak})$}
        \UnaryInfC{$\Gamma, x:p, \Gamma',\Gamma'' \vdash q$}
        \RightLabel{$(\operatorname{weak})$}
        \UnaryInfC{$\Gamma, x:p, \Gamma', y:r, \Gamma'' \vdash q$}
        \DisplayProof
        &
        $\sim_c$
        &
        \AxiomC{$\Gamma,\Gamma',\Gamma'' \vdash q$}
        \RightLabel{$(\operatorname{weak})$}
        \UnaryInfC{$\Gamma,\Gamma', y: r,\Gamma''\vdash q$}
        \RightLabel{$(\operatorname{weak})$}
        \UnaryInfC{$\Gamma, x:p, \Gamma', y:r, \Gamma'' \vdash q$}
        \DisplayProof
        &
        \tagarray{\label{comm_weak_weak}}
    \end{tabular}
    \end{center}
    
    \begin{center}
    \begin{tabular}{ >{\centering}m{7.1cm} >{\centering}m{0.5cm} >{\centering}m{7.1cm} >{\centering}m{0.5cm}}
    \AxiomC{$\Gamma, x: p, x': p, \Gamma', \Gamma'' \vdash q$}
        \RightLabel{$({\operatorname{ctr}})$}
        \UnaryInfC{$\Gamma, x:p, \Gamma', \Gamma'' \vdash q$}
        \RightLabel{$({\operatorname{weak}})$}
        \UnaryInfC{$\Gamma, x:p, \Gamma', y:r, \Gamma'' \vdash q$}
        \DisplayProof
        &
        $\sim_c$
        &
        \AxiomC{$\Gamma, x:p, x':p, \Gamma', \Gamma'' \vdash q$}
        \RightLabel{$({\operatorname{weak}})$}
        \UnaryInfC{$\Gamma, x:p, x':p, \Gamma', y:r, \Gamma'' \vdash q$}
        \RightLabel{$({\operatorname{ctr}})$}
        \UnaryInfC{$\Gamma, x:p, \Gamma', y:r, \Gamma'' \vdash q$}
        \DisplayProof
        &
        \tagarray{\label{comm_ctr_weak}}
    \end{tabular}
    \end{center}
    
    \begin{center}
    \begin{tabular}{ >{\centering}m{7.1cm} >{\centering}m{0.5cm} >{\centering}m{7.1cm} >{\centering}m{0.5cm}}
    \AxiomC{$\Gamma,\Gamma', x: p, x': p, \Gamma'' \vdash q$}
        \RightLabel{$({\operatorname{ctr}})$}
        \UnaryInfC{$\Gamma, \Gamma', x:p, \Gamma'' \vdash q$}
        \RightLabel{$({\operatorname{weak}})$}
        \UnaryInfC{$\Gamma, y:r, \Gamma', x:p, \Gamma'' \vdash q$}
        \DisplayProof
        &
        $\sim_c$
        &
        \AxiomC{$\Gamma,\Gamma', x:p, x':p, \Gamma'' \vdash q$}
        \RightLabel{$({\operatorname{weak}})$}
        \UnaryInfC{$\Gamma, y:r, \Gamma',x:p, x':p, \Gamma'' \vdash q$}
        \RightLabel{$({\operatorname{ctr}})$}
        \UnaryInfC{$\Gamma, y:r, \Gamma', x:p, \Gamma'' \vdash q$}
        \DisplayProof
        &
        \tagarray{\label{comm_ctr_weak2}}
    \end{tabular}
    \end{center}
    
    \begin{center}
    \begin{tabular}{  >{\centering}m{7.1cm} >{\centering}m{0.5cm} >{\centering}m{7.1cm} >{\centering}m{0.5cm}}
    \AxiomC{$\Gamma, x:p, y:q, \Gamma', \Gamma'' \vdash r$}
        \RightLabel{$({\operatorname{ex}})$}
        \UnaryInfC{$\Gamma, y:q, x:p, \Gamma', \Gamma'' \vdash r$}
        \RightLabel{$({\operatorname{weak}})$}
        \UnaryInfC{$\Gamma, y:q, x:p, \Gamma', z:s, \Gamma'' \vdash r$}
        \DisplayProof
        &
        $\sim_c$
        &
        \AxiomC{$\Gamma, x:p, y:q, \Gamma', \Gamma'' \vdash r$}
        \RightLabel{$({\operatorname{weak}})$}
        \UnaryInfC{$\Gamma, x:p, y:q, \Gamma', z:s, \Gamma'' \vdash r$}
        \RightLabel{$({\operatorname{ex}})$}
        \UnaryInfC{$\Gamma, y:q, x:p, \Gamma', z:s, \Gamma'' \vdash r$}
        \DisplayProof
        &
        \tagarray{\label{comm_ex_weak}}
    \end{tabular}
    \end{center}
    
    \begin{center}
    \begin{tabular}{  >{\centering}m{7.1cm} >{\centering}m{0.5cm} >{\centering}m{7.1cm} >{\centering}m{0.5cm}}
    \AxiomC{$\Gamma, \Gamma', x:p, y:q, \Gamma'' \vdash r$}
        \RightLabel{$({\operatorname{ex}})$}
        \UnaryInfC{$\Gamma, \Gamma', y:q, x:p, \Gamma'' \vdash r$}
        \RightLabel{$({\operatorname{weak}})$}
        \UnaryInfC{$\Gamma, z:s, \Gamma', y:q, x:p, \Gamma'' \vdash r$}
        \DisplayProof
        &
        $\sim_c$
        &
        \AxiomC{$\Gamma, \Gamma', x:p, y:q, \Gamma'' \vdash r$}
        \RightLabel{$({\operatorname{weak}})$}
        \UnaryInfC{$\Gamma, z:s, \Gamma', x:p, y:q, \Gamma'' \vdash r$}
        \RightLabel{$({\operatorname{ex}})$}
        \UnaryInfC{$\Gamma, z:s, \Gamma', y:q, x:p, \Gamma'' \vdash r$}
        \DisplayProof
        &
        \tagarray{\label{comm_ex_weak2}}
    \end{tabular}
    \end{center}
    
    \begin{center}
    \begin{tabular}{ >{\centering}m{7.1cm} >{\centering}m{0.5cm} >{\centering}m{7.1cm} >{\centering}m{0.5cm}}
        \AxiomC{$\Gamma, x:p, x': p, \Gamma', y:q, y': q, \Gamma'' \vdash r$}
        \RightLabel{$(\operatorname{ctr})$}
        \UnaryInfC{$\Gamma, x:p, \Gamma', y:q, y': q, \Gamma'' \vdash r$}
        \RightLabel{$(\operatorname{ctr})$}
        \UnaryInfC{$\Gamma, x:p, \Gamma', y:q, \Gamma'' \vdash r$}
        \DisplayProof
        &
        $\sim_c$
        &
        \AxiomC{$\Gamma, x:p, x':p, \Gamma', y:p, y':p, \Gamma'' \vdash r$}
        \RightLabel{$(\operatorname{ctr})$}
        \UnaryInfC{$\Gamma, x:p, x':p, \Gamma', y:p, \Gamma'' \vdash r$}
        \RightLabel{$(\operatorname{ctr})$}
        \UnaryInfC{$\Gamma, x:p, \Gamma', y:p, \Gamma'' \vdash r$}
        \DisplayProof
        &
        \tagarray{\label{comm_ctr_ctr2}}
    \end{tabular}
    \end{center}
 
    \begin{center}
    \begin{tabular}{ >{\centering}m{7.1cm} >{\centering}m{0.5cm} >{\centering}m{7.1cm} >{\centering}m{0.5cm}}
         \AxiomC{$\Gamma, x:p, y:q, \Gamma', z: r, z': r \vdash s$}
        \RightLabel{$({\operatorname{ex}})$}
        \UnaryInfC{$\Gamma, y:q, x:p, \Gamma', z:r, z':r \vdash s$}
        \RightLabel{$({\operatorname{ctr}})$}
        \UnaryInfC{$\Gamma, y:q, x:p, \Gamma', z:r \vdash s$}
        \DisplayProof
        &
        $\sim_c$
        &
        \AxiomC{$\Gamma, x:p, y:q, \Gamma', z:r, z':r \vdash s$}
        \RightLabel{$({\operatorname{ctr}})$}
        \UnaryInfC{$\Gamma, x:p, y:q, \Gamma', z:r \vdash s$}
        \RightLabel{$({\operatorname{ex}})$}
        \UnaryInfC{$\Gamma, y:q, x:p, \Gamma', z:r \vdash s$}
        \DisplayProof   
        &
        \tagarray{\label{comm_ex_ctr}}
    \end{tabular}
    \end{center}
    
    \begin{center}
    \begin{tabular}{ >{\centering}m{7.1cm} >{\centering}m{0.5cm} >{\centering}m{7.1cm} >{\centering}m{0.5cm}}
         \AxiomC{$\Gamma, z: r, z': r, \Gamma', x:p, y:q, \Gamma'' \vdash s$}
        \RightLabel{$({\operatorname{ex}})$}
        \UnaryInfC{$\Gamma, z: r, z': r, \Gamma', y:q, x:p, \Gamma'' \vdash s$}
        \RightLabel{$({\operatorname{ctr}})$}
        \UnaryInfC{$\Gamma, z:r, \Gamma', y:q, x:p, \Gamma'' \vdash s$}
        \DisplayProof
        &
        $\sim_c$
        &
        \AxiomC{$\Gamma, z: r, z': r, \Gamma', x:p, y:q, \Gamma'' \vdash s$}
        \RightLabel{$({\operatorname{ctr}})$}
        \UnaryInfC{$\Gamma, z:r, \Gamma', x:p, y:q, \Gamma'' \vdash s$}
        \RightLabel{$({\operatorname{ex}})$}
        \UnaryInfC{$\Gamma, z:r, \Gamma', y:q, x:p, \Gamma'' \vdash s$}
        \DisplayProof   
        &
        \tagarray{\label{comm_ex_ctr2}}
    \end{tabular}
    \end{center}

    
Next are the pairs involving $(R \imp)$:
    \begin{center}
    \begin{tabular}{>{\centering}m{7.1cm} >{\centering}m{0.5cm} >{\centering}m{7.1cm} >{\centering}m{0.5cm}}    
            \AxiomC{$\Gamma, x:p, y:q, \Gamma', z:r,\Gamma'' \vdash s$}
            \RightLabel{$(R \imp)$}
            \UnaryInfC{$\Gamma, x:p, y:q, \Gamma', \Gamma'' \vdash r \imp s$}
            \RightLabel{$(\operatorname{ex})$}
            \UnaryInfC{$\Gamma, y:q, x:p, \Gamma',\Gamma'' \vdash r \imp s $}
            \DisplayProof
            &
            $\sim_c$
            &
            \AxiomC{$\Gamma, x:p, y:q, \Gamma', z:r, \Gamma'' \vdash s$}
            \RightLabel{$(\operatorname{ex})$}
            \UnaryInfC{$\Gamma, y:q, x:p, \Gamma', z:r, \Gamma'' \vdash s$}
            \RightLabel{$(R \imp)$}
            \UnaryInfC{$\Gamma, y:q, x:p, \Gamma', \Gamma'' \vdash r \imp s$}
            \DisplayProof
            &
            \tagarray{\label{comm_R_ex}}
    \end{tabular}
    \end{center}
    
    \begin{center}
    \begin{tabular}{ >{\centering}m{7.1cm} >{\centering}m{0.5cm} >{\centering}m{7.1cm} >{\centering}m{0.5cm}}    
    \AxiomC{$\Gamma, x:p, \Gamma'\vdash q$}
            \RightLabel{$(R \imp)$}
            \UnaryInfC{$\Gamma, \Gamma' \vdash p \imp q$}
            \RightLabel{$(\operatorname{weak})$}
            \UnaryInfC{$\Gamma, y:r, \Gamma' \vdash p \imp q$}
            \DisplayProof
            &
            $\sim_c$
            &
            \AxiomC{$\Gamma, x:p,\Gamma'\vdash q$}
            \RightLabel{$(\operatorname{weak})$}
            \UnaryInfC{$\Gamma, y:r, x:p,\Gamma' \vdash q$}
            \RightLabel{$(R \imp)$}
            \UnaryInfC{$\Gamma, y:r, \Gamma' \vdash p \imp q$}
            \DisplayProof
            &
            \tagarray{\label{comm_R_weak}}
    \end{tabular}
    \end{center}

    \begin{center}
    \begin{tabular}{ >{\centering}m{7.1cm} >{\centering}m{0.5cm} >{\centering}m{7.1cm} >{\centering}m{0.5cm}}    
            \AxiomC{$\Gamma, x:p, x':p, \Gamma', y:q, \Gamma'' \vdash r$}
            \RightLabel{$(R \imp)$}
            \UnaryInfC{$\Gamma, x:p, x':p, \Gamma', \Gamma'' \vdash q \imp r$}
            \RightLabel{$(\operatorname{ctr})$}
            \UnaryInfC{$\Gamma, x:p, \Gamma', \Gamma'' \vdash q \imp r$}
            \DisplayProof
            &
            $\sim_c$
            &
            \AxiomC{$\Gamma, x:p,x':p,\Gamma',y:q,\Gamma'' \vdash r$}
            \RightLabel{$(\operatorname{ctr})$}
            \UnaryInfC{$\Gamma, x:p, \Gamma', y:q,\Gamma'' \vdash r$}
            \RightLabel{$(R \imp)$}
            \UnaryInfC{$\Gamma, x:p, \Gamma',\Gamma'' \vdash q \imp r $}
            \DisplayProof    
            &
            \tagarray{\label{comm_R_ctr}}
    \end{tabular}
    \end{center}

    \begin{center}
    \begin{tabular}{ >{\centering}m{7.1cm} >{\centering}m{0.5cm} >{\centering}m{7.1cm} >{\centering}m{0.5cm}}    
            \AxiomC{$\Gamma \vdash r$}
            \AxiomC{$\Delta, z:s, x:p, \Delta' \vdash q$}
            \RightLabel{$(R \imp)$}
            \UnaryInfC{$\Delta, z:s, \Delta' \vdash p \imp q$}
            \RightLabel{$(L \imp)$}
            \BinaryInfC{$y: r \imp s, \Gamma, \Delta, \Delta' \vdash p \imp q$}
            \DisplayProof
            &
            $\sim_c$
            &
            \AxiomC{$\Gamma \vdash r$}
            \AxiomC{$\Delta, z:s, x:p, \Delta' \vdash q$}
            \RightLabel{$(L \imp)$}
            \BinaryInfC{$y: r \imp s,\Gamma, \Delta, x:p, \Delta' \vdash q$}
            \RightLabel{$(R \imp)$}
            \UnaryInfC{$y: r\imp s, \Gamma, \Delta, \Delta' \vdash p \imp q$}
            \DisplayProof
            &
            \tagarray{\label{comm_R_L}}            
    \end{tabular}
    \end{center}
Those pairs involving $(L \imp)$ (one of which was considered already above):
            \begin{center}
            \begin{tabular}{ >{\centering}m{10cm} >{\centering}m{0.5cm}}    
                \AxiomC{$\Gamma, \Gamma' \vdash p$}
                \RightLabel{$(\operatorname{weak})$}
                \UnaryInfC{$\Gamma, z:r, \Gamma' \vdash p$}
                \AxiomC{$\Delta, x:q, \Delta' \vdash r$}
                \RightLabel{$(L \imp)$}
                \BinaryInfC{$y:p \imp q, \Gamma, z:r, \Gamma', \Delta, \Delta' \vdash r$}
                \DisplayProof\\\vspace{0.5cm}
                $\sim_c$\\\vspace{0.5cm}
                \AxiomC{$\Gamma, \Gamma' \vdash p$}
                \AxiomC{$\Delta, x:q, \Delta' \vdash r$}
                \RightLabel{$(L \imp)$}
                \BinaryInfC{$y: p \imp q, \Gamma, \Gamma', \Delta, \Delta' \vdash r$}
                \RightLabel{$(\operatorname{weak})$}
                \UnaryInfC{$y:p \imp q, \Gamma, z:r, \Gamma', \Delta, \Delta' \vdash r$}
                \DisplayProof\\\vspace{0.5cm}
                $\sim_c$\\\vspace{0.5cm}
                \AxiomC{$\Gamma, \Gamma' \vdash p$}
                \AxiomC{$\Delta, x:q, \Delta' \vdash r$}
                \RightLabel{$(\operatorname{weak})$}
                \UnaryInfC{$z:r, \Delta, x:q, \Delta' \vdash r$}
                \RightLabel{$(L \imp)$}
                \BinaryInfC{$y:p \imp q, \Gamma, \Gamma', z:r, \Delta, \Delta' \vdash r$}
                \doubleLine
                \RightLabel{$(\operatorname{ex})$}
                \UnaryInfC{$y: p \imp q, \Gamma, z:r, \Gamma', \Delta, \Delta' \vdash r$}
                \DisplayProof
                & 
                \tagarray{\label{comm_weak_L}}            
            \end{tabular}
            \end{center}
            
            \begin{center}
            \begin{tabular}{ >{\centering}m{10cm} >{\centering}m{0.5cm}}    
                 \AxiomC{$\Gamma, z:r, z':r, \Gamma' \vdash p$}
                \AxiomC{$\Delta, x:q, \Delta' \vdash r$}
                \RightLabel{$(L \imp)$}
                \BinaryInfC{$y:p \imp q, \Gamma, z:r,z':r, \Gamma', \Delta, \Delta' \vdash r$}
                \RightLabel{$(\operatorname{ctr})$}
                \UnaryInfC{$y:p \imp q, \Gamma, z:r, \Gamma', \Delta, \Delta' \vdash r$}
                \DisplayProof\\\vspace{0.5cm}
                $\sim_c$\\\vspace{0.5cm}
                \AxiomC{$\Gamma, z:r, z':r, \Gamma' \vdash p$}
                \RightLabel{$(\operatorname{ctr})$}
                \UnaryInfC{$\Gamma, z:r, \Gamma' \vdash p$}
                \AxiomC{$\Delta, x:q, \Delta' \vdash r$}
                \RightLabel{$(L \imp)$}
                \BinaryInfC{$y:p \imp q,\Gamma, z:r, \Gamma', \Delta, \Delta' \vdash r$}
                \DisplayProof
                &
                \tagarray{\label{comm_L_ctr}}
            \end{tabular}
            \end{center}

            \begin{center}
            \begin{tabular}{ >{\centering}m{10cm} >{\centering}m{0.5cm}}
                \AxiomC{$\Gamma \vdash p$}
                \AxiomC{$\Delta, x:q, z:r, z':r, \Delta' \vdash r$}
                \RightLabel{$(L \imp)$}
                \BinaryInfC{$y:p \imp q,\Gamma, \Delta, z:r,z':r, \Delta' \vdash r$}
                \RightLabel{$(\operatorname{ctr})$}
                \UnaryInfC{$y:p \imp q, \Gamma, \Delta, z:r,\Delta' \vdash r$}
                \DisplayProof
                \\\vspace{0.5cm}
                $\sim_c$\\\vspace{0.5cm}
                \AxiomC{$\Gamma \vdash p$}
                \RightLabel{$(L \imp)$}
                \AxiomC{$\Delta, x:q, z:r, z':r, \Delta' \vdash r$}
                \RightLabel{$(\operatorname{ctr})$}
                \UnaryInfC{$\Delta, x:q, z:r, \Delta' \vdash r$}
                \RightLabel{$(L \imp)$}
                \BinaryInfC{$y:p \imp q, \Gamma, \Delta,z:r,\Delta' \vdash r$}
                \DisplayProof
                &
                \tagarray{\label{comm_L_ctr2}}
            \end{tabular}
            \end{center}

            \begin{center}
            \begin{tabular}{ >{\centering}m{10cm} >{\centering}m{0.5cm}}
                \AxiomC{$\Gamma, z:r,z':s, \Gamma' \vdash p$}
                \AxiomC{$\Delta, x:q,\Delta' \vdash r$}
                \RightLabel{$(L \imp)$}
                \BinaryInfC{$y: p \imp q,\Gamma, z:r, z':s, \Gamma', \Delta, \Delta' \vdash r$}
                \RightLabel{$(\operatorname{ex})$}
                \UnaryInfC{$y:p\imp q,\Gamma, z':s,z:r, \Gamma', \Delta,  \Delta' \vdash r$}
                \DisplayProof\\\vspace{0.5cm}
                $\sim_c$\\\vspace{0.5cm}
                \AxiomC{$\Gamma, z:r, z':s, \Gamma' \vdash p$}
                \RightLabel{$(\operatorname{ex})$}
                \UnaryInfC{$\Gamma, z':s, z:r, \Gamma' \vdash p$}
                \AxiomC{$\Delta, x:q, \Delta' \vdash r$}
                \RightLabel{$(L \imp)$}
                \BinaryInfC{$y: p \imp q, \Gamma, z':s, z:r, \Gamma', \Delta, \Delta' \vdash r$}
                \DisplayProof
                &
                \tagarray{\label{comm_L_ex}}
            \end{tabular}
            \end{center}
        
            \begin{center}
            \begin{tabular}{ >{\centering}m{10cm} >{\centering}m{0.5cm}}
                \AxiomC{$\Gamma \vdash p$}
                \AxiomC{$\Delta, z:r, z':s,x:q, \Delta'\vdash r$}
                \RightLabel{$(L \imp)$}
                \BinaryInfC{$y: p\imp q, \Gamma, \Delta, z:r, z':s, \Delta'\vdash r$}
                \RightLabel{$(\operatorname{ex})$}
                \UnaryInfC{$y:p\imp q, \Gamma, \Delta, z':s, z:s, \Delta'\vdash r$}
                \DisplayProof\\\vspace{0.5cm}
                $\sim_c$\\\vspace{0.5cm}
                \AxiomC{$\Gamma \vdash p$}
                \AxiomC{$\Delta, z:r, z':s, x:q, \Delta'\vdash r$}
                \RightLabel{$(\operatorname{ex})$}
                \UnaryInfC{$\Delta, z':s, z:r, x:q, \Delta'\vdash r$}
                \RightLabel{$(L \imp)$}
                \BinaryInfC{$y: p \imp q, \Gamma, \Delta, z':s, z:r, \Delta' \vdash r$}
                \DisplayProof
                &
                \tagarray{\label{comm_L_ex2}}
            \end{tabular}
            \end{center}

            \begin{center}
            \begin{tabular}{ >{\centering}m{10cm} >{\centering}m{0.5cm}}
                \AxiomC{$\Gamma \vdash p$}
                \AxiomC{$\Delta \vdash q$}
                \AxiomC{$\Theta, x:r,y:l, \Theta' \vdash s$}
                \RightLabel{$(L \imp)$}
                \BinaryInfC{$z: q \imp l, \Delta, \Theta, x:r, \Theta' \vdash s$}
                \RightLabel{$(L \imp)$}
                \BinaryInfC{$z': p \imp r, \Gamma, z: q \imp l, \Delta, \Theta, \Theta' \vdash s$}
                \DisplayProof\\\vspace{0.5cm}
                $\sim_c$\\\vspace{0.5cm}
                \AxiomC{$\Delta \vdash q$}
                \AxiomC{$\Gamma \vdash p$}
                \AxiomC{$\Theta, x:r, y:l, \Theta' \vdash s$}
                \RightLabel{$(L \imp)$}
                \BinaryInfC{$z': p \imp r, \Gamma, \Theta, y:l, \Theta' \vdash r$}
                \RightLabel{$(L \imp)$}
                \BinaryInfC{$z:q \imp l ,\Delta,z': p \imp r, \Gamma, \Theta,   \Theta' \vdash s$}
                \doubleLine
                \RightLabel{$(\operatorname{ex})$}
                \UnaryInfC{$z': p \imp r, \Gamma, z: q \imp l, \Delta, \Theta, \Theta' \vdash s$}
                \DisplayProof
                &
                \tagarray{\label{comm_L_L}}
            \end{tabular}
            \end{center}
            
            \begin{center}
            \begin{tabular}{ >{\centering}m{10cm} >{\centering}m{0.5cm}}
            \AxiomC{$\Gamma \vdash l$}
                \AxiomC{$\Delta, x:p, \Delta' \vdash q$}
                \AxiomC{$\Theta, y:r, \Theta' \vdash s$}
                \RightLabel{$(L \imp)$}
                \BinaryInfC{$z: q \imp r, \Delta, x:p, \Delta', \Theta, \Theta' \vdash s$}
                \RightLabel{$(L \imp)$}
                \BinaryInfC{$z': l \imp p, \Gamma, z: q \imp r,\Delta, \Delta', \Theta,  \Theta' \vdash s$}
                \DisplayProof\\\vspace{0.5cm}
                $\sim_c$\\\vspace{0.5cm}
                \AxiomC{$\Gamma \vdash l$}
                \AxiomC{$\Delta, x:p, \Delta' \vdash q$}
                \RightLabel{$(L \imp)$}
                \BinaryInfC{$z':l\imp p, \Gamma, \Delta, \Delta' \vdash q$}
                \AxiomC{$\Theta, y:r, \Theta' \vdash s$}
                \RightLabel{$(L \imp)$}
                \BinaryInfC{$z: q \imp r, z':l\imp p,\Gamma, \Delta, \Delta', \Theta,  \Theta' \vdash s$}
                \doubleLine
                \RightLabel{$(\operatorname{ex})$}
                \UnaryInfC{$z': l \imp p, \Gamma, z: q \imp r,\Delta, \Delta', \Theta,  \Theta' \vdash s$}
                \DisplayProof
                &
                \tagarray{\label{comm_L_L2}}
            \end{tabular}
            \end{center}
\end{definition}

Should the order of two variables $x:p,y:p$ that are contracted be logically significant? We are prevented from identifying contraction on $x:p,y:p$ with contraction on $y:p,x:p$ because the former leaves $x:p$ and the latter $y:p$, but a sufficient cocommutativity principle is expressed by \eqref{co_ctr_comm_alt}. Similarly \eqref{co_ctr_assoc} expresses that contraction is coassociative. The rule \eqref{co_weak_ctr} is counitality, which says that we attach no logical meaning to contraction with a variable which has been weakened in. These principles assert that contraction is \emph{coalgebraic}, a point of view further ramified in linear logic.

\begin{definition}[($co$-equivalence)]
\label{co_equivalence} We define $\sim_{co}$ to be the smallest compatible equivalence relation on preproofs satisfying
    \begin{center}
    \begin{tabular}{ >{\centering}m{6cm} >{\centering}m{0.5cm} >{\centering}m{6cm} >{\centering}m{0.5cm}}
    \AxiomC{$\Gamma, x: p, y: p, z: p,\Delta \vdash q$}
        \RightLabel{$({\operatorname{ctr}})$}
        \UnaryInfC{$\Gamma, x:p, z: p, \Delta \vdash q$}
        \RightLabel{$({\operatorname{ctr}})$}
        \UnaryInfC{$\Gamma, x:p, \Delta \vdash q$}
        \DisplayProof
        &
        $\sim_{co}$
        &
        \AxiomC{$\Gamma, x: p, y: p, z: p,\Delta \vdash q$}
        \RightLabel{$({\operatorname{ctr}})$}
        \UnaryInfC{$\Gamma, x:p, y: p, \Delta \vdash q$}
        \RightLabel{$({\operatorname{ctr}})$}
        \UnaryInfC{$\Gamma, x:p, \Delta \vdash q$}
        \DisplayProof
        &
        \tagarray{\label{co_ctr_assoc}}
    \end{tabular}
    \end{center}
    \begin{center}
    \begin{tabular}{ >{\centering}m{6cm} >{\centering}m{0.5cm} >{\centering}m{6cm} >{\centering}m{0.5cm}}
            \AxiomC{$\pi$}
        \noLine
        \UnaryInfC{$\vdots$}
        \noLine
    \UnaryInfC{$\Gamma, x: p, y:p,\Delta \vdash q$}
        \RightLabel{$({\operatorname{ctr}})$}
        \UnaryInfC{$\Gamma, x:p, \Delta \vdash q$}
        \DisplayProof
        &
        $\sim_{co}$
        &
        \AxiomC{$\operatorname{subst}^{str}( \pi, y, x)$}
        \noLine
        \UnaryInfC{$\vdots$}
        \noLine
        \UnaryInfC{$\Gamma, x: p, x: p, \Delta \vdash q$}
        \RightLabel{$({\operatorname{ex}})$}
        \UnaryInfC{$\Gamma, x:p, x: p, \Delta \vdash q$}
        \RightLabel{$({\operatorname{ctr}})$}
        \UnaryInfC{$\Gamma, x:p, \Delta \vdash q$}
        \DisplayProof
        &
        \tagarray{\label{co_ctr_comm_alt}}
    \end{tabular}
    \end{center}
\begin{center}
\begin{tabular}{ >{\centering}m{6cm} >{\centering}m{0.5cm} >{\centering}m{6cm} >{\centering}m{0.5cm}} 
        \AxiomC{$\Gamma, x:p, \Delta \vdash q$}
        \DisplayProof
        & $\sim_{co}$ &
        \AxiomC{$\Gamma, x:p, \Delta\vdash q$}
        \RightLabel{$(\operatorname{weak})$}
        \UnaryInfC{$\Gamma, x:p,x':p, \Delta \vdash q$}
        \RightLabel{$(\operatorname{ctr})$}
        \UnaryInfC{$\Gamma, x:p, \Delta \vdash q$}
        \DisplayProof
        &
        \tagarray{\label{co_weak_ctr}}
\end{tabular}
\end{center}
\end{definition}

\begin{remark}\label{remark:zucker_1} The relation \eqref{co_ctr_assoc} appears as the contraction conversion \cite[\S 3.1.2 (b)(i)]{zucker} of Zucker. The sequent calculus of \cite{zucker} does not contain explicit weakening or exchange so the relations there do not include \eqref{co_ctr_comm_alt} or \eqref{co_weak_ctr}. This omission obscures the coalgebraic structure, which we believe to be an important logical principle.
\end{remark}

\begin{remark}\label{remark:old_co_ctr_comm2}
The rule \eqref{co_weak_ctr} has a left-handed version
\begin{center}
\begin{tabular}{ >{\centering}m{5cm} >{\centering}m{0.5cm} >{\centering}m{5cm}} 
        \AxiomC{$\pi$}
        \noLine
        \UnaryInfC{$\vdots$}
        \noLine
        \UnaryInfC{$\Gamma, x:p, \Delta\vdash q$}
        \RightLabel{$(\operatorname{weak})$}
        \UnaryInfC{$\Gamma, x':p,x:p, \Delta \vdash q$}
        \RightLabel{$(\operatorname{ctr})$}
        \UnaryInfC{$\Gamma, x':p, \Delta \vdash q$}
        \DisplayProof
        & $\stackrel{\eqref{co_ctr_comm_alt}}{\sim}$ &
        \AxiomC{$\operatorname{subst}^{str}(\pi,x,x')$}
        \noLine
        \UnaryInfC{$\vdots$}
        \noLine
        \UnaryInfC{$\Gamma, \textcolor{blue}{x'}:p, \Delta\vdash q$}
        \RightLabel{$(\operatorname{weak})$}
        \UnaryInfC{$\Gamma, x':p,\textcolor{blue}{x'}:p, \Delta \vdash q$}
        \RightLabel{$(\operatorname{ex}$)}
        \UnaryInfC{$\Gamma, \textcolor{blue}{x'}:p,x':p, \Delta \vdash q$}
        \RightLabel{$(\operatorname{ctr})$}
        \UnaryInfC{$\Gamma, \textcolor{blue}{x'}:p, \Delta \vdash q$}
        \DisplayProof
\end{tabular}
\end{center}
\begin{center}
\begin{tabular}{ >{\centering}m{5cm} >{\centering}m{0.5cm} >{\centering}m{5cm}} 
        & $\sim_{\tau}$ &
        \AxiomC{$\operatorname{subst}^{str}(\pi,x,x')$}
        \noLine
        \UnaryInfC{$\vdots$}
        \noLine
        \UnaryInfC{$\Gamma, \textcolor{blue}{x'}:p, \Delta\vdash q$}
        \RightLabel{$(\operatorname{weak})$}
        \UnaryInfC{$\Gamma, \textcolor{blue}{x'}:p,x':p, \Delta \vdash q$}
        \RightLabel{$(\operatorname{ctr})$}
        \UnaryInfC{$\Gamma, \textcolor{blue}{x'}:p, \Delta \vdash q$}
        \DisplayProof
\end{tabular}
\end{center}
\begin{center}
\begin{tabular}{ >{\centering}m{5cm} >{\centering}m{0.5cm} >{\centering}m{5cm}} 
        & $\stackrel{\eqref{co_weak_ctr}}{\sim}$ &
        \AxiomC{$\operatorname{subst}^{str}(\pi,x,x')$}
        \noLine
        \UnaryInfC{$\vdots$}
        \noLine
        \UnaryInfC{$\Gamma, \textcolor{blue}{x'}:p, \Delta \vdash q$}
        \DisplayProof
\end{tabular}
\end{center}
\end{remark}

Principles \eqref{lambda_L_L_ctr} and \eqref{lambda_weak_L} below are of profound importance, as they are the internal manifestation in our system of the Brouwer-Heyting-Kolmogorov interpretation of proofs in intuitionistic logic \cite{troelstra}. Under that interpretation a proof of a hypothesis $y: p \imp q$ reads as a transformation of proofs of $p$ to proofs of $q$. Rule \eqref{lambda_L_L_ctr} expresses that if the output proof of $q$ is to be used multiple times the transformation must be employed once for each copy. Rule \eqref{lambda_weak_L} expresses that if the output is not needed, neither is the transformation nor any of its inputs. Note that these principles mirror \eqref{cut:log_vs_ctr} and \eqref{cut:log_vs_weak} and therefore in some sense realise $(L \imp)$ as an \emph{internalised cut}. We develop this point of view more systematically in Section \ref{section:internal_bhk}.

\begin{definition}[($\lambda$-equivalence)]
\label{inefficiency}
We define $\sim_{\lambda}$ to be the smallest compatible equivalence relation on preproofs satisfying
    \begin{center}
    \begin{tabular}{ >{\centering}m{10cm} >{\centering}m{0.5cm}}
        \AxiomC{$\pi_1$}
        \noLine
        \UnaryInfC{$\vdots$}
        \noLine
        \UnaryInfC{$\Gamma \vdash p$}
        \AxiomC{$\pi_2$}
        \noLine
        \UnaryInfC{$\vdots$}
        \noLine
        \UnaryInfC{$\Delta, x:q, x':q,\Delta' \vdash r$}
        \RightLabel{$(\operatorname{ctr})$}
        \UnaryInfC{$\Delta, x:q, \Delta' \vdash r$}
        \RightLabel{$(L \imp)$}
        \BinaryInfC{$y: p \imp q, \Gamma, \Delta, \Delta' \vdash r$}
        \DisplayProof
        
        \vspace{0.5cm}
        $\sim_\lambda$
        \vspace{0.5cm}

        \AxiomC{$\pi_1$}
        \noLine
        \UnaryInfC{$\vdots$}
        \noLine
        \UnaryInfC{$\Gamma \vdash p$}
        \AxiomC{$\pi_1$}
        \noLine
        \UnaryInfC{$\vdots$}
        \noLine
        \UnaryInfC{$\Gamma \vdash p$}
        \AxiomC{$\pi_2$}
        \noLine
        \UnaryInfC{$\vdots$}
        \noLine
        \UnaryInfC{$\Delta, x:q, x': q, \Delta' \vdash r$}
        \RightLabel{$(L \imp)$}
        \BinaryInfC{$y': p \imp q, \Gamma, \Delta, x:q, \Delta' \vdash r$}
        \RightLabel{$(L \imp)$}
        \BinaryInfC{$y: p\imp q, \Gamma,y': p\imp q, \Gamma, \Delta,  \Delta' \vdash r$}
        \doubleLine
        \RightLabel{$(\operatorname{ex})$}
        \UnaryInfC{$y: p\imp q, y': p\imp q, \Gamma, \Gamma, \Delta,  \Delta' \vdash r$}
        \RightLabel{$(\operatorname{ctr})$}
        \UnaryInfC{$y: p \imp q, \Gamma, \Gamma, \Delta, \Delta'\vdash r$}
        \doubleLine
        \RightLabel{$(\operatorname{ctr/ex})$}
        \UnaryInfC{$y: p \imp q, \Gamma, \Delta, \Delta'\vdash r$}
        \DisplayProof        
        &
        \tagarray{\label{lambda_L_L_ctr}}
    \end{tabular}
    \end{center}
    
\begin{center}
\begin{tabular}{ >{\centering}m{6cm} >{\centering}m{0.5cm} >{\centering}m{6cm} >{\centering}m{0.5cm}} 
     \AxiomC{$\pi_1$}
     \noLine
     \UnaryInfC{$\vdots$}
     \noLine
     \UnaryInfC{$\Gamma \vdash p$}
     \AxiomC{$\pi_2$}
     \noLine
     \UnaryInfC{$\vdots$}
     \noLine
     \UnaryInfC{$\Delta, \Delta' \vdash r$}
     \RightLabel{$(\operatorname{weak})$}
     \UnaryInfC{$\Delta, x: q, \Delta' \vdash r$}
     \RightLabel{$(L \imp)$}
     \BinaryInfC{$y: p \imp q, \Gamma, \Delta, \Delta' \vdash r$}
     \DisplayProof &$\sim_\lambda$&
     \AxiomC{$\pi_2$}
     \noLine
     \UnaryInfC{$\vdots$}
     \noLine
     \UnaryInfC{$\Delta, \Delta' \vdash r$}
     \RightLabel{$(\operatorname{weak})$}
     \UnaryInfC{$y: p \imp q, \Delta, \Delta' \vdash r$}
     \doubleLine
     \RightLabel{$(\operatorname{weak})$}
     \UnaryInfC{$y: p \imp q, \Gamma, \Delta, \Delta' \vdash r$}
     \DisplayProof
     &
     \tagarray{\label{lambda_weak_L}}
\end{tabular}
\end{center}
\end{definition}


\begin{remark}\label{remark:zucker_2} The relation \eqref{lambda_L_L_ctr} appears as contraction conversion \cite[\S 3.1.2 (b)(iii)]{zucker} of Zucker. The sequent calculus of \cite{zucker} does not contain explicit weakening so the relations there do not include \eqref{lambda_weak_L}. The relation \eqref{lambda_L_L_ctr} is also implicit in \cite[Lemma 12]{kleene} and \cite[Lemma 2]{mints} and \eqref{lambda_weak_L} is implicit in \cite[Lemma 4]{kleene} and \cite[Lemma 1]{mints}.
\end{remark}

Let us consider the assertion that $(\operatorname{ax})$, which in our system is available for any formula $p$, should be restricted to \emph{atomic} formulas. Let $\Sigma^\Gamma_q$ denote the set of preproofs of $\Gamma \vdash q$ under our system and $\Pi^\Gamma_q$ the set of preproofs under this system with a restricted axiom rule. Clearly $\Pi^\Gamma_q \subseteq \Sigma^\Gamma_q$ and if $\Sigma^\Gamma_q$ is nonempty then so is $\Pi^\Gamma_q$. Since the restriction on the axiom rule does not affect provability we are free to adopt it, either directly by changing the deduction rules, or indirectly by keeping the deduction rules as given but adopting an equivalence relation on preproofs which effectively makes the axiom rule on compound formulas a derived rule:

\begin{definition}[($\eta$-equivalence)]
\label{etaequivalence}
We define $\sim_\eta$ to be the smallest compatible equivalence relation on preproofs such that for arbitrary formulas $p,q$
\begin{center}
\begin{tabular}{ >{\centering}m{7cm} >{\centering}m{0.5cm} >{\centering}m{5cm} >{\centering}m{0.5cm}} 
 \AxiomC{}
 \RightLabel{$(\operatorname{ax})$}
 \UnaryInfC{$x: p \vdash p$}
 \AxiomC{}
 \RightLabel{$(\operatorname{ax})$}
 \UnaryInfC{$y: q \vdash q$}
 \RightLabel{$(L\imp)$}
 \BinaryInfC{$z: p \imp q, x:p, \vdash q$}
 \RightLabel{$(R \imp)$}
 \UnaryInfC{$z: p\imp q \vdash p \imp q$}
 \DisplayProof
 &$\sim_\eta$&
 \AxiomC{}
 \RightLabel{$(\operatorname{ax})$}
 \UnaryInfC{$z: p \imp q \vdash p \imp q$}
 \DisplayProof
 &
 \tagarray{\label{eta}}
 \end{tabular}
\end{center}
\end{definition}

The logical justification of the cut-elimination transformations is the \emph{inversion principle} \cite[\S II]{prawitz} which states that the left introduction rule $(L \imp)$ is, in a sense, the inverse of the right introduction rule $(R \imp)$. This principle is made manifest in the cut-elimination theorem of Gentzen (Theorem \ref{cutfree}). To make the point in a slightly different way, note that the $(\operatorname{cut})$ rule asserts that an occurrence of $A$ on the left of the turnstile is precisely as strong as an occurrence on the right; see \cite[\S 3.2.1, \S 3.3.3]{girard_blind}. The cut-elimination theorem says that this balance of strength is implicit already in the rules without $(\operatorname{cut})$.

Note that rule \eqref{cut:ax_left} below uses ancestor substitution (Definition \ref{defn:ancestor_sub}). It is convenient to call a deduction rule \emph{proper} if it is not $\operatorname{(cut)}$. 

\begin{definition}[(Single step cut reduction)]
\label{cutreduction}
We define $\to_{\operatorname{cut}}$ to be the smallest compatible relation (not necessarily an equivalence relation) on preproofs containing:
\begin{itemize}
    \item For any proper deduction rule $(r)$ 
\begin{center}
\begin{tabular}{ >{\centering}m{7cm} >{\centering}m{0.5cm} >{\centering}m{4cm} >{\centering}m{0.5cm}} 
        \AxiomC{}
        \RightLabel{$({\operatorname{ax}})$}
        \UnaryInfC{$x:p \vdash p$}
        \AxiomC{$\pi$}
        \noLine
        \UnaryInfC{$\vdots$}
        \RightLabel{$(r)$}
        \UnaryInfC{$y:p, \Gamma \vdash q$}
        \RightLabel{$({\operatorname{cut}})$}
        \BinaryInfC{$x:p, \Gamma \vdash q$}
        \DisplayProof
        & $\to_{\operatorname{cut}}$ &
        \AxiomC{$\operatorname{subst}^{str}(\pi,y,x)$}
        \noLine
        \UnaryInfC{$\vdots$}
        \RightLabel{$(r)$}
        \UnaryInfC{$x:p, \Gamma \vdash q$}
        \DisplayProof
        & \tagarray{\label{cut:ax_left}}
\end{tabular}
\end{center}

\begin{center}
\begin{tabular}{>{\centering}m{7cm} >{\centering}m{0.5cm} >{\centering}m{4cm} >{\centering}m{0.5cm}} 
        \AxiomC{$\pi$}
        \noLine
        \UnaryInfC{$\vdots$}
        \RightLabel{$(r)$}
        \UnaryInfC{$\Gamma\vdash p$}
        \AxiomC{}
        \RightLabel{$({\operatorname{ax}})$}
        \UnaryInfC{$x:p \vdash p$}
        \RightLabel{$({\operatorname{cut}})$}
        \BinaryInfC{$\Gamma\vdash p$}
        \DisplayProof
        & $\to_{\operatorname{cut}}$ &
        \AxiomC{$\pi$}
        \noLine
        \UnaryInfC{$\vdots$}
        \RightLabel{$(r)$}
        \UnaryInfC{$\Gamma\vdash p$}
        \DisplayProof
        & \tagarray{\label{cut:ax_right}}
\end{tabular}
\end{center}
    \item Let $(r_0)$ be a structural rule, $(r)$ any proper deduction rule. Then 
\begin{center}
\begin{tabular}{>{\centering}m{6.5cm} >{\centering}m{0.5cm} >{\centering}m{6cm} >{\centering}m{0.5cm}}
        \AxiomC{$\pi_1$}
        \noLine
        \UnaryInfC{$\vdots$}
        \UnaryInfC{$\Gamma \vdash p$}
        \RightLabel{$(r_0)$}
        \UnaryInfC{$\Gamma' \vdash p$}
        \AxiomC{$\pi_2$}
        \noLine
        \UnaryInfC{$\vdots$}
        \RightLabel{$(r)$}
        \UnaryInfC{$y:p, \Delta \vdash s$}
        \RightLabel{$({\operatorname{cut}})$}
        \BinaryInfC{$\Gamma', \Delta \vdash s$}
        \DisplayProof
        &$\to_{\operatorname{cut}}$&
        \AxiomC{$\pi_1$}
        \noLine
        \UnaryInfC{$\vdots$}
        \UnaryInfC{$\Gamma \vdash p$}
        \AxiomC{$\pi_2$}
        \noLine
        \UnaryInfC{$\vdots$}
        \RightLabel{$(r)$}
        \UnaryInfC{$y:p, \Delta \vdash s$}
        \RightLabel{$({\operatorname{cut}})$}
        \BinaryInfC{$\Gamma, \Delta \vdash s$}
        \RightLabel{$(r_0)$}
        \UnaryInfC{$\Gamma', \Delta \vdash s$}
        \DisplayProof
        & \tagarray{\label{cut:struc_vs_any}}
\end{tabular}
\end{center}
\item $(L \imp)$ on the left and $(r)$ any proper deduction rule 
\begin{center}
\begin{tabular}{>{\centering}m{10cm} >{\centering}m{1cm}}
            \AxiomC{$\pi_1$}
            \noLine
            \UnaryInfC{$\vdots$}
            \noLine
            \UnaryInfC{$\Gamma \vdash p$}
            \AxiomC{$\pi_2$}
            \noLine
            \UnaryInfC{$\vdots$}
            \noLine
            \UnaryInfC{$\Delta, x:q, \Theta \vdash s$}
            \RightLabel{$(L\imp)$}
            \BinaryInfC{$y:p \imp q, \Gamma, \Delta, \Theta \vdash s$}
            \AxiomC{$\pi_3$}
            \noLine
            \UnaryInfC{$\vdots$}
            \RightLabel{$(r)$}
            \UnaryInfC{$z:s, \Lambda \vdash l$}
            \RightLabel{$({\operatorname{cut}})$}
            \BinaryInfC{$y: p \imp q, \Gamma, \Delta, \Theta, \Lambda \vdash l$}
            \DisplayProof\\\vspace{0.5cm}
            $\to_{\operatorname{cut}}$\\\vspace{0.5cm}
            \AxiomC{$\pi_1$}
            \noLine
            \UnaryInfC{$\vdots$}
            \noLine
            \UnaryInfC{$\Gamma \vdash p$}
            \AxiomC{$\pi_2$}
            \noLine
            \UnaryInfC{$\vdots$}
            \noLine
            \UnaryInfC{$\Delta, x:q, \Theta \vdash s$}
            \AxiomC{$\pi_3$}
            \noLine
            \UnaryInfC{$\vdots$}
            \RightLabel{$(r)$}
            \UnaryInfC{$z:s, \Lambda \vdash l$}
            \RightLabel{$({\operatorname{cut}})$}
            \BinaryInfC{$\Delta, x:q, \Theta, \Lambda \vdash l$}
            \RightLabel{$(L \imp)$}
            \BinaryInfC{$y: p \imp q, \Gamma, \Delta, \Theta, \Lambda \vdash l$}
            \DisplayProof
            &
            \tagarray{\label{cut:L_vs_any}}
        \end{tabular}
        \end{center}
        
\item For any logical rule $(r_1)$ and structural rule $(r_0)$, where the cut variable $x:p$ was not manipulated by $(r_0)$:
\begin{center}
\begin{tabular}{>{\centering}m{6.5cm} >{\centering}m{0.5cm} >{\centering}m{6.5cm} >{\centering}m{0.5cm}}
        \AxiomC{$\pi_1$}
        \noLine
        \UnaryInfC{$\vdots$}
        \RightLabel{$(r_1)$}
        \UnaryInfC{$\Gamma \vdash p$}
        \AxiomC{$\pi_2$}
        \noLine
        \UnaryInfC{$\vdots$}
        \noLine
        \UnaryInfC{$x:p, \Delta \vdash q$}
        \RightLabel{$(r_0)$}
        \UnaryInfC{$x:p, \Delta' \vdash q$}
        \RightLabel{$(\operatorname{cut})$}
        \BinaryInfC{$\Gamma, \Delta' \vdash q$}
        \DisplayProof
        & $\to_{\operatorname{cut}}$ &
        \AxiomC{$\pi_1$}
        \noLine
        \UnaryInfC{$\vdots$}
        \RightLabel{$(r_1)$}
        \UnaryInfC{$\Gamma \vdash p$}
        \AxiomC{$\pi_2$}
        \noLine
        \UnaryInfC{$\vdots$}
        \noLine
        \UnaryInfC{$x:p, \Delta \vdash q$}
        \RightLabel{$(\operatorname{cut})$}
        \BinaryInfC{$\Gamma, \Delta \vdash q$}
        \RightLabel{$(r_0)$}
        \UnaryInfC{$\Gamma, \Delta' \vdash q$}
        \DisplayProof
        & \tagarray{\label{cut:log_vs_struc_np}}
\end{tabular}
\end{center}

\item For any logical rule $(r_1)$:
\begin{center}
\begin{tabular}{>{\centering}m{10cm} >{\centering}m{1cm}}
        \AxiomC{$\pi_1$}
        \noLine
        \UnaryInfC{$\vdots$}
        \RightLabel{$(r_1)$}
        \UnaryInfC{$\Gamma \vdash p$}
        \AxiomC{$\pi_2$}
        \noLine
        \UnaryInfC{$\vdots$}
        \noLine
        \UnaryInfC{$\Delta, y:p, y':p, \Theta \vdash s$}
        \RightLabel{$({\operatorname{ctr}})$}
        \UnaryInfC{$\Delta, y:p, \Theta \vdash s$}
        \RightLabel{$({\operatorname{cut}})$}
        \BinaryInfC{$\Gamma, \Delta, \Theta \vdash s$}
        \DisplayProof\\\vspace{0.5cm}
         $\to_{\operatorname{cut}}$\\\vspace{0.5cm}
        \AxiomC{$\pi_1$}
        \noLine
        \UnaryInfC{$\vdots$}
        \RightLabel{$(r_1)$}
        \UnaryInfC{$\Gamma \vdash p$}
        \AxiomC{$\pi_1$}
        \noLine
        \UnaryInfC{$\vdots$}
        \RightLabel{$(r_1)$}
        \UnaryInfC{$\Gamma \vdash p$}
        \AxiomC{$\pi_2$}
        \noLine
        \UnaryInfC{$\vdots$}
        \noLine
        \UnaryInfC{$\Delta, y:p, y':p, \Theta \vdash s$}
        \RightLabel{$({\operatorname{cut}})$}
        \BinaryInfC{$\Gamma, \Delta, y':p, \Theta \vdash s$}
        \RightLabel{$({\operatorname{cut}})$}
        \BinaryInfC{$\Gamma, \Gamma, \Delta, \Theta \vdash s$}
        \RightLabel{$({\operatorname{ctr/ex}})$}
        \doubleLine
        \UnaryInfC{$\Gamma, \Delta, \Theta \vdash s$}
        \DisplayProof
        & 
        \tagarray{\label{cut:log_vs_ctr}}
\end{tabular}
\end{center}
\begin{center}
\begin{tabular}{>{\centering}m{6cm} >{\centering}m{0.5cm} >{\centering}m{4cm} >{\centering}m{0.5cm}}
        \AxiomC{$\pi_1$}
        \noLine
        \UnaryInfC{$\vdots$}
        \RightLabel{$(r_1)$}
        \UnaryInfC{$\Gamma \vdash p$}
        \AxiomC{$\pi_2$}
        \noLine
        \UnaryInfC{$\vdots$}
        \noLine
        \UnaryInfC{$\Delta, \Theta \vdash q$}
        \RightLabel{$({\operatorname{weak}})$}
        \UnaryInfC{$\Delta, x:p, \Theta \vdash q$}
        \RightLabel{$({\operatorname{cut}})$}
        \BinaryInfC{$\Gamma, \Delta, \Theta \vdash q$}
        \DisplayProof
        & $\to_{\operatorname{cut}}$ &
        \AxiomC{$\pi_2$}
        \noLine
        \UnaryInfC{$\vdots$}
        \noLine
        \UnaryInfC{$\Delta, \Theta \vdash q$}
        \RightLabel{$({\operatorname{weak}})$}
        \doubleLine
        \UnaryInfC{$\Gamma, \Delta, \Theta \vdash q$}
        \DisplayProof
        &
        \tagarray{\label{cut:log_vs_weak}}
\end{tabular}
\end{center}


\end{itemize}
The remaining cases correspond to having a logical rule on both the left and right:
\begin{itemize}
\item $(R\imp)$ on the left and $(R\imp)$ on the right: 
\begin{center}
\begin{tabular}{>{\centering}m{10cm} >{\centering}m{1cm}}
        \AxiomC{$\pi_1$}
        \noLine
        \UnaryInfC{$\vdots$}
        \noLine
        \UnaryInfC{$\Gamma, x:p, \Delta \vdash q$}
        \RightLabel{$(R\imp)$}
        \UnaryInfC{$\Gamma, \Delta \vdash p \imp q$}
        \AxiomC{$\pi_2$}
        \noLine
        \UnaryInfC{$\vdots$}
        \noLine
        \UnaryInfC{$y: p \imp q, \Theta, z:s \vdash l$}
        \RightLabel{$(R\imp)$}
        \UnaryInfC{$y: p \imp q, \Theta \vdash s \imp l$}
        \RightLabel{$({\operatorname{cut}})$}
        \BinaryInfC{$\Gamma, \Delta, \Theta \vdash s \imp l$}
        \DisplayProof\\\vspace{0.5cm}
        $\to_{\operatorname{cut}}$\\\vspace{0.5cm} 
        \AxiomC{$\pi_1$}
        \noLine
        \UnaryInfC{$\vdots$}
        \noLine
        \UnaryInfC{$\Gamma, x:p, \Delta \vdash q$}
        \RightLabel{$(R\imp)$}
        \UnaryInfC{$\Gamma, \Delta \vdash p \imp q$}
        \AxiomC{$\pi_2$}
        \noLine
        \UnaryInfC{$\vdots$}
        \noLine
        \UnaryInfC{$y: p \imp q, \Theta, z : s \vdash l$}
        \RightLabel{$({\operatorname{cut}})$}
        \BinaryInfC{$\Gamma, \Delta, \Theta, z:s \vdash l$}
        \RightLabel{$(R\imp)$}
        \UnaryInfC{$\Gamma, \Delta, \Theta \vdash s \imp l$}
        \DisplayProof
        &
        \tagarray{\label{cut:R_vs_R}}
        \end{tabular}
        \end{center}
        
    \item $(R\imp)$ on the left and $(L\imp)$ on the right: 
\begin{center}
\begin{tabular}{>{\centering}m{10cm} >{\centering}m{1cm}}
            \AxiomC{$\pi_1$}
            \noLine
            \UnaryInfC{$\vdots$}
            \noLine
            \UnaryInfC{$\Gamma, x:p, \Delta \vdash q$}
            \RightLabel{$(R\imp)$}
            \UnaryInfC{$\Gamma, \Delta \vdash p \imp q$}
            \AxiomC{$\pi_2$}
            \noLine
            \UnaryInfC{$\vdots$}
            \noLine
            \UnaryInfC{$\Theta \vdash p$}
            \AxiomC{$\pi_3$}
            \noLine
            \UnaryInfC{$\vdots$}
            \noLine
            \UnaryInfC{$\Lambda, x': q, \Omega \vdash s$}
            \RightLabel{$(L\imp)$}
            \BinaryInfC{$y: p \imp q, \Theta, \Lambda, \Omega \vdash s$}
            \RightLabel{$({\operatorname{cut}})$}
            \BinaryInfC{$\Gamma, \Delta, \Theta, \Lambda, \Omega \vdash s$}
            \DisplayProof\\\vspace{0.5cm}
            $\to_{\operatorname{cut}}$\\\vspace{0.5cm}
            \AxiomC{$\pi_2$}
            \noLine
            \UnaryInfC{$\vdots$}
            \noLine
            \UnaryInfC{$\Theta \vdash p$}
            \AxiomC{$\pi_1$}
            \noLine
            \UnaryInfC{$\vdots$}
            \noLine
            \UnaryInfC{$\Gamma, x:p, \Delta \vdash q$}
            \RightLabel{$({\operatorname{cut}})$}
            \BinaryInfC{$\Theta, \Gamma, \Delta \vdash q$}
            \AxiomC{$\pi_3$}
            \noLine
            \UnaryInfC{$\vdots$}
            \noLine
            \UnaryInfC{$\Lambda, x':q, \Omega \vdash s$}
            \RightLabel{$({\operatorname{cut}})$}
            \BinaryInfC{$\Theta, \Gamma, \Delta, \Lambda, \Omega \vdash s$}
            \doubleLine
            \RightLabel{$({\operatorname{ex}})$}
            \UnaryInfC{$\Gamma, \Delta, \Theta, \Lambda, \Omega \vdash s$}
            \DisplayProof
            &
        \tagarray{\label{cut:R_vs_L}}
        \end{tabular}
        \end{center}
        \item $(R \imp)$ on the left and $(L \imp)$ on the right but $(L \imp)$ does not introduce the variable which is involved in the $(\operatorname{cut})$
\begin{center}
\begin{tabular}{>{\centering}m{10cm} >{\centering}m{1cm}}
            \AxiomC{$\pi_1$}
            \noLine
            \UnaryInfC{$\vdots$}
            \RightLabel{$(R\imp)$}
            \UnaryInfC{$\Gamma \vdash p$}
            \AxiomC{$\pi_2$}
            \noLine
            \UnaryInfC{$\vdots$}
            \noLine
            \UnaryInfC{$\Delta \vdash q$}
            \AxiomC{$\pi_3$}
            \noLine
            \UnaryInfC{$\vdots$}
            \noLine
            \UnaryInfC{$x:p, \Theta, y: r, \Lambda \vdash s$}
            \RightLabel{$(L \imp)$}
            \BinaryInfC{$z: q \imp r, \Delta, x:p, \Theta, \Lambda \vdash s$}
            \RightLabel{$(\operatorname{cut})$}
            \BinaryInfC{$\Gamma,z: q \imp r, \Delta, \Theta, \Lambda \vdash s$}
            \DisplayProof\\\vspace{0.5cm}
            $\to_{\operatorname{cut}}$\\\vspace{0.5cm}
            \AxiomC{$\pi_2$}
            \noLine
            \UnaryInfC{$\vdots$}
            \noLine
            \UnaryInfC{$\Delta \vdash q$}
            \AxiomC{$\pi_1$}
            \noLine
            \UnaryInfC{$\vdots$}
            \RightLabel{$(R \imp)$}
            \UnaryInfC{$\Gamma \vdash p$}
            \AxiomC{$\pi_3$}
            \noLine
            \UnaryInfC{$\vdots$}
            \noLine
            \UnaryInfC{$x:p, \Theta, y:r, \Lambda \vdash s$}
            \RightLabel{$(\operatorname{cut})$}
            \BinaryInfC{$\Gamma, \Theta, y:r, \Lambda \vdash s$}
            \RightLabel{$(L \imp)$}
            \BinaryInfC{$z:q \imp r, \Delta, \Gamma, \Theta, \Lambda \vdash s$}
            \doubleLine
            \RightLabel{$(\operatorname{ex})$}
            \UnaryInfC{$\Gamma, z: q \imp r, \Delta, \Theta, \Lambda \vdash s$}
            \DisplayProof
            &
            \tagarray{\label{cut:R_vs_L_nonp}}
        \end{tabular}
        \end{center}
        and
\begin{center}
\begin{tabular}{>{\centering}m{10cm} >{\centering}m{1cm}}
            \AxiomC{$\pi_1$}
            \noLine
            \UnaryInfC{$\vdots$}
            \RightLabel{$(R\imp)$}
            \UnaryInfC{$\Gamma \vdash p$}
            \AxiomC{$\pi_2$}
            \noLine
            \UnaryInfC{$\vdots$}
            \noLine
            \UnaryInfC{$x:p, \Delta \vdash q$}
            \AxiomC{$\pi_3$}
            \noLine
            \UnaryInfC{$\vdots$}
            \noLine
            \UnaryInfC{$\Theta, y:r, \Lambda \vdash s$}
            \RightLabel{$(L \imp)$}
            \BinaryInfC{$z: q \imp r, x:p, \Delta, \Theta, \Lambda \vdash s$}
            \RightLabel{$(\operatorname{cut})$}
            \BinaryInfC{$\Gamma, z: q \imp r, \Delta, \Theta, \Lambda \vdash s$}
            \DisplayProof\\\vspace{0.5cm}
            $\to_{\operatorname{cut}}$\\\vspace{0.5cm}
            \AxiomC{$\pi_1$}
            \noLine
            \UnaryInfC{$\vdots$}
            \RightLabel{$(R \imp)$}
            \UnaryInfC{$\Gamma \vdash p$}
            \AxiomC{$\pi_2$}
            \noLine
            \UnaryInfC{$\vdots$}
            \noLine
            \UnaryInfC{$x:p, \Delta \vdash q$}
            \RightLabel{$(\operatorname{cut})$}
            \BinaryInfC{$\Gamma, \Delta \vdash q$}
            \AxiomC{$\pi_3$}
            \noLine
            \UnaryInfC{$\vdots$}
            \noLine
            \UnaryInfC{$\Theta, y:r, \Lambda \vdash q$}
            \RightLabel{$(L \imp)$}
            \BinaryInfC{$z: q \imp r, \Gamma, \Delta, \Theta, \Lambda \vdash s$}
            \doubleLine
            \RightLabel{$(\operatorname{ex})$}
            \UnaryInfC{$\Gamma, z: q \imp r, \Delta, \Theta, \Lambda \vdash s$}
            \DisplayProof
            &
            \tagarray{\label{cut:R_vs_L_nonp2}}
        \end{tabular}
        \end{center}
    \end{itemize}
\end{definition}

\begin{definition}\label{cutequivalence}
We define $\sim_{\operatorname{cut}}$ to be the smallest equivalence relation on preproofs containing the relation $\to_{\operatorname{cut}}$.
\end{definition}


\begin{definition}[(Proof equivalence)]
We define $\sim_p$ to be the smallest compatible equivalence relation on preproofs containing the union of
\begin{itemize}
\item $\alpha$-equivalence (Definition \ref{alphaequivalence}),
\item $\tau$-equivalence (Definition \ref{tauequivalence}),
\item Commuting conversions (Definition \ref{commutingequivalence}),
\item $co$-equivalence (Definition \ref{co_equivalence}),
\item $\lambda$-equivalence (Definition \ref{inefficiency}),
\item $\eta$-equivalence (Definition \ref{etaequivalence}),
\item Cut equivalence (Definition \ref{cutequivalence}).
\end{itemize}
A \emph{proof} is an equivalence class of preproofs under proof equivalence. We say that two preproofs are \emph{equivalent} if they are equivalent under $\sim_p$.
\end{definition}

\subsection{Cut-elimination}

Why give yet another proof of cut-elimination? The structure of our proof is similar to Gentzen's \cite{gentzen} but we avoid the ``mix'' rule by making use of commuting conversions. We include the details so as to make clear which conversions are used. The treatment in the literature most similar to ours is \cite{cutctr}, however there the induction is structured differently and the focus is on weakening rather than contraction trees.

At a conceptual level, in order to justify the generating rules for proof equivalence, particularly the $\lambda$-equivalence rules, we have chosen our cut-elimination transformations (Definition \ref{cutreduction}) to bring out as clearly as possible the parallels between $(\operatorname{cut})$ and $(L \imp)$ (see Remark \ref{remark:limp_internal_cut}). Our proof of cut-elimination reinforces this connection, with some of the key steps in eliminating $(\operatorname{cut})$ repeated below to eliminate a subset of $(L \imp)$ rules in Section \ref{chiso} (see Lemma \ref{lemma:L_normal_form} and Lemma \ref{lemma:preimage_app}).

\begin{definition}
The \emph{width} $w(q)$ of a formula $q$ is the number of occurrences of $\imp$.
\end{definition}

\begin{definition}
The \emph{height} of a preproof $\pi$, denoted $h(\pi)$, is one less than the number of deduction rules encountered on the longest path in the underlying tree of the preproof.
\end{definition}

Note that two preproofs can be equivalent under $\sim_p$ but have different heights. A proof consisting of an axiom rule has height zero. A preproof which does not contain an occurrence of the $(\operatorname{cut})$ rule is called \emph{cut-free}.

\begin{thm}
\label{cutfree}
Every preproof is equivalent under $\sim_p$ to a cut-free preproof.
\end{thm}
\begin{proof}
Given any preproof $\pi$, we can choose an instance of the $(\operatorname{cut})$ rule in $\pi$ which is at the greatest possible height, and apply Proposition \ref{prop:actualwork} below to the subproof given by taking this as the root. Iterating this finitely many times yields the result.
\end{proof}

With reference to the prototype contraction in Definition \ref{defn:deduction_rules} we say that the variables $x:p, y:p$ are \emph{involved} in that deduction rule.

\begin{definition}\label{defn:active_contract} Let $\pi$ be a preproof. We say that a particular instance of $(\operatorname{ctr})$ in the proof tree is \emph{active} with respect to an occurrence of a variable $x:p$ in the preproof if the involved variables in the contraction are weak ancestors of that occurrence.
\end{definition}

We begin with an easy special case:

\begin{lemma}\label{lemma:ax_actualwork} Suppose given a preproof $\pi$ of the form
\begin{center}
    \AxiomC{$\pi_1$}
    \noLine
    \UnaryInfC{$\vdots$}
    \RightLabel{$(R \imp)$}
    \UnaryInfC{$\Gamma \vdash p$}
    \AxiomC{$\pi_2$}
    \noLine
    \UnaryInfC{$\vdots$}
    \RightLabel{$(r)$}
    \UnaryInfC{$x:p, \Delta \vdash q$}
    \RightLabel{$(\operatorname{cut})$}
    \BinaryInfC{$\Gamma, \Delta \vdash q$}
    \DisplayProof
\end{center}
where $\pi_1$ and $\pi_2$ are both cut-free, the cut variable $x:p$ is introduced in $\pi_2$ by an axiom rule and $\pi_2$ contains no active contractions with respect to the displayed occurrence of $x:p$. Then $\pi$ is equivalent under $\sim_p$ to a cut-free preproof.
\end{lemma}
\begin{proof}
By induction on the height of $\pi_2$. In the base case $\pi$ is
\begin{center}
        \AxiomC{$\pi_1$}
        \noLine
        \UnaryInfC{$\vdots$}
        \RightLabel{$(R \imp)$}
        \UnaryInfC{$\Gamma \vdash p$}
        \AxiomC{}
        \RightLabel{$(\operatorname{ax})$}
        \UnaryInfC{$x:p \vdash p$}
        \RightLabel{$(\operatorname{cut})$}
        \BinaryInfC{$\Gamma \vdash p$}
        \DisplayProof
\end{center}
which is equivalent by \eqref{cut:ax_right} to $\pi_1$. For the inductive step where $\pi_2$ has height $> 0$ we break into cases depending on the rule $(r)$:
\begin{itemize}
\item $(r)$ is a structural rule. Since $x:p$ is introduced by $(\operatorname{ax})$ and there are no active contractions in $\pi_2$, the cut variable $x:p$ is not manipulated by $(r)$ and so this case follows by the inductive hypothesis and \eqref{cut:log_vs_struc_np}.
\item $(r) = (R \imp)$ by \eqref{cut:R_vs_R} and the inductive hypothesis.
\item $(r) = (L \imp)$ by \eqref{cut:R_vs_L_nonp} and \eqref{cut:R_vs_L_nonp2} and the inductive hypothesis, using that $x:p$ is not introduced by $(L \imp)$.
\end{itemize}
This completes the inductive step and the proof of the lemma.
\end{proof}

\begin{proposition}\label{prop:actualwork} Any preproof $\pi$ of the form
\begin{center}
    \AxiomC{$\pi_1$}
    \noLine
    \UnaryInfC{$\vdots$}
    \RightLabel{$(r_1)$}
    \UnaryInfC{$\Gamma \vdash p$}
    \AxiomC{$\pi_2$}
    \noLine
    \UnaryInfC{$\vdots$}
    \RightLabel{$(r_2)$}
    \UnaryInfC{$x:p, \Delta \vdash q$}
    \RightLabel{$(\operatorname{cut})$}
    \BinaryInfC{$\Gamma, \Delta \vdash q$}
    \DisplayProof
\end{center}
where $\pi_1$ and $\pi_2$ are both cut-free, is equivalent under $\sim_p$ to a cut-free preproof.
\end{proposition}
\begin{proof}
Let $P(w, n)$ denote the following statement: any preproof $\pi$ with cut-free branches $\pi_1, \pi_2$ and final cut variable $x:p$ (as above) satisfying $w(p) = w$ and $n = h(\pi_1) + h(\pi_2)$ is equivalent under $\sim_p$ to a cut-free preproof. Let $P(w)$ denote $\forall n P(w,n)$. We will prove $\forall w P(w)$ by induction on $w$. Thus we must show $P(0)$ and that if for all $v < w$ $P(v)$ then $P(w)$. We refer to this as the \emph{outer induction}.

\textbf{Base case of the outer induction:} to prove $P(0)$ (that is, $\forall n P(0,n)$) we proceed by induction on $n$, which we refer to as the \emph{inner induction}. In the base case $P(0,0)$ of the inner induction both $(r_1),(r_2)$ are axiom rules, so the claim follows from \eqref{cut:ax_left}, \eqref{cut:ax_right}. Now assume $n > 0$ and that $P(0,k)$ holds for all $k < n$. If $(r_1)$ is $(\operatorname{ax})$ then we are done by \eqref{cut:ax_left}. If $(r_1)$ is a structural rule then the claim follows by applying the inner inductive hypothesis and \eqref{cut:struc_vs_any}. If $(r_1)$ is a logical rule then since $w(x:p) = 0$ it must be $(L \imp)$ and the claim follows from \eqref{cut:L_vs_any} and the inner inductive hypothesis.

\textbf{Inductive step of the outer induction}: now suppose that $w > 0$ is fixed and $P(v)$ holds for all $v < w$. To prove $P(w)$ (that is, $\forall n P(w,n)$) we proceed by induction on $n$, which we again refer to as the inner induction. If $n \le 1$ then one of $(r_1),(r_2)$ is $(\operatorname{ax})$ so the claim follows from \eqref{cut:ax_left}, \eqref{cut:ax_right}. Suppose now that $n > 1$ and that $P(w,k)$ holds for all $k < n$. We again divide into cases depending on the final deduction rules $(r_1),(r_2)$. Some cases follow from the inner inductive hypothesis as in the proof of the base case of the outer induction above, and we will not repeat them. The new cases that are easily dispensed with:
\begin{itemize}
\item $(r_1) = (R \imp), (r_2) = (R \imp)$ follows by \eqref{cut:R_vs_R} and the inner inductive hypothesis.
\item $(r_1) = (R \imp), (r_2) = (L \imp)$ may be divided into two subcases. Either the $(L \imp)$ does not introduce the cut variable $x:p$, in which case the claim follows by \eqref{cut:R_vs_L_nonp} and the inner inductive hypothesis, or the $(L \imp)$ \emph{does} introduce the cut variable $x$ of type $p = r \imp s$, in which case $\pi$ is by \eqref{cut:R_vs_L} equivalent to a proof of the form
\begin{center}
            \AxiomC{$\pi'_2$}
            \noLine
            \UnaryInfC{$\vdots$}
            \noLine
            \UnaryInfC{$\Theta \vdash r$}
            \AxiomC{$\pi'_1$}
            \noLine
            \UnaryInfC{$\vdots$}
            \noLine
            \UnaryInfC{$\Gamma', y:r, \Gamma'' \vdash s$}
            \RightLabel{$({\operatorname{cut}})$}
            \BinaryInfC{$\Theta, \Gamma', \Gamma'' \vdash s$}
            \AxiomC{$\pi''_2$}
            \noLine
            \UnaryInfC{$\vdots$}
            \noLine
            \UnaryInfC{$\Lambda, z:s, \Omega \vdash s$}
            \RightLabel{$({\operatorname{cut}})$}
            \BinaryInfC{$\Theta, \Gamma', \Gamma'', \Lambda, \Omega \vdash s$}
            \doubleLine
            \RightLabel{$({\operatorname{ex}})$}
            \UnaryInfC{$\Gamma, \Delta \vdash q$}
            \DisplayProof
\end{center}
where $\Gamma = \Gamma', \Gamma''$ and $\Delta = \Theta, \Lambda, \Omega$. Since both of these cuts involve types of lower width than $p$, the claim follows from the outer inductive hypothesis.

\item $(r_1)$ is logical and $(r_2)$ is one of $(\operatorname{weak}), (\operatorname{ex})$ follow from the inner inductive hypothesis and \eqref{cut:log_vs_weak}, \eqref{tau_cut} respectively.
\item $(r_1)$ is $(L \imp)$ and $(r_2)$ is $(\operatorname{ctr})$ follows as above in the proof of the base case of the outer induction, by the inner inductive hypothesis and \eqref{cut:L_vs_any}.
\end{itemize}
The only remaining case is where $(r_1) = (R \imp)$ and $(r_2) = (\operatorname{ctr})$, which will occupy the rest of the proof. In this case $\pi$ is of the form
    \begin{center}
    \begin{tabular}{ >{\centering}m{10cm} >{\centering}m{0.5cm}}
    \AxiomC{$\pi_1$}
    \noLine
    \UnaryInfC{$\vdots$}
    \RightLabel{$(R \imp)$}
            \UnaryInfC{$\Gamma \vdash p$}
        \AxiomC{$\pi'_2$}
        \noLine
        \UnaryInfC{$\vdots$}
        \noLine
        \UnaryInfC{$x_1:p, x_2:p, \Delta \vdash q$}
        \RightLabel{$({\operatorname{ctr}})$}
        \UnaryInfC{$x_1:p, \Delta \vdash q$}
        \RightLabel{$({\operatorname{cut}})$}
        \BinaryInfC{$\Gamma, \Delta \vdash q$}
        \DisplayProof
        &
\end{tabular}
\end{center}
where we set $x_1 = x$. Using $\tau$-equivalence, commuting conversions and $co$-equivalence we can manipulate $\pi_2$ (meaning $\pi_2'$ plus the final contraction) so that all the active contractions with respect to the cut variable $x_1:p$ occur at the bottom of the proof tree (see Lemma \ref{lemma:contract_normal_form} and Remark \ref{remark:coassoc_tree} below). Note that the final deduction rule of $\pi_2$ is, by hypothesis, an active contraction.\footnote{It is possible that $\pi_2$ contains other contractions on variables of type $p$, perhaps even the variable $x:p$ but which are not weak ancestors of the cut variable; these we all ignore.} After this step we see that $\pi$ is equivalent to
\begin{center}
\begin{tabular}{ >{\centering}m{10cm} >{\centering}m{0.5cm}}
        \AxiomC{$\pi_1$}
        \noLine
        \UnaryInfC{$\vdots$}
        \RightLabel{$(R \imp)$}
        \UnaryInfC{$\Gamma \vdash p$}
        \AxiomC{$\pi_2''$}
        \noLine
        \UnaryInfC{$\vdots$}
        \RightLabel{$(r)$}
        \UnaryInfC{$x_1, x_2, \ldots, x_l, \Delta \vdash q$}
        \RightLabel{$(\operatorname{ctr})$}
        \UnaryInfC{$x_1, x_2, \ldots, x_{l-1}, \Delta \vdash q$}
        \noLine
        \UnaryInfC{$\vdots$}
        \RightLabel{$(\operatorname{ctr})$}
        \UnaryInfC{$x_1, x_2, \Delta \vdash q$}
        \RightLabel{$({\operatorname{ctr}})$}
        \UnaryInfC{$x_1, \Delta \vdash q$}
        \RightLabel{$({\operatorname{cut}})$}
        \BinaryInfC{$\Gamma, \Delta \vdash q$}
        \DisplayProof
        &
        \tagarray{\label{eq:cut_tree_actualwork}}
\end{tabular}
    \end{center}
where $\pi_2''$ is cut-free and contains no active contractions with respect to $x_1$. To reduce clutter we have dropped the types from the variables $x_i:p$. By repeated applications of \eqref{cut:log_vs_ctr} we obtain the following preproof equivalent to $\pi$:
\begin{center}
        \AxiomC{$\pi_1$}
        \noLine
        \UnaryInfC{$\vdots$}
        \RightLabel{$(R \imp)$}
        \UnaryInfC{$\Gamma \vdash p$}
        \AxiomC{$\pi_1$}
        \noLine
        \UnaryInfC{$\vdots$}
        \RightLabel{$(R \imp)$}
        \UnaryInfC{$\Gamma \vdash p$}
        \AxiomC{$\pi_1$}
        \noLine
        \UnaryInfC{$\vdots$}
        \RightLabel{$(R \imp)$}
        \UnaryInfC{$\Gamma \vdash p$}
        \AxiomC{$\pi_1$}
        \noLine
        \UnaryInfC{$\vdots$}
        \RightLabel{$(R \imp)$}
        \UnaryInfC{$\Gamma \vdash p$}
        \AxiomC{$\pi_2''$}
        \noLine
        \UnaryInfC{$\vdots$}
        \RightLabel{$(r)$}
        \UnaryInfC{$x_1,\ldots,x_l,\Delta \vdash q$}
        \RightLabel{$(\operatorname{cut})$}
        \BinaryInfC{$\Gamma, x_1,\ldots,x_{l-1}, \Delta \vdash q$}
        \UnaryInfC{$\vdots$}
        \RightLabel{$(\operatorname{cut})$}
        \BinaryInfC{$(l-2)\Gamma,x_1,x_2,\Delta \vdash q$}
        \RightLabel{$({\operatorname{cut}})$}
        \BinaryInfC{$(l-1)\Gamma, x_1, \Delta \vdash q$}
        \RightLabel{$({\operatorname{cut}})$}
        \BinaryInfC{$l \Gamma,\Delta \vdash q$}
        \RightLabel{$(\operatorname{ctr/ex})$}
        \doubleLine
        \UnaryInfC{$\Gamma, \Delta \vdash q$}
        \DisplayProof
\end{center}
where $r \Gamma$ denotes the concatenation of $r$ copies of the sequence $\Gamma$. Note that this proof contains no active contractions for the final cut variable $x:p = x_1:p$.

The variable $x_i$ is introduced inside $\pi_2''$ by an instance $(r_i)$ of a deduction rule which is $(\operatorname{weak}), (L \imp)$ or $(\operatorname{ax})$. Possibly using \eqref{co_ctr_comm_alt} to rearrange the ordering, we may assume that there is an integer $1 \le m \le l$ such that for all $1 \le i \le m$ the variable $x_i$ is introduced by either $(\operatorname{weak})$ or $(L \imp)$ and for $i > m$ it is introduced by $(\operatorname{ax})$.\footnote{Note that the ordering on the $x_i$ has no meaning, and we do not require $(r_i)$ to be in any sense ``above'' or ``below'' $(r_j)$ if $i < j$.} First we deal with the cases $1 \le i \le m$. Using commuting conversions $(r_i)$ may be commuted downwards in $\pi_2''$ past the rule $(r)$. Further by \eqref{cut:log_vs_struc_np}, \eqref{cut:R_vs_L_nonp} the rule $(r_i)$ may be commuted past not only the $(\operatorname{cut})$ directly below $(r)$ but every cut down to the one that is actually against the variable $x_i:p$ introduced by $(r_i)$. Here we use in an essential way that the active contractions have been accounted for in the the previous step.

At the end of this process we see that $\pi$ is equivalent to a preproof, roughly of the same shape as above, with $l$ copies of $\pi_1$ being cut against the ``trunk'' of the tree at the ``crown'' of which is a preproof $\pi_2'''$ of $x_{m+1},\ldots,x_l, \Delta' \vdash q$ derived from $\pi_2''$. The first $m$ of these copies of $\pi_1$ are cut against variables $x_1,\ldots,x_m$ introduced immediately before the cut, and the final $l-m$ copies of $\pi_1$ are cut against a proof of $x_{m+1},\ldots,x_i,\Delta' \vdash q$ for some $m+1 \le i \le l$. These final $l-m$ cuts may be eliminated using Lemma \ref{lemma:ax_actualwork} (this does not use either the inner or outer inductive hypothesis) noting that in the notation of that lemma, any variable in $\Delta$ introduced by an $(\operatorname{ax})$ in $\pi_2$ is still introduced by an $(\operatorname{ax})$ in the cut-free proof produced which is equivalent to $\pi$, so that the lemma may be applied multiple times. The remaining cuts on $x_1,\ldots,x_m$ may then be sequentially eliminated using either \eqref{cut:log_vs_weak} or \eqref{cut:R_vs_L} and the outer inductive hypothesis. The end result is a cut-free preproof equivalent to $\pi$.
\end{proof}

\begin{remark} Note that the proof of cut-elimination (including the proof of the existence of contraction normal from in Lemma \ref{lemma:contract_normal_form}) only uses $\tau$-equivalence, $co$-equivalence, commuting conversions and the cut-elimination transformations \eqref{cut:ax_left}-\eqref{cut:R_vs_L_nonp2} of Definition \ref{cutreduction} (note that all of these cut-elimination transformations are used). So the cut-elimination theorem holds without $\lambda$-equivalence or $\eta$-equivalence.
\end{remark}

In the rest of the section we develop the notion of a contraction normal form, which was used in the proof of cut-elimination. To avoid conflicting with the notation for the preproof $\pi_2$ there we denote the subject of following by $\varphi$.

\begin{definition} Let $\varphi$ be a cut-free preproof of $x:p, \Delta \vdash q$. The \emph{contraction tree} of $(\varphi, x:p)$ is the labelled oriented graph whose vertices are the final occurrence of $x:p$ together with all weak ancestors of $x:p$ in $\varphi$, where we draw an edge $y:p \rightarrow z:p$ if $z:p$ is an immediate weak ancestor of $y:p$ in $\varphi$. We label each edge with the corresponding deduction rule. The final occurrence of $x:p$ is the root of the tree.

A \emph{slack vertex} of $(\varphi, x:p)$ is a trivalent vertex $z:p$ in the contraction tree where the incoming edge $y:p \rightarrow z:p$ is labelled by any rule other than a contraction active with respect to the final occurrence of $x:p$. The \emph{slack} of $(\varphi, x:p)$ is the number of slack vertices. We say $(\varphi, x:p)$ is in \emph{contraction normal form} if it has a slack of zero.
\end{definition}

\begin{example}\label{example:slack_2} The contraction tree of the pair $\underline{2}, y: p \imp p$ of Example \ref{example:weak_ancestor_2} is
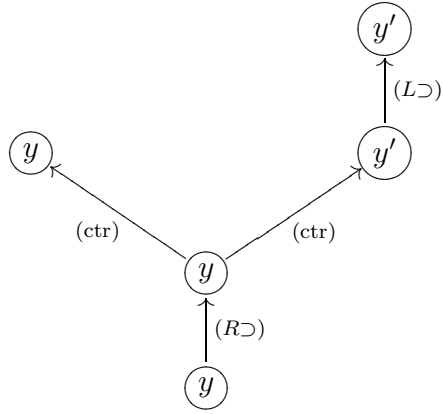
\[
\xymatrix@C+2pc{
& & *++[o][F]{y'}\\
*++[o][F]{y} & & *++[o][F]{y'}\ar[u]_-{(L \imp)}\\
& *++[o][F]{y}\ar[ul]^-{(\operatorname{ctr})}\ar[ur]_-{(\operatorname{ctr})}\\
& *++[o][F]{y}\ar[u]_-{(R \imp)}
}
\]
The pair $\underline{2}, y: p \imp p$ therefore has a slack of $1$. Using \eqref{comm_R_ctr} we see that $\underline{2}$ is equivalent under $\sim_p$ to the following proof $\underline{2}'$ in which the weak ancestors of the final $y: p \imp p$ are again marked in blue:
\begin{prooftree}
        \AxiomC{}
        \RightLabel{$({\operatorname{ax}})$}
        \UnaryInfC{$x:p \vdash p$}
        \AxiomC{}
        \RightLabel{$({\operatorname{ax}})$}
        \UnaryInfC{$x:p \vdash p$}
        \AxiomC{}
        \RightLabel{$({\operatorname{ax}})$}
        \UnaryInfC{$x:p \vdash p$}
        \RightLabel{$(L \imp)$}
        \BinaryInfC{$\textcolor{blue}{y': p \imp p}, x:p \vdash p$}
        \RightLabel{$(L \imp)$}
        \BinaryInfC{$\textcolor{blue}{y: p \imp p}, \textcolor{blue}{y': p \imp p}, x:p \vdash p$}
        \RightLabel{$(R \imp)$}
        \UnaryInfC{$\textcolor{blue}{y: p \imp p}, \textcolor{blue}{y': p \imp p} \vdash p \imp p$}
        \RightLabel{$(\operatorname{ctr})$}
        \UnaryInfC{$\textcolor{blue}{y: p \imp p} \vdash p \imp p$}
\end{prooftree}
The contraction tree of $(\underline{2}', y: p \imp p)$ is
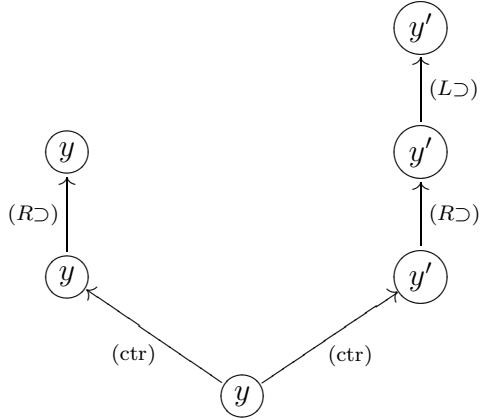
\[
\xymatrix@C+2pc{
& & *++[o][F]{y'}\\
*++[o][F]{y} & & *++[o][F]{y'}\ar[u]_-{(L \imp)}\\
*++[o][F]{y}\ar[u]^-{(R \imp)} & & *++[o][F]{y'}\ar[u]_-{(R \imp)}\\
& *++[o][F]{y}\ar[ul]^-{(\operatorname{ctr})}\ar[ur]_-{(\operatorname{ctr})}
}
\]
which has slack zero, so $(\underline{2}', y: p \imp p)$ is in contraction normal form.
\end{example}

\begin{lemma}\label{lemma:contract_normal_form} Any cut-free preproof $\varphi$ of $x:p, \Delta \vdash q$ is equivalent under $\sim_p$ to a cut-free preproof in contraction normal form.
\end{lemma}
\begin{proof}
Consider a slack vertex $y:p$ in $(\varphi, x:p)$ with incoming edge labelled by the rule $(r)$. If $(r)$ is $(\operatorname{ex})$ using \eqref{tau_ctr_ex}, \eqref{comm_ex_ctr},\eqref{comm_ex_ctr2}, or $(r)$ is $(\operatorname{weak})$ using \eqref{comm_ctr_weak}, or $(r)$ is $(R \imp)$ using \eqref{comm_R_ctr}, or $(r)$ is $(L \imp)$ using \eqref{comm_L_ctr}, \eqref{comm_L_ctr2}, or $(r)$ is an $(\operatorname{ctr})$ which is not active for the final occurrence of $x:p$ by \eqref{comm_ctr_ctr2}, we have an equivalence of preproofs $\varphi \sim_p \varphi'$ under which the contraction tree is changed around $y:p$ as follows:
    \begin{center}
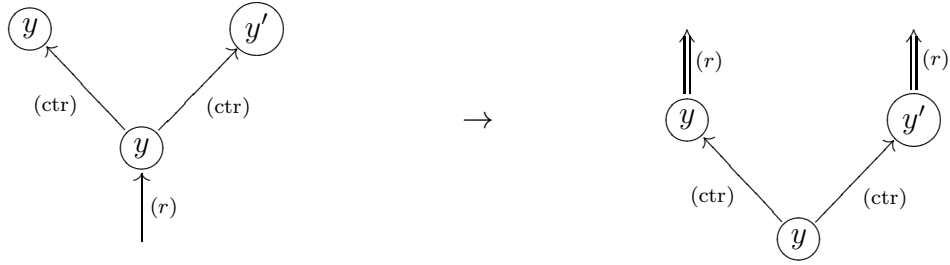
    \begin{tabular}{ >{\centering}m{6cm} >{\centering}m{2cm} >{\centering}m{6cm}}
\[
\xymatrix{
*++[o][F]{y} & & *++[o][F]{y'}\\
& *++[o][F]{y}\ar[ul]^-{(\operatorname{ctr})}\ar[ur]_-{(\operatorname{ctr})}\\
& {}\ar[u]_-{(r)}
}
\]
        & $\rightarrow$ &
\[
\xymatrix{
{} & & {}\\
*++[o][F]{y}\ar@{=>}[u]_-{(r)} & & *++[o][F]{y'}\ar@{=>}[u]_-{(r)}\\
& *++[o][F]{y}\ar[ul]^-{(\operatorname{ctr})}\ar[ur]_-{(\operatorname{ctr})}\\
}
\]
    \end{tabular}
    \end{center}
The doubled arrows reflect the fact that in the case \eqref{tau_ctr_ex} there are two edges labelled $(\operatorname{r})$ rather than one. Note that if $(r)$ is $(L \imp)$ then the contraction cannot be as in \eqref{lambda_L_L_ctr} because this contraction cannot be active with respect to the final occurrence of $x:p$.\footnote{Note that the transformation from $\varphi$ to $\varphi'$ may act nontrivially on other parts of the contraction tree: for instance if $\varphi$ at $y:p$ is as in \eqref{comm_L_ctr} then there are two occurrences of $\Delta$ (which, if $\Delta$ contains a weak ancestor of $x:p$ will contribute two vertices in the contraction tree) whereas $\varphi'$ contains three occurrences of $\Delta$.}

Let $\mathbb{S} = \mathbb{S}(\varphi, x:p)$ be the set of vertices in the contraction tree with two outgoing edges, or what is the same, the set of active contractions for $x:p$ in $\varphi$. We define the \emph{depth} $d(y:p)$ of such a vertex to be the number of rules $(r)$ on the unique path from that vertex to the root which are not active contractions for the final occurrence of $x:p$. The proof of the lemma is by induction on the integer
\[
n(\varphi,x:p) = \sum_{y:p \in \mathbb{S}} d(y:p)\,.
\]
In the base case $n = 0$ the pair $(\varphi, x:p)$ is already in contraction normal form and there is nothing to prove. Given $(\varphi,x:p)$ with $n(\varphi,x:p) > 0$ there exists a slack vertex $y:p$ in $\varphi$ and we let $\varphi \sim_p \varphi'$ be the corresponding transformation as constructed above. There is a canonical bijection
\[
f: \mathbb{S}(\varphi, x:p) \lto \mathbb{S}(\varphi', x:p)
\]
and by inspection of the proof transformations $d( f(z:p) ) \le d(z:p)$ for every $z:p$ in $\mathbb{S}(\varphi,x:p)$. By construction $d( f(y:p) ) < d(y:p)$ so that $n(\varphi',x:p) < n(\varphi,x:p)$ and the claim follows by the inductive hypothesis.
\end{proof}

\begin{remark}\label{remark:coassoc_tree} In a cut-free preproof in contraction normal form, all the active contractions appear the bottom of the tree but the pattern of these contractions is arbitrary. In the proof of Proposition \ref{prop:actualwork}, specifically in \eqref{eq:cut_tree_actualwork}, we assume that the contractions may be organised such that only the rightmost two ancestors in the list are ever contracted; this is possible by \eqref{co_ctr_assoc} and \eqref{co_ctr_comm_alt}.
\end{remark}

\subsection{The category of proofs}
\label{categoryofproofs}

Under the Brouwer-Heyting-Kolmogorov interpretation of intuitionistic logic \cite{troelstra} a proof of $\Gamma \vdash p \imp q$ is viewed as a transformation from proofs of $p$ to proofs of $q$. Thus it is natural to view such proofs as \emph{morphisms} from $p$ to $q$ in a category where objects are formulas, morphisms are proofs and composition is $(\operatorname{cut})$. Throughout this section $\Gamma$ is a sequence of variables. Let $\Psi_{\imp}$ denote the set of formulas.

\begin{definition} For a formula $p$ we denote by $\Sigma^\Gamma_p$ the set of preproofs of $\Gamma \vdash p$.
\end{definition}


\begin{definition}\label{defn:convert_closed_to_open} Given a preproof $\pi$ of $\Gamma \vdash p \imp q$ and $x:p$ let $\pi\{x\}$ denote
\begin{center}
\begin{tabular}{ >{\centering}m{10cm} >{\centering}m{0.5cm}}
\AxiomC{$\pi$}
\noLine
\UnaryInfC{$\vdots$}
\noLine
\UnaryInfC{$\Gamma \vdash p \imp q$}
\AxiomC{}
\RightLabel{$({\operatorname{ax}})$}
\UnaryInfC{$x : p \vdash p$}
\AxiomC{}
\RightLabel{$({\operatorname{ax}})$}
\UnaryInfC{$y: q \vdash q$}
\RightLabel{$(L\imp)$}
\BinaryInfC{$z: p \imp q, x:p \vdash q$}
\RightLabel{$(\operatorname{cut})$}
\BinaryInfC{$\Gamma,x:p \vdash q$}
\DisplayProof
&
\tagarray{\label{move_to_left}}
\end{tabular}
\end{center}
This preproof is independent up to $\sim_p$ of $y:p, z:q$ by \eqref{alpha_L} and \eqref{alpha_cut}.
\end{definition}

\begin{lemma}\label{lemma:semi_normal_form} Any preproof $\pi$ of $\Gamma \vdash p \imp q$ is equivalent under $\sim_p$ to
\begin{center}
\begin{tabular}{ >{\centering}m{10cm} >{\centering}m{0.5cm}}
\AxiomC{$\pi\{x\}$}
\noLine
\UnaryInfC{$\vdots$}
\noLine
\UnaryInfC{$\Gamma, x:p \vdash q$}
\RightLabel{$(R \imp)$}
\UnaryInfC{$\Gamma \vdash p \imp q$}
\DisplayProof
&
\tagarray{\label{move_to_left2}}
\end{tabular}
\end{center}
\end{lemma}
\begin{proof}
By Theorem \ref{cutfree} we may assume $\pi$ is cut-free. Consider walking the tree underlying the preproof $\pi$ starting from the root, and taking the right hand branch at every $(L \imp)$ rule. This walk must eventually encounter a $(R \imp)$ rule. Take the first such rule and by commuting conversions \eqref{comm_R_ex},\eqref{comm_R_weak},\eqref{comm_R_ctr},\eqref{comm_R_L} move this rule down so that it is the final rule in a preproof of the form
\begin{center}
\begin{tabular}{ >{\centering}m{10cm} >{\centering}m{0.5cm}}
\AxiomC{$\psi$}
\noLine
\UnaryInfC{$\vdots$}
\noLine
\UnaryInfC{$\Gamma, x:p \vdash q$}
\RightLabel{$(R \imp)$}
\UnaryInfC{$\Gamma \vdash p \imp q$}
\DisplayProof
&
\tagarray{\label{move_to_left3}}
\end{tabular}
\end{center}
which is equivalent to $\pi$ under $\sim_p$. Now observe that $\pi\{x\}$ is equivalent under $\sim_p$ to
\begin{center}
\AxiomC{$\psi$}
\noLine
\UnaryInfC{$\vdots$}
\noLine
\UnaryInfC{$\Gamma, x:p \vdash q$}
\RightLabel{$(R \imp)$}
\UnaryInfC{$\Gamma \vdash p \imp q$}
\AxiomC{}
\RightLabel{$({\operatorname{ax}})$}
\UnaryInfC{$x : p \vdash p$}
\AxiomC{}
\RightLabel{$({\operatorname{ax}})$}
\UnaryInfC{$y: q \vdash q$}
\RightLabel{$(L\imp)$}
\BinaryInfC{$z: p \imp q, x:p \vdash q$}
\RightLabel{$(\operatorname{cut})$}
\BinaryInfC{$\Gamma,x:p \vdash q$}
\DisplayProof
\end{center}
which is by \eqref{cut:R_vs_L} equivalent to
\begin{center}
\AxiomC{}
\RightLabel{$({\operatorname{ax}})$}
\UnaryInfC{$x : p \vdash p$}
\AxiomC{$\psi$}
\noLine
\UnaryInfC{$\vdots$}
\noLine
\UnaryInfC{$\Gamma, x:p \vdash q$}
\RightLabel{$(\operatorname{cut})$}
\BinaryInfC{$x:p, \Gamma \vdash q$}
\AxiomC{}
\RightLabel{$({\operatorname{ax}})$}
\UnaryInfC{$y : q \vdash q$}
\RightLabel{$(\operatorname{cut})$}
\BinaryInfC{$x:p, \Gamma \vdash q$}
\RightLabel{$(\operatorname{ex})$}
\doubleLine
\UnaryInfC{$\Gamma,x:p \vdash q$}
\DisplayProof
\end{center}
which is equivalent by \eqref{cut:ax_left},\eqref{cut:ax_right},\eqref{tau_ex_ex} to $\psi$ which completes the proof.
\end{proof}

\begin{definition}\label{definition:s_gamma}
The category $\cat{S}_\Gamma$ has objects $\Psi_{\imp} \cup \lbrace \boldsymbol{1}\rbrace$ and morphisms
\begin{align*}
\cat{S}_\Gamma(p,q) &= \Sigma^\Gamma_{p\imp q}\,/\sim_p\\
\cat{S}_\Gamma(\boldsymbol{1},q) &= \Sigma^\Gamma_{q}\,/\sim_p
\end{align*}
with special cases $\cat{S}_\Gamma(p,\boldsymbol{1}) = \lbrace \ast \rbrace$, $\cat{S}_\Gamma(\boldsymbol{1},\boldsymbol{1}) = \lbrace \ast \rbrace$. For formulas $p,q,r$ composition
\[\cat{S}_\Gamma(q,r) \times \cat{S}_\Gamma(p,q) \to \cat{S}_\Gamma(p,r)\]
sends the pair $(\psi, \pi)$ to the proof $\psi \circ \pi$ given by
\begin{center}
\begin{tabular}{ >{\centering}m{10cm} >{\centering}m{0.5cm}}
    \AxiomC{$\pi\{x\}$}
    \noLine
    \UnaryInfC{$\vdots$}
    \UnaryInfC{$\Gamma,x:p \vdash q$}
    \AxiomC{$\psi\{y\}$}
    \noLine
    \UnaryInfC{$\vdots$}
    \UnaryInfC{$\Gamma,y: q \vdash r$}
    \RightLabel{$(\operatorname{cut})$}
    \BinaryInfC{$\Gamma,x:p,\Gamma \vdash r$}
    \doubleLine
    \RightLabel{$(\operatorname{ex}/\operatorname{ctr})$}
    \UnaryInfC{$\Gamma, x:p \vdash r$}
    \RightLabel{$(R\imp)$}
    \UnaryInfC{$\Gamma\vdash p \imp r$}
    \DisplayProof
    &
    \tagarray{\label{composite_proofs_1}}
\end{tabular}
\end{center}
The special cases of the composition map are defined as follows: for formulas $p,q$ the map $\cat{S}_\Gamma(p,q) \times \cat{S}_\Gamma(\boldsymbol{1},p) \to \cat{S}_\Gamma(\boldsymbol{1},q)$ sends $(\psi, \pi)$ to
\begin{center}
\begin{tabular}{ >{\centering}m{10cm} >{\centering}m{0.5cm}}
\AxiomC{$\pi$}
\noLine
\UnaryInfC{$\vdots$}
\UnaryInfC{$\Gamma \vdash p$}
\AxiomC{$\psi\{x\}$}
\noLine
\UnaryInfC{$\vdots$}
\UnaryInfC{$\Gamma,x:p \vdash q$}
\RightLabel{$(\operatorname{cut})$}
\BinaryInfC{$\Gamma,\Gamma\vdash q$}
\doubleLine
\RightLabel{$(\operatorname{ex}/\operatorname{ctr})$}
\UnaryInfC{$\Gamma\vdash q$}
\DisplayProof
&
    \tagarray{\label{composite_proofs_2}}
\end{tabular}
\end{center}
and the map $\cat{S}_\Gamma(\boldsymbol{1},q) \times \cat{S}_\Gamma(p,\boldsymbol{1}) \to \cat{S}_\Gamma(p,q)$ sends $(\pi, \ast)$ to
\begin{center}
\begin{tabular}{ >{\centering}m{10cm} >{\centering}m{0.5cm}}
    \AxiomC{$\pi$}
    \noLine
    \UnaryInfC{$\vdots$}
    \UnaryInfC{$\Gamma \vdash q$}
    \RightLabel{$(\operatorname{weak})$}
    \UnaryInfC{$\Gamma, x:p \vdash q$}
    \RightLabel{$(R\imp)$}
    \UnaryInfC{$\Gamma \vdash p \imp q$}
    \DisplayProof
    &
    \tagarray{\label{composite_proofs_3}}
\end{tabular}
\end{center}
and $\cat{S}_\Gamma(\boldsymbol{1},p) \times \cat{S}_\Gamma(\boldsymbol{1},\boldsymbol{1}) \to \cat{S}_\Gamma(\boldsymbol{1},p)$ is the projection.
\end{definition}

Note that the composition $\psi \circ \pi$ depends as a preproof on the choices of intermediate variables $x:p, y:q$ but is independent of these choices by \eqref{alpha_R} and \eqref{alpha_cut}. The identity morphism $1_p: p \lto p$ in $\cat{S}_\Gamma$ for a formula $p$ is the proof
\begin{center}
\AxiomC{}
\RightLabel{$(\operatorname{ax})$}
\UnaryInfC{$x:p \vdash p$}
\RightLabel{$(\operatorname{weak})$}
\doubleLine
\UnaryInfC{$\Gamma, x:p \vdash p$}
\RightLabel{$(R \imp)$}
\UnaryInfC{$\Gamma \vdash p \imp p$}
\DisplayProof
\end{center}

\section{Lambda calculus}\label{section:lambda_calc}


We define a category $\cat{L}$ whose objects are the types of simply-typed lambda calculus, and whose morphisms are the terms of that calculus. The natural desiderata for such a category are that the fundamental algebraic structure of lambda calculus, function application and lambda abstraction, should be realised by categorical algebra. 

We assume familiarity with simply-typed lambda calculus; some details are recalled in Appendix \ref{section:intro_lambda}. Following Church's original presentation our lambda calculus only contains function types and $\Phi_{\typearrow}$ denotes the set of simple types. We write $\Lambda_\sigma$ for the set of $\alpha$-equivalence classes of lambda terms of type $\sigma$. 

\begin{definition}[(Category of lambda terms)]\label{definition:lambda_cat} The category $\cat{L}$ has objects
\[
\operatorname{ob}(\cat{L}) = \Phi_{\typearrow} \cup \{ \bold{1} \}
\]
and morphisms given for types $\sigma, \tau \in \Phi_{\typearrow}$ by
\begin{align*}
\cat{L}(\sigma, \tau) &= \Lambda_{\sigma \typearrow \tau}/\!=_{\beta\eta}\,\\
\cat{L}(\bold{1}, \sigma) &= \Lambda_{\sigma}/\!=_{\beta\eta}\,\\
\cat{L}(\sigma, \bold{1}) &= \{ \star \}\,\\
\cat{L}(\bold{1},\bold{1}) &= \{ \star \}\,,
\end{align*}
where $\star$ is a new symbol. For $\sigma, \tau, \rho \in \Phi_{\typearrow}$ the composition rule is the function
\begin{gather*}
\cat{L}(\tau, \rho) \times \cat{L}(\sigma, \tau) \lto \cat{L}(\sigma, \rho)\,\\
(N,M) \longmapsto \lambda x^\sigma \ldot (N \, (M \, x))\,,
\end{gather*}
where $x \notin \FV(N) \cup \FV(M)$. We write the composite as $N \circ M$. In the remaining special cases the composite is given by the rules
\begin{align*}
\cat{L}(\tau, \rho) \times \cat{L}(\bold{1}, \tau) \lto \cat{L}(\bold{1}, \rho)\,, \qquad & N \circ M = (N \, M)\,,\\
\cat{L}(\bold{1}, \rho) \times \cat{L}(\bold{1}, \bold{1}) \lto \cat{L}(\bold{1}, \rho)\,, \qquad & N \circ \star = N\,,\\
\cat{L}(\bold{1}, \rho) \times \cat{L}(\sigma, \bold{1}) \lto \cat{L}(\sigma, \rho)\,, \qquad & N \circ \star = \lambda t^\sigma \ldot N\,,
\end{align*}
where in the final rule $t \notin \FV(N)$. All other cases are trivial. Note that these functions, which have been described using a choice of representatives from a $\beta\eta$-equivalence class, are nonetheless well-defined.
\end{definition}

For terms $M,N$ the expression $M = N$ always means equality of terms (that is, up to $\alpha$-equivalence) and we write $M =_{\beta\eta}$ if we want to indicate equality up to $\beta\eta$-equivalence (for example as morphisms in the category $\cat{L}$). Since the free variable set of a lambda term is not invariant under $\beta$-reduction, some care is necessary in defining the category $\cat{L}_Q$ below. Let $\twoheadrightarrow_\beta$ denote multi-step $\beta$-reduction \cite[Definition 1.3.3]{sorensen}.

\begin{lemma}\label{lemma:beta_reduce_FV} If $M \twoheadrightarrow_\beta N$ then $\FV(N) \subseteq \FV(M)$.
\end{lemma}

\begin{definition} Given a term $M$ we define
\[
\FV_\beta(M) = \bigcap_{N =_{\beta} M} \FV(N)
\]
where the intersection is over all terms $N$ which are $\beta$-equivalent to $M$. 
\end{definition}

Clearly if $M =_{\beta} M'$ then $\FV_\beta(M) = \FV_\beta(M')$.

\begin{lemma} Given terms $M: \sigma \typearrow \rho$ and $N : \sigma$ we have
\[
\FV_\beta( (M N) ) \subseteq \FV_\beta(M) \cup \FV_\beta(N)\,.
\]
\end{lemma}
\begin{proof}
We may assume $M,N$ $\beta$-normal, in which case there is a chain of $\beta$-reductions $(M N) \twoheadrightarrow_\beta \widehat{(M N)}$ whence we are done by Lemma \ref{lemma:beta_reduce_FV}.
\end{proof}


By the same argument

\begin{lemma}\label{lemma:comp} Given $M: \sigma \typearrow \rho$ and $N: \tau \typearrow \sigma$ we have
\be\label{eq:weak_func}
\FV_\beta( M \circ N ) \subseteq \FV_\beta(M) \cup \FV_\beta(N)\,.
\ee
\end{lemma}

Given a set $Q$ of variables we write $\Lambda^Q_\sigma$ for the set of lambda terms $M$ of type $\sigma$ with $\operatorname{FV}(M) \subseteq Q$. Let $=_{\beta\eta}$ denote the induced relation on this subset of $\Lambda_\sigma$.

\begin{lemma}\label{lemma:twolambdaQs} For any type $\sigma$ and set $Q$ of variables the image of the injective map
\be\label{eq:twolambdaQs}
\Lambda^Q_p/=_{\beta\eta} \lto \Lambda_p/=_{\beta\eta}
\ee
is the set of equivalence classes of terms $M$ with $\operatorname{FV}_\beta(M) \subseteq Q$.
\end{lemma}
\begin{proof}
Since the simply-typed lambda calculus is strongly normalising \cite[Theorem 3.5.1]{sorensen} and confluent \cite[Theorem 3.6.3]{sorensen} there is a unique normal form $\widehat{M}$ in the $\beta$-equivalence class of $M$, and $\FV_\beta(M) = \FV(\widehat{M})$. Hence if $\operatorname{FV}_\beta(M) \subseteq Q$ then $\operatorname{FV}(\widehat{M}) \subseteq Q$ and so $M$ is in the image of \eqref{eq:twolambdaQs}.
\end{proof}



\begin{definition}\label{definition:lambda_Q} For a set of variables $Q$ we define a subcategory $\cat{L}_Q \subseteq \cat{L}$ by
\[
\operatorname{ob}(\cat{L}_Q) = \operatorname{ob}(\cat{L}) = \Phi_{\typearrow} \cup \{ \bold{1} \}
\]
and for types $\sigma, \rho$
\begin{align*}
\cat{L}_Q(\sigma, \rho) &= \{ M \in \cat{L}(\sigma, \rho) \l \FV_\beta(M) \subseteq Q \}\,,\\
\cat{L}_Q(\bold{1}, \sigma) &= \{ M \in \cat{L}(\bold{1}, \sigma) \l \FV_\beta(M) \subseteq Q \}\,,\\
\cat{L}_Q(\sigma, \bold{1} ) &= \cat{L}(\sigma, \bold{1}) = \{ \star \}\,,\\
\cat{L}_Q(\bold{1}, \bold{1}) &= \cat{L}(\bold{1}, \bold{1}) = \{ \star \}\,.
\end{align*}
Note that the last two lines have the same form using the convention that $\FV_\beta(\star) = \emptyset$.
\end{definition}

The fact that $\cat{L}_Q$ is a subcategory follows from Lemma \ref{lemma:comp}. 

\begin{remark}
We sketch how function application and lambda abstraction in the simply-typed lambda calculus are realised as natural categorical algebra in $\cat{L}$. Function application is composition, and lambda abstraction is given by a universal property involving factorisation of morphisms in $\cat{L}$ through morphisms in $\cat{L}_Q$. 

To explain, let $M \in \cat{L}(\sigma, \rho)$ be a morphism and $q: \tau$ a variable. We can consider the set of all commutative diagrams in $\cat{L}$ of the form
\be
\xymatrix@C+1pc@R+1pc{
\sigma \ar[rr]^-{M}\ar[dr]_-{f} & & \rho\\
& \kappa \ar[ur]
}
\ee
where $q \notin \FV_\beta(f)$. Taking $f = \lambda q.M$ gives the universal such factorisation.
\end{remark}

\begin{remark}
In the standard approach to associating a category to the simply-typed lambda calculus, due to Lambek and Scott \cite[\S I.11]{lambek_scott}, one extends the lambda calculus to include product types and the objects of the category $\cat{C}_{\typearrow, \times}$ are the types of the extended calculus (which includes an empty product $\bold{1}$) and the set $\cat{C}_{\typearrow, \times}(\sigma, \rho)$ is a set of equivalence classes of pairs $(x:\sigma, M:\rho)$ where $x$ is a variable and $M$ is a term with $\FV(M) \subseteq \{ x \}$. 

The relation to the approach given above is as follows: for $Q$ finite $\cat{L}_Q$ may be viewed as a polynomial category over $\cat{L}_{\emptyset}$ and if we write $\cat{L}_{\emptyset}^{\neq \bold{1}} \subseteq \cat{L}_{\emptyset}$ for the subcategory whose objects are types $\Phi_{\typearrow}$ there is an equivalence of categories $\cat{C}_{\typearrow} \cong \cat{L}_{\emptyset}^{\neq \bold{1}}$ where $\cat{C}_{\typearrow}$ denotes the full subcategory of $\scr{C}_{\typearrow, \times}$ whose objects are elements of the set $\Phi_{\typearrow}$.
\end{remark}

\section{Gentzen-Mints-Zucker duality}
\label{chiso}

We have defined a category of formulas and proofs $\cat{S}_\Gamma$ in intuitionistic sequent calculus (Definition \ref{definition:s_gamma}) for any finite sequence $\Gamma$ of variables, and a category of types and terms $\cat{L}_Q$ in simply-typed lambda calculus (Definition \ref{definition:lambda_Q}) for any set of variables $Q$. In logic we have associated variables to formulas and in lambda calculus to types, but identifying atomic formulas with atomic types and $\imp$ with $\rightarrow$ gives a bijection between the set $\Psi_{\imp}$ of formulas and the set $\Phi_{\rightarrow}$ of types, and we now make this identification. 

Given a sequence $\Gamma$ of variables we denote by $[\Gamma]$ the underlying \emph{set} of variables
\[
[ x_1: p_1, \ldots, x_n: p_n ] = \{ x_1: p_1, \ldots, x_n: p_n \}\,.
\]
We prove that $\cat{S}_\Gamma \cong \cat{L}_{[\Gamma]}$ if $\Gamma$ is repetition-free. To define the precise translation from proofs to lambda terms, we have to pay close attention to the variables annotating hypotheses in proofs, and this requires some preliminary comments.

Given a preproof $\pi$ of $\Gamma \vdash p$ an equivalence class $\bold{x}$ of $\approx_{str}$ (Definition \ref{defn:approx}) can be written as a sequence $\bold{x} = ( x_1, \ldots, x_n )$ of copies $x_i$ of a variable $x:p$, with $x_1$ introduced in one of $(\operatorname{ax}), (\operatorname{weak}), (L \imp)$ and $x_n$ either in the antecedent of the final sequent (labelling the root node of $\pi$) or eliminated in $(\operatorname{cut}), (\operatorname{ctr}), (R \imp)$ or $(L \imp)$. If $x_n$ is in the antecedent of the final sequent we say $\bold{x}$ is a \emph{boundary class} otherwise it is an \emph{interior class}.

\begin{definition} A preproof $\pi$ is \emph{well-labelled} if for any interior class $\bold{x}$ of occurrences of a variable $x:p$ in $\pi$, the only occurrences of $x:p$ in $\pi$ are the ones in $\bold{x}$.
\end{definition}


\begin{lemma}\label{lemma:all_well} Every preproof is equivalent under $\sim_p$ to a well-labelled preproof.
\end{lemma}
\begin{proof}
Using $\alpha$-equivalence.
\end{proof}

\begin{example}\label{example:well_2} The preproof $\underline{2}$ of Example \ref{example:weak_ancestor_2} is not well-labelled, but it is equivalent under $\sim_\alpha$ to the following well-labelled preproof:
\begin{prooftree}
        \AxiomC{}
        \RightLabel{$({\operatorname{ax}})$}
        \UnaryInfC{$\textcolor{red}{x:p} \vdash p$}
        \AxiomC{}
        \RightLabel{$({\operatorname{ax}})$}
        \UnaryInfC{$\textcolor{magenta}{x':p} \vdash p$}
        \AxiomC{}
        \RightLabel{$({\operatorname{ax}})$}
        \UnaryInfC{$\textcolor{green}{x'':p} \vdash p$}
        \RightLabel{$(L \imp)$}
        \BinaryInfC{$\textcolor{magenta}{x':p}, \textcolor{blue}{y': p \imp p} \vdash p$}
        \RightLabel{$(L \imp)$}
        \BinaryInfC{$\textcolor{red}{x:p}, \textcolor{cyan}{y: p \imp p}, \textcolor{blue}{y': p \imp p} \vdash p$}
        \RightLabel{$(\operatorname{ctr})$}
        \UnaryInfC{$\textcolor{red}{x:p}, \textcolor{cyan}{y: p \imp p} \vdash p$}
        \RightLabel{$(R \imp)$}
        \UnaryInfC{$\textcolor{cyan}{y:p \imp p} \vdash p \imp p$}
\end{prooftree}
\end{example}

In the following $\Gamma$ is a sequence of variables, possibly empty, with $Q = [\Gamma]$. Given a sequent $\Gamma \vdash p$ we let $\Sigma^\Gamma_p$ denote the set of all preproofs of that sequent, and given a finite set $Q$ of variables we denote by $\Lambda^Q_p$ the subset of $\Lambda_p$ consisting of terms with free variables contained in $Q$. Below we make use of the substitution operation of Definition \ref{defn:subst}.

\begin{definition}[(Translation)]\label{definition:translation} We let
\be
f^\Gamma_p: \Sigma^\Gamma_p \lto \Lambda^Q_p
\ee
denote the function defined on well-labelled preproofs by annotating the succedent of the deduction rules of Definition \ref{sequentcalc} with lambda terms so that each preproof may be read as a construction of a term:
        \begin{center}
        \begin{tabular}{ >{\centering}m{10cm} >{\centering}m{0.5cm}}
        \AxiomC{}
        \RightLabel{$(\operatorname{ax})$}
        \UnaryInfC{$x:p \vdash x:p$}
        \DisplayProof
        &
        \tagarray{\label{functor_ax}}
        \end{tabular}
        \end{center}

        \begin{center}
        \begin{tabular}{ >{\centering}m{10cm} >{\centering}m{0.5cm}}
        \AxiomC{$\Gamma \vdash N : p$}
        \AxiomC{$\Delta, x:p, \Theta \vdash M : q$}
        \RightLabel{$(\operatorname{cut})$}
        \BinaryInfC{$\Gamma, \Delta, \Theta \vdash M[x := N] : q$}
        \DisplayProof
        &
        \tagarray{\label{functor_cut}}
        \end{tabular}
        \end{center}
        
        \begin{center}
        \begin{tabular}{ >{\centering}m{10cm} >{\centering}m{0.5cm}}
        \AxiomC{$\Gamma, x:p,y:p, \Delta \vdash M : q$}
        \RightLabel{$(\operatorname{ctr})$}
        \UnaryInfC{$\Gamma, x:p, \Delta \vdash M[y := x] : q$}
        \DisplayProof
        &
        \tagarray{\label{functor:ctr}}
        \end{tabular}
        \end{center}
        
        \begin{center}
        \begin{tabular}{ >{\centering}m{10cm} >{\centering}m{0.5cm}}
        \AxiomC{$\Gamma, \Delta \vdash M : q$}
        \RightLabel{$(\operatorname{weak})$}
        \UnaryInfC{$\Gamma, x:p, \Delta \vdash M : q$}
        \DisplayProof
        &
        \tagarray{\label{functor:weak}}
        \end{tabular}
        \end{center}
        
        \begin{center}
        \begin{tabular}{ >{\centering}m{10cm} >{\centering}m{0.5cm}}
        \AxiomC{$\Gamma, x:p,y:q, \Delta \vdash M : r$}
        \RightLabel{$(\operatorname{ex})$}
        \UnaryInfC{$\Gamma, y:q,x:p, \Delta \vdash M : r$}
        \DisplayProof
        &
        \tagarray{\label{functor_ex}}
        \end{tabular}
        \end{center}
        
        \begin{center}
        \begin{tabular}{ >{\centering}m{10cm} >{\centering}m{0.5cm}}
        \AxiomC{$\Gamma, x:p, \Delta \vdash M : q$}
        \RightLabel{$(R \imp)$}
        \UnaryInfC{$\Gamma,\Delta \vdash \lambda x. M : p \imp q $}
        \DisplayProof
        &
        \tagarray{\label{functor_R}}
        \end{tabular}
        \end{center}
        
        \begin{center}
        \begin{tabular}{ >{\centering}m{10cm} >{\centering}m{0.5cm}}
        \AxiomC{$\Gamma \vdash N : p$}
        \AxiomC{$\Delta, x:q, \Theta \vdash M : r$}
        \RightLabel{$(L \imp)$}
        \BinaryInfC{$y: p\imp q,\Gamma,\Delta, \Theta \vdash M[x := (y\,N)] : r$}
        \DisplayProof
        &
        \tagarray{\label{functor_L}}
        \end{tabular}
        \end{center}
Given a well-labelled preproof $\pi$ annotated as above, $f^\Gamma_p(\pi)$ is the lambda term annotating the succedent on the root of $\pi$. If $\pi$ is not well-labelled, we first $\alpha$-rename as necessary using \eqref{alpha_cut}, \eqref{alpha_ctr}, \eqref{alpha_R}, \eqref{alpha_L} any interior equivalence class under $\approx_{str}$ to obtain a preproof $\pi'$ which is well-labelled and define $f^\Gamma_p(\pi) := f^\Gamma_p(\pi')$. This term is independent of choices made during $\alpha$-renaming. We refer to $f^\Gamma_p(\pi)$ as the \emph{translation} of $\pi$.
\end{definition}

\begin{remark}\label{remark:translation_who} The function from sequent calculus proofs to derivations in natural deduction is implicit in Gentzen \cite{gentzen} as the concatenation of a translation from sequent calculus LJ to the Hilbert-style system LHJ \cite[\S V.5]{gentzen} and a translation from LHJ to natural deduction NJ \cite[\S V.3]{gentzen}. The map from LJ to NJ is also discussed very briefly by Prawitz \cite[p.90-91]{prawitz}. The translation to natural deduction appears explicitly in Zucker \cite{zucker} and the translation to lambda terms in Mints \cite{mints}. For a textbook treatment of the former see \cite[\S 3.3.1]{troelstra_sch} and for the latter \cite[\S 7.4]{sorensen}.
\end{remark}

\begin{remark} The constraint that $\pi$ is well-labelled is necessary for Definition \ref{definition:translation} to capture the intended translation from proofs to lambda terms. For example if there are additional occurrences of $y$ in the part of the antecedent labelled $\Delta$ in the numerator of the contraction rule which are not in the same $\approx_{str}$-equivalence class as the occurrence $y$ being contracted, then the substitution $M[ y := x ]$ will rewrite these other occurrences to $x$, which is not what we intend.
\end{remark}


\begin{definition}\label{definition:simo} We define $\sim_o$ to be the smallest compatible equivalence relation on preproofs containing the union of $\alpha$-equivalence, $\tau$-equivalence, commuting conversions, $co$-equivalence and $\lambda$-equivalence.
\end{definition}

\begin{lemma}\label{lemma:F_welldefined} Let $\pi,\pi'$ be preproofs of $\Gamma \vdash p$. Then
\begin{itemize}
\item[(i)] If $\pi \sim_o \pi'$ then $f^\Gamma_p(\pi) = f^\Gamma_p(\pi')$.
\item[(ii)] If $\pi \sim_p \pi'$ then $f^\Gamma_p(\pi) =_{\beta \eta} f^\Gamma_p(\pi')$.
\end{itemize}
\end{lemma}
\begin{proof}
By inspection of the generating relations.
\end{proof}

\begin{remark} A more precise statement than Lemma \ref{lemma:F_welldefined} is that if $\pi, \pi'$ are related by any of the generating relations for proof equivalence other than \eqref{eta}, \eqref{cut:R_vs_L} then $f^\Gamma_p(\pi) = f^\Gamma_p(\pi)$. The translation of \eqref{eta} is $\eta$-equivalence
\begin{center}
\begin{tabular}{ >{\centering}m{7cm} >{\centering}m{0.5cm} >{\centering}m{5cm} >{\centering}m{0.5cm}} 
 \AxiomC{}
 \RightLabel{$(\operatorname{ax})$}
 \UnaryInfC{$x: p \vdash \textcolor{blue}{x}:p$}
 \AxiomC{}
 \RightLabel{$(\operatorname{ax})$}
 \UnaryInfC{$y: q \vdash \textcolor{blue}{y}:q$}
 \RightLabel{$(L\imp)$}
 \BinaryInfC{$z: p \imp q, x:p \vdash \textcolor{blue}{(z \, x)}: q$}
 \RightLabel{$(R \imp)$}
 \UnaryInfC{$z: p\imp q \vdash \textcolor{blue}{\lambda x . (z \, x)}: p \imp q$}
 \DisplayProof
 &$\sim_\eta$&
 \AxiomC{}
 \RightLabel{$(\operatorname{ax})$}
 \UnaryInfC{$z: p \imp q \vdash \textcolor{blue}{z}: p \imp q$}
 \DisplayProof
 &
 \tagarray{\label{eta_annotated}}
 \end{tabular}
\end{center}
and the translation of \eqref{cut:R_vs_L} is $\beta$-reduction.
\end{remark}

\begin{example}\label{example:translation_2} The lambda term associated to the well-labelled $\underline{2}$ from Example \ref{example:well_2} is
\begin{prooftree}
        \AxiomC{}
        \RightLabel{$({\operatorname{ax}})$}
        \UnaryInfC{$x:p \vdash x:p$}
        \AxiomC{}
        \RightLabel{$({\operatorname{ax}})$}
        \UnaryInfC{$x':p \vdash x':p$}
        \AxiomC{}
        \RightLabel{$({\operatorname{ax}})$}
        \UnaryInfC{$x'':p \vdash \textcolor{blue}{x''}:p$}
        \RightLabel{$(L \imp)$}
        \BinaryInfC{$y': p \imp p, x':p \vdash \textcolor{blue}{(y'\, x')} : p$}
        \RightLabel{$(L \imp)$}
        \BinaryInfC{$y: p \imp p, y': p \imp p, x:p \vdash \textcolor{blue}{(y' \, (y \, x))} : p$}
        \RightLabel{$(\operatorname{ctr})$}
        \UnaryInfC{$y: p \imp p , x:p\vdash \textcolor{blue}{(y \, (y \, x))}:p$}
        \RightLabel{$(R \imp)$}
        \UnaryInfC{$y:p \imp p \vdash \textcolor{blue}{\lambda x . (y \, (y \, x))} : p \imp p$}
\end{prooftree}
\end{example}

\begin{lemma}\label{curryhowardfunctor} For any sequence $\Gamma$ there is a functor $F_\Gamma: \cat{S}_\Gamma \lto \cat{L}_Q$ which is the identity on objects and which is defined on morphisms for formulas $p,q$ by
\begin{align*}
F_\Gamma(p,q) = f^\Gamma_{p \imp q} &: \cat{S}_\Gamma(p,q) \lto \cat{L}_\Gamma(p,q)\,,\\
F_\Gamma(\boldsymbol{1}, q) = f_q &: \cat{S}_\Gamma(\boldsymbol{1},q) \lto \cat{L}_\Gamma(\boldsymbol{1}, q)\,.
\end{align*}
\end{lemma}
\begin{proof}
For any formula $p$ it is clear that $F_\Gamma(1_p) = 1_p$. If $p,q,r$ are formulas we need to show that the diagram
\be\label{eq:functor_square}
\xymatrix{
\cat{S}_\Gamma(q,r) \times \cat{S}_\Gamma(p,q) \ar[r]\ar[d]_-{f^\Gamma_{q \imp r} \times f^\Gamma_{p \imp q}} & \cat{S}_\Gamma(p,r) \ar[d]^-{f^\Gamma_{p \imp r}}\\
\cat{L}_Q(q,r) \times \cat{L}_Q(p,q) \ar[r] & \cat{L}_Q(p,r)
}
\ee
commutes. Let a pair of preproofs $\psi, \pi$ of $\Gamma \vdash q \imp r$ and $\Gamma \vdash p \imp q$ respectively be given. We may assume by Lemma \ref{lemma:semi_normal_form} that $\psi, \pi$ are obtained respectively by $(R \imp)$ rules from preproofs $\psi\{y\}, \pi\{x\}$ of sequents $\Gamma, y:q \vdash r$ and $\Gamma, x:p \vdash q$. If the translations of these preproofs are $M, N$ respectively then the following annotated proof tree
\begin{center}
\begin{tabular}{ >{\centering}m{10cm} >{\centering}m{0.5cm}}
    \AxiomC{$\pi\{x\}$}
    \noLine
    \UnaryInfC{$\vdots$}
    \UnaryInfC{$\Gamma,x:p \vdash N:q$}
    \AxiomC{$\psi\{y\}$}
    \noLine
    \UnaryInfC{$\vdots$}
    \UnaryInfC{$\Gamma,y: q \vdash M:r$}
    \RightLabel{$(\operatorname{cut})$}
    \BinaryInfC{$\Gamma,x:p,\Gamma \vdash M[ y := N ]:r$}
    \doubleLine
    \RightLabel{$(\operatorname{ex}/\operatorname{ctr})$}
    \UnaryInfC{$\Gamma, x:p \vdash M[ y := N ]:r$}
    \RightLabel{$(R\imp)$}
    \UnaryInfC{$\Gamma\vdash \lambda x. M[ y := N ] : p \imp r$}
    \DisplayProof
    &
    \tagarray{\label{composite_proofs_functor}}
\end{tabular}
\end{center}
computes that $f^\Gamma_{p \imp r}( \psi \circ \pi ) = \lambda x. M[ y := N ]$. The other way around \eqref{eq:functor_square} gives
\begin{align*}
f^\Gamma_{q \imp r}( \psi ) \circ f^\Gamma_{p \imp q}( \pi ) &= (\lambda y.M) \circ (\lambda x.N)\\
&= \lambda x.( \lambda y.M ( \lambda x.N \, x) )\\
&=_{\beta} \lambda x.( \lambda y.M \, N )\\
&=_{\beta} \lambda x.M[ y := N ]
\end{align*}
as required. The remaining special cases are left to the reader.
\end{proof}


\begin{lemma}\label{lemma:cutfree_means_betanormal} If $\pi$ is cut-free then $f^\Gamma_p(\pi)$ is a $\beta$-normal form.
\end{lemma}
\begin{proof}
We may assume that $\pi$ is well-labelled. Without the cut rule the only occurrences of applications are those introduced by  \eqref{functor_L} which have the form $(y \, N)$ with $y$ a variable and so $f^\Gamma_p(\pi)$ contains no $\beta$-redexes.
\end{proof}

The translation from preproofs to lambda terms is not well-behaved if the sequent $\Gamma$ contains repetitions, as the following example shows:

\begin{example}\label{example:counter} The lambda term associated to the following well-labelled preproof $\underline{001}$ is (part of) the standard representation in lambda calculus of the binary integer $001$:
\begin{center}
\AxiomC{}
\RightLabel{$(\operatorname{ax})$}
\UnaryInfC{$x:p \vdash x:p$}
\AxiomC{}
\RightLabel{$(\operatorname{ax})$}
\UnaryInfC{$x':p \vdash x':p$}
\AxiomC{}
\RightLabel{$(\operatorname{ax})$}
\UnaryInfC{$x'':p \vdash x'':p$}
\AxiomC{}
\RightLabel{$(\operatorname{ax})$}
\UnaryInfC{$x''':p \vdash x''':p$}
\RightLabel{$(L \imp)$}
\BinaryInfC{$y'':p \imp p, x'':p \vdash (y'' \, x''):p$}
\RightLabel{$(L \imp)$}
\BinaryInfC{$y':p \imp p, x':p, y'':p \imp p \vdash (y'' \, (y' \, x')):p$}
\RightLabel{$(L \imp)$}
\BinaryInfC{$y:p \imp p, x:p, y':p \imp p, y'':p \imp p  \vdash (y'' \, (y' \, (y \, x))):p$}
\RightLabel{$(\operatorname{ex})$}
\UnaryInfC{$y:p \imp p, y':p \imp p,x:p,  y'':p \imp p  \vdash (y'' \, (y' \, (y \, x))):p$}
\RightLabel{$(\operatorname{ctr})$}
\UnaryInfC{$y:p \imp p, x:p, y'':p \imp p \vdash (y'' \, (y \, (y \, x))):p$}
\DisplayProof
\end{center}
Note that the following preproof, denoted $\underline{001}'$ is well-labelled and differs only in the variable annotations chosen to be introduced by one of the $(L \imp)$ rules (shown highlighted):
\begin{center}
\AxiomC{}
\RightLabel{$(\operatorname{ax})$}
\UnaryInfC{$x:p \vdash x:p$}
\AxiomC{}
\RightLabel{$(\operatorname{ax})$}
\UnaryInfC{$x':p \vdash x':p$}
\AxiomC{}
\RightLabel{$(\operatorname{ax})$}
\UnaryInfC{$x'':p \vdash x'':p$}
\AxiomC{}
\RightLabel{$(\operatorname{ax})$}
\UnaryInfC{$x''':p \vdash x''':p$}
\RightLabel{$(L \imp)$}
\BinaryInfC{$\textcolor{blue}{y}:p \imp p, x'':p \vdash (y \, x''):p$}
\RightLabel{$(L \imp)$}
\BinaryInfC{$y':p \imp p, x':p, \textcolor{blue}{y}:p \imp p \vdash (y \, (y' \, x')):p$}
\RightLabel{$(L \imp)$}
\BinaryInfC{$y:p \imp p, x:p, y':p \imp p, \textcolor{blue}{y}:p \imp p  \vdash (y \, (y' \, (y \, x))):p$}
\RightLabel{$(\operatorname{ex})$}
\UnaryInfC{$y:p \imp p, y':p \imp p,x:p,  \textcolor{blue}{y}:p \imp p  \vdash (y \, (y' \, (y \, x))):p$}
\RightLabel{$(\operatorname{ctr})$}
\UnaryInfC{$y:p \imp p, x:p, \textcolor{blue}{y}:p \imp p \vdash (y \, (y \, (y \, x))):p$}
\DisplayProof
\end{center}
From the encoding of $011$ we may in a similar way construct a well-labelled preproof $\underline{011}'$ of the same sequent as $\underline{001}'$ which is not equivalent under $\sim_p$ to $\underline{001}'$ but whose translation is the same lambda term. The clash of variables is innocuous in sequent calculus because we have enough additional information to disambiguate the role of the two variables, but this information is not present in the lambda term. This shows that the map from $\sim_p$-equivalence classes of preproofs to $\beta\eta$-equivalence classes of terms is \emph{not} injective if $\Gamma$ has multiple occurrences of variables.

Another, simpler, example of this phenomenon is the following pair of preproofs where the variable introduced by the weakening is highlighted:
\begin{center}
\begin{tabular}{ >{\centering}m{5cm} >{\centering}m{5cm} >{\centering}m{0.5cm}}
\AxiomC{}
\RightLabel{$(\operatorname{ax})$}
\UnaryInfC{$z:s \vdash z:s$}
\RightLabel{$(\operatorname{weak})$}
\UnaryInfC{$\textcolor{blue}{z:s}, z:s \vdash z:s$}
\DisplayProof
&
\AxiomC{}
\RightLabel{$(\operatorname{ax})$}
\UnaryInfC{$z:s \vdash z:s$}
\RightLabel{$(\operatorname{weak})$}
\UnaryInfC{$z:s, \textcolor{blue}{z:s} \vdash z:s$}
\DisplayProof
&
\tagarray{\label{ex:var_counter}}
\end{tabular}
\end{center}
If we consider the effect of cutting another proof against the first $z:s$ we see that they cannot be equivalent under $\sim_p$ but their translations are the same lambda term.
\end{example}

\begin{definition}
We say that $\Gamma$ is \emph{repetition-free} it for any variable $x:p$ the sequence $\Gamma$ contains at most one occurrence of $x:p$.
\end{definition}

\begin{thm}[Gentzen-Mints-Zucker duality]
\label{gentzen_mints_zucker}
If $\Gamma$ is repetition-free then the translation functor $F_\Gamma: \cat{S}_\Gamma \lto \cat{L}_Q$ is an isomorphism of categories.
\end{thm}
\begin{proof}
The functor is a bijection on objects, so we have to show that it is fully faithful and this follows immediately from Proposition \ref{prop:curry_howard_actualwork} below, using Lemma \ref{lemma:twolambdaQs}.
\end{proof}

\begin{proposition}\label{prop:curry_howard_actualwork} For any sequent $\Gamma \vdash p$ with $\Gamma$ repetition-free there is a bijection
\be\label{eq:most_important_map}
\xymatrix{
\Sigma^\Gamma_p / \sim_p \ar^-{\cong}[r] & \Lambda^Q_p / \! =_{\beta \eta}
}
\ee
induced by the function $f^\Gamma_p$.
\end{proposition}

We prove the proposition in a series of lemmas. Recall that by combining Theorem \ref{cutfree} and Lemma \ref{lemma:all_well} any preproof is equivalent under $\sim_p$ to a cut-free well-labelled preproof. Recall from Definition \ref{defn:active_contract} the notion of a contraction rule \emph{active} for a variable occurrence.

\begin{definition}\label{definition:Limp_normal_form} A preproof $\pi$ of $\Gamma \vdash p$ is called a \emph{$(L \imp)$-normal form} if 
\begin{itemize}
\item[(i)] it is cut-free and well-labelled
\item[(ii)] no contraction is active for a variable occurrence eliminated in a $(L \imp)$ rule.
\item[(iii)] no variable occurrence introduced by a $(\operatorname{weak})$ rule is equivalent under the relation $\approx_{str}$ to a variable occurrence eliminated by a $(L \imp)$ rule.
\end{itemize}
\end{definition}

\begin{lemma}\label{lemma:L_normal_form} Every preproof $\pi$ is equivalent under $\sim_p$ to a $(L \imp)$-normal form.
\end{lemma}
\begin{proof}
The proof parallels the proof of cut-elimination in Proposition \ref{prop:actualwork}. We may assume $\pi$ is cut-free and well-labelled. Call a $(L \imp)$ rule in $\pi$ \emph{defective} if either condition (ii) or (iii) of Definition \ref{definition:Limp_normal_form} fails for that particular rule. Applying the following reasoning to each defective $(L \imp)$ rule in $\pi$ from greatest to lowest height, it suffices to consider the case where $\pi$ is
\begin{center}
\begin{tabular}{ >{\centering}m{10cm} >{\centering}m{0.5cm}}
 \AxiomC{$\pi_1$}
 \noLine
 \UnaryInfC{$\vdots$}
 \noLine
 \UnaryInfC{$\Delta \vdash p$}
 \AxiomC{$\pi_2$}
 \noLine
 \UnaryInfC{$\vdots$}
 \noLine
 \UnaryInfC{$x:q, \Theta \vdash s$}
 \RightLabel{$(L \imp)$}
 \BinaryInfC{$y: p \imp q,\Delta, \Theta \vdash s$}
 \DisplayProof
 &
 \tagarray{\label{eq:L_cut_free_0}}
 \end{tabular}
\end{center}
for $(L\imp)$-normal forms $\pi_1$ and $\pi_2$. By Lemma \ref{lemma:contract_normal_form} we can put the pair $(\pi_2, x:q)$ in contraction normal form, so that $\pi$ equivalent under $\sim_p$ to a preproof of the form
\begin{center}
 \AxiomC{$\pi_1$}
 \noLine
 \UnaryInfC{$\vdots$}
 \noLine
 \UnaryInfC{$\Delta \vdash p$}
 \AxiomC{$\pi_2'$}
 \noLine
 \UnaryInfC{$\vdots$}
 \noLine
 \UnaryInfC{$x_1,...,x_l, \Theta \vdash s$}
 \RightLabel{$(\operatorname{ctr})$}
 \UnaryInfC{$x_1,...,x_{l-1}, \Theta \vdash s$}
 \noLine
 \UnaryInfC{$\vdots$}
 \RightLabel{$(\operatorname{ctr})$}
 \UnaryInfC{$x_1, x_2, \Theta \vdash s$}
 \RightLabel{$(\operatorname{ctr})$}
 \UnaryInfC{$x_1, \Theta \vdash s$}
 \RightLabel{$(L \imp)$}
 \BinaryInfC{$y: p \imp q, \Delta, \Theta\vdash s$}
 \DisplayProof
\end{center}
where $x_1 = x$ and we drop the formula $q$ from the notation. Considering the algorithm implicit in the proof of Lemma \ref{lemma:contract_normal_form} we see that we may assume $\pi_2'$ to be a $(L \imp)$-normal form containing no active contractions for $x_1$. By zero or more applications of $\eqref{lambda_L_L_ctr}$ and commuting conversions we obtain a preproof equivalent to $\pi$ of the form
\begin{center}
\begin{tabular}{ >{\centering}m{12cm} >{\centering}m{0.5cm}}
 \AxiomC{$\pi_1$}
 \noLine
 \UnaryInfC{$\vdots$}
 \noLine
 \UnaryInfC{$\Delta \vdash p$}
 \AxiomC{$\pi_1$}
 \noLine
 \UnaryInfC{$\vdots$}
 \noLine
 \UnaryInfC{$\Delta \vdash p$}
 \AxiomC{$\pi_1$}
 \noLine
 \UnaryInfC{$\vdots$}
 \noLine
 \UnaryInfC{$\Delta \vdash p$}
 \AxiomC{$\pi_2'$}
 \noLine
 \UnaryInfC{$\vdots$}
 \noLine
 \UnaryInfC{$x_1,\ldots,x_l,\Theta \vdash s$}
 \RightLabel{$(L \imp)$}
 \BinaryInfC{$y_l,\Delta, x_1,\ldots,x_{l-1}, \Theta \vdash s$}
 \noLine
 \UnaryInfC{$\vdots$}
 \noLine
 \UnaryInfC{$y_3, \Delta, \ldots, y_l, \Delta, x_1,x_2, \Theta \vdash s$}
 \RightLabel{$(L \imp)$}
 \BinaryInfC{$y_2, \Delta, y_3, \Delta, \ldots, y_l, \Delta, x_1, \Theta \vdash s$}
 \RightLabel{$(L \imp)$}
 \BinaryInfC{$y_1, \Delta, \ldots, y_l, \Delta, \Theta \vdash s$}
 \doubleLine
 \RightLabel{$(\operatorname{ex}/\operatorname{ctr})$}
 \UnaryInfC{$y_1, \Delta, \Theta \vdash s$}
 \DisplayProof
 &
 \tagarray{\label{eq:tower_L_cuts}}
 \end{tabular}
\end{center}
where each $y_i$ is a variable of type $p \imp q$ and $y_1 = y$. Note that no contraction in this preproof is active for any of the variables eliminated in a $(L \imp)$ rule. For $1 \le i \le l$ the rule which introduces $x_i$ in $\pi_2'$ must be either $(\operatorname{ax})$ or $(\operatorname{weak})$. If $x_i$ is introduced by $(\operatorname{weak})$ then this rule can be moved, using commuting conversions, down to the corresponding $(L \imp)$ rule and then eliminated with \eqref{lambda_weak_L}. Repeating this finitely many times yields a $(L \imp)$-normal form $\pi'$ equivalent to $\pi$.
\end{proof}

\begin{lemma}\label{lemma:var_occurs} Given a preproof
\begin{center}
\begin{tabular}{ >{\centering}m{10cm} >{\centering}m{0.5cm}}
 \AxiomC{$\pi_1$}
 \noLine
 \UnaryInfC{$\vdots$}
 \noLine
 \UnaryInfC{$\Delta \vdash p$}
 \AxiomC{$\pi_2$}
 \noLine
 \UnaryInfC{$\vdots$}
 \noLine
 \UnaryInfC{$x:q, \Theta \vdash s$}
 \RightLabel{$(L \imp)$}
 \BinaryInfC{$y: p \imp q,\Delta, \Theta \vdash s$}
 \DisplayProof
 &
 \tagarray{\label{eq:L_cut_free}}
 \end{tabular}
\end{center}
which is a $(L \imp)$-normal form, the variable $x$ occurs as a free variable in $f^{x:q, \Theta}_s(\pi_2)$.
\end{lemma}
\begin{proof}
By induction on the height of $\pi_2$. The base case is where $\pi_2$ is an axiom rule, which is clear. Suppose the height of $\pi_2$ is positive. If $x:q$ is introduced in an axiom rule, then no subsequent rule can remove it. If $x:q$ is introduced by a $(L \imp)$ rule then by the inductive hypothesis that $(L \imp)$ rule eliminates a variable $z:s$ which occurred as a free variable in the translation $R$ of its right hand branch, yielding a term $R[ z := (x \, L) ]$ which contains an occurrence of $x$ as a free variable.
\end{proof}

\begin{lemma}\label{lemma:yeehoo} Suppose $\pi$ is a preproof of $\Gamma \vdash p$ which satisfies
\begin{itemize}
\item $\pi$ is a $(L \imp)$-normal form
\item $\pi$ contains no variable occurrences introduced by $(\operatorname{weak})$ which are $\approx_{str}$-equivalent to occurrences in the final sequent $\Gamma$.
\end{itemize}
Then with $M = f^\Gamma_p(\pi)$ we have $[ \Gamma ] = \operatorname{FV}(M)$.
\end{lemma}
\begin{proof}
By construction of the translation $\operatorname{FV}(M) \subseteq [ \Gamma ]$, so we need only argue that any variable $x:q$ in $\Gamma$ occurs as a free variable in $M$. This is clear if a strong ancestor of $x:q$ is introduced by $(\operatorname{ax})$. If a strong ancestor of $x:q$ is introduced by $(L \imp)$ then it follows from Lemma \ref{lemma:var_occurs}.
\end{proof}

Recall from Definition \ref{definition:simo} the relation $\sim_o$ which is weaker than $\sim_p$.

\begin{lemma}\label{lemma:preimage_vars} If $\pi$ is a preproof of $\Gamma \vdash p$ which is a $(L \imp)$-normal form and $f^\Gamma_s(\pi)$ is a variable $z:s$ then $\pi$ is equivalent under $\sim_o$ to
\begin{center}
\begin{tabular}{ >{\centering}m{10cm} >{\centering}m{0.5cm}}
 \AxiomC{}
 \RightLabel{$(\operatorname{ax})$}
 \UnaryInfC{$z:s \vdash s$}
 \doubleLine
 \RightLabel{$(\operatorname{weak})$}
 \UnaryInfC{$\Gamma \vdash s$}
 \DisplayProof
 &
 \tagarray{\label{eq:var_normal_form}}
 \end{tabular}
\end{center}
We call a preproof such as \eqref{eq:var_normal_form} a \emph{variable normal form}.
\end{lemma}
\begin{proof}
We deduce by inspection of the translation in Definition \ref{definition:translation} that the only possible rules which appear in $\pi$ are $(\operatorname{ax}), (\operatorname{ctr}), (\operatorname{weak}), (\operatorname{ex})$, or a $(L \imp)$ rule in which the eliminated variable does not occur, so that no substitution takes place. But by Lemma \ref{lemma:var_occurs} no such $(L \imp)$ rule can occur in $\pi$ from which we deduce that $\pi$ must be $(L \imp)$-free.

The preproof $\pi$ of $\Gamma \vdash s$ contains precisely one $(\operatorname{ax})$ rule and otherwise consists entirely of structural rules. Consider an occurrence of $(\operatorname{ctr})$ in $\pi$ which contracts $x:p,x':p$. A strong ancestor of either $x:p$ or $x':p$ must be introduced by $(\operatorname{weak})$ and using commuting conversions we may move this rule down the tree and eliminate it with the contraction using \eqref{co_weak_ctr} and Remark \ref{remark:old_co_ctr_comm2}. Thus $\pi$ is equivalent under $\sim_o$ to a preproof containing no contractions, and similarly one may also using commuting conversions and \eqref{tau_ex_ex},\eqref{tau_weak_ex} to eliminate all exchanges. The resulting preproof is of the desired form. Note that $\Gamma$ may contain multiple occurrences of $z:s$.
\end{proof}

\begin{lemma}\label{lemma:preimage_abstractions} If $\pi$ is a preproof of $\Gamma \vdash p$ which is cut-free and well-labelled and $f^\Gamma_{q \imp r}(\pi)$ is an abstraction $\lambda x . N$ then $\pi$ is equivalent under $\sim_o$ to a preproof of the form
\begin{center}
\begin{tabular}{ >{\centering}m{10cm} >{\centering}m{0.5cm}}
    \AxiomC{$\psi$}
    \noLine
    \UnaryInfC{$\vdots$}
    \noLine
    \UnaryInfC{$\Delta, x:q, \Delta' \vdash r$}
    \RightLabel{$(R \imp)$}
    \UnaryInfC{$\Gamma \vdash q \imp r$}
 \DisplayProof
 &
 \tagarray{\label{eq:abstract_normal_form}}
 \end{tabular}
\end{center}
where $f^{\Delta,x:q,\Delta'}_r( \psi ) = N$. We call such a preproof an \emph{abstraction normal form}. 
\end{lemma}
\begin{proof}
By the proof of Lemma \ref{lemma:semi_normal_form} we see $\pi$ is equivalent under $\sim_o$ to a preproof \eqref{eq:abstract_normal_form} using relations that do not change the translated term, so that the translation of \eqref{eq:abstract_normal_form} is still $\lambda x. N$. From this the claim follows. 
\end{proof}

\begin{definition} Given a lambda term $M$ let $\operatorname{FV}^{seq}(M)$ denote the sequence of distinct free variables in $M$ ordered by first occurrence.
\end{definition}

\begin{example} Let $M = (y'' \, (y \, (y\, x))):p$ be as in Example \ref{example:counter}. Then $\operatorname{FV}^{seq}(M)$ is the sequence $y'':p \imp p, y:p \imp p, x:p$.
\end{example}

\begin{definition}\label{definition:ladder} A \emph{ladder} is a sequence of rules of the form
 \begin{center}
\begin{tabular}{ >{\centering}m{10cm} >{\centering}m{0.5cm}}
\AxiomC{$\Gamma, x_1:p_1,\ldots,x_n:p_n,y:q,\Delta \vdash q$}
 \RightLabel{$(\operatorname{ex})$}
 \UnaryInfC{$\Gamma, x_1:p_1,\ldots,y:q,x_n:p_n,\Delta \vdash q$}
 \noLine
 \UnaryInfC{$\vdots$}
 \RightLabel{$(\operatorname{ex})$}
 \UnaryInfC{$\Gamma, y:q, x_1:p_1, \ldots, x_n:p_n, \Delta \vdash p$}
 \DisplayProof
 &
 \tagarray{\label{eq:ladder}}
 \end{tabular}
 \end{center} 
 The \emph{tail index} of a ladder is the position of $y:q$ in $\Gamma, y:q, x_1:p_1, \ldots, x_n:p_n, \Delta \vdash p$. A ladder is \emph{maximal} in a preproof $\pi$ if there is no larger ladder in $\pi$ containing it. We write $(\operatorname{lad})^i$ for a maximal ladder with tail index $i$.
\end{definition}

\begin{definition}\label{definition:derived_ctr} A \emph{derived contraction} is a sequence of rules of the form
 \begin{center}
\begin{tabular}{ >{\centering}m{10cm} >{\centering}m{0.5cm}}
\AxiomC{$\Gamma, x:p, \Delta, y:p, \Theta \vdash q$}
 \RightLabel{$(\operatorname{ex})$}
 \doubleLine
 \UnaryInfC{$\Gamma, x:p, y:p, \Delta, \Theta \vdash q$}
 \RightLabel{$(\operatorname{ctr})$}
 \UnaryInfC{$\Gamma,x:p, \Delta, \Theta \vdash p$}
 \DisplayProof
 &
 \tagarray{\label{eq:d_ctr}}
 \end{tabular}
 \end{center} 
The \emph{tail index} of a derived contraction is the position of $x:p$ in $\Gamma, x:p, \Delta, \Theta$ and the \emph{head index} is the position of $y:p$ in $\Gamma, x:p, \Delta, y:p, \Theta$. We write $(\operatorname{dctr})^{i,j}$ to stand for a derived contraction with tail index $i$ and head index $j$. A derived contraction in a preproof $\pi$ is \emph{maximal} if there is no larger derived contraction in $\pi$ containing it.
\end{definition}

\begin{definition}\label{definition:index_weak} The \emph{index} of a weakening rule, with reference to the rule schemata of Definition \ref{defn:deduction_rules}, is the position of $x:p$ in $\Gamma, x:p, \Delta$. We write $(\operatorname{weak})^i$ for a weakening rule with index $i$.
\end{definition}

\begin{definition}\label{definition:well_ordered} A preproof $\pi$ of $\Gamma \vdash p$ is \emph{well-ordered} if $\Gamma = \operatorname{FV}^{seq}(M)$ where $M = f^\Gamma_p(\pi)$.
\end{definition}

It is more difficult to give a normal form for preproofs whose translation is an application. Note that technically speaking we should require that $\Gamma$ contains no variable from the canonical series \eqref{eq:canonical_series}.

\begin{lemma}\label{lemma:preimage_app} If $\Gamma$ is repetition-free and $\pi$ is a preproof of $\Gamma \vdash p$ which is a $(L \imp)$-normal form and $f^\Gamma_p(\pi)$ is an application $(M^1 \, M^2)$ with $M^1: r \imp p$ and $M^2 : r$ then $\pi$ is equivalent under $\sim_o$ to a preproof of the form
\begin{center}
 \AxiomC{$\tau_b$}
 \noLine
 \UnaryInfC{$\vdots$}
 \noLine
 \UnaryInfC{$\Gamma_b \vdash L_b:p_b$}
 \AxiomC{$\tau_{b-1}$}
 \noLine
 \UnaryInfC{$\vdots$}
 \noLine
 \UnaryInfC{$\Gamma_{b-1} \vdash L_{b-1}:p_{b-1}$}
 \AxiomC{$\tau_1$}
 \noLine
 \UnaryInfC{$\vdots$}
 \noLine
 \UnaryInfC{$\Gamma_1 \vdash L_1:p_1$}
    \AxiomC{$\zeta$}
    \noLine
    \UnaryInfC{$\vdots$}
    \noLine
    \UnaryInfC{$\Delta \vdash R : s$}
    \AxiomC{}
    \RightLabel{$(\operatorname{ax})$}
    \UnaryInfC{$x:p \vdash x:p$}
    \RightLabel{$(L \imp)$}
    \BinaryInfC{$y, \Delta \vdash (y \, R):p$}
 \RightLabel{$(L \imp)$}
 \BinaryInfC{$y_1, \Gamma_1,\Delta \vdash ((y_1 \, L_1) \, R):p$}
 \noLine
 \UnaryInfC{$\vdots$}
 \noLine
 \UnaryInfC{$y_{b-2}, \Gamma_{b-2},\ldots,\Gamma_1,\Delta \vdash p$}
 \RightLabel{$(L \imp)$}
 \BinaryInfC{$y_{b-1}, \Gamma_{b-1},\ldots,\Gamma_1,\Delta \vdash p$}
 \RightLabel{$(L \imp)$}
 \BinaryInfC{$y_{b}, \Gamma_{b},\ldots,\Gamma_1,\Delta \vdash p$}
  \doubleLine
 \RightLabel{$(\operatorname{dctr})$}
 \UnaryInfC{$\Theta \vdash p$}
 \RightLabel{$(\operatorname{lad})$}
 \doubleLine
 \UnaryInfC{$\Theta' \vdash p$}
 \RightLabel{$(\operatorname{weak})$}
 \doubleLine
 \UnaryInfC{$\Gamma \vdash (M^1 \, M^2):p$}
 \DisplayProof
\end{center}
with the following properties:
\begin{itemize}
\item[(i)] $\zeta$ and $\tau_j$ are $(L \imp)$-normal forms for $1 \le j \le b$.
\item[(ii)] $\zeta$ and $\tau_j$ are well-ordered for $1 \le j \le b$.
\item[(iii)] No variable occurrence in $\Gamma$ has more than one weak ancestor in $\Delta$, and no variable occurrence in $\Gamma$ has more than one weak ancestor in $\Gamma_j$ for $1 \le j \le b$.
\item[(iv)] The series of derived contractions
\begin{center}
\begin{tabular}{ >{\centering}m{10cm} >{\centering}m{0.5cm}}
\AxiomC{}
\noLine
 \UnaryInfC{$y_{b}, \Gamma_{b},\ldots,\Gamma_1,\Delta \vdash p$}
 \doubleLine
 \RightLabel{$(\operatorname{dctr})$}
 \UnaryInfC{$\Theta \vdash p$}
 \DisplayProof
 &
 \tagarray{\label{eq:tiny_ctr}}
 \end{tabular}
 \end{center} 
is of the form $(\operatorname{dctr})^{a_1,b_1},(\operatorname{dctr})^{a_2,b_2},\ldots,(\operatorname{dctr})^{a_m,b_m}$ with
\[
(a_1,b_1) \le (a_2,b_2) \le \cdots \le (a_m,b_m)
\]
in the lexicographic order.
 \item[(v)] Let $\Lambda(u)$ be the sequence obtained from the numerator in \eqref{eq:tiny_ctr} by deleting from $y_b,\Gamma_b,\ldots,\Gamma_1,\Delta$ any variable occurrence which is either not of type $u$ or which is a strong ancestor of an occurrence in $\Theta$. Then for any $u$ we require that $\Lambda(u)$ is equal to an initial segment (possibly empty) of some fixed ``canonical series'' of variables
 \be\label{eq:canonical_series}
 \aleph^u_1:u,\aleph^u_2:u,\aleph^u_3:u,\ldots
 \ee
 \item[(vi)] The series of weakenings
 \begin{center}
\begin{tabular}{ >{\centering}m{10cm} >{\centering}m{0.5cm}}
\AxiomC{}
\noLine
 \UnaryInfC{$\Theta' \vdash p$}
 \doubleLine
 \RightLabel{$(\operatorname{weak})$}
 \UnaryInfC{$\Gamma \vdash p$}
 \DisplayProof
 &
 \tagarray{\label{eq:tiny_weak}}
 \end{tabular}
 \end{center}
 is of the form $(\operatorname{weak})^{d_1},(\operatorname{weak})^{d_2},\ldots,(\operatorname{weak})^{d_m}$ with $d_1 < d_2 < \cdots < d_m$.
 \item[(vii)] The series of maximal ladders
\begin{center}
\begin{tabular}{ >{\centering}m{10cm} >{\centering}m{0.5cm}}
\AxiomC{}
\noLine
 \UnaryInfC{$\Theta \vdash p$}
 \doubleLine
 \RightLabel{$(\operatorname{lad})$}
 \UnaryInfC{$\Theta' \vdash p$}
 \DisplayProof
 &
 \tagarray{\label{eq:tiny_lad}}
 \end{tabular}
 \end{center} 
is of the form $(\operatorname{lad})^{c_1},(\operatorname{lad})^{c_2},\ldots,(\operatorname{lad})^{c_n}$ with $c_1 < c_2 < \ldots < c_n$.
\end{itemize}
This representation is unique, in the following sense: any other such representation involves the same index $b$, the same sequents $\Gamma_j \vdash p_j$ and the same lambda terms $R$ and $L_j$ for $1 \le j \le b$. We call such a preproof an \emph{application normal form}.
\end{lemma}
\begin{proof}
Walk the tree underlying the preproof $\pi$ starting from the root, taking the right hand branch at every $(L \imp)$ rule, and stop at the first instance of the $(L \imp)$ rule which satisfies the following property: the preproof constituting the right hand branch has for its translation under Definition \ref{definition:translation} a variable $x:p$ and this is the variable eliminated by the $(L \imp)$ rule. Note that a $(L \imp)$ rule satisfying this property will be encountered on the  walk, because $M$ is an application. By Lemma \ref{lemma:preimage_vars} the preproof $\pi$ is therefore equivalent under $\sim_o$ to a preproof of the form
\begin{center}
\begin{tabular}{ >{\centering}m{10cm} >{\centering}m{0.5cm}}
    \AxiomC{$\zeta$}
    \noLine
    \UnaryInfC{$\vdots$}
    \noLine
    \UnaryInfC{$\Delta \vdash R : s$}
    \AxiomC{}
    \RightLabel{$(\operatorname{ax})$}
    \UnaryInfC{$x:p \vdash x:p$}
    \RightLabel{$(L \imp)$}
    \BinaryInfC{$y:s\imp p, \Delta \vdash (y \, R):p$}
    \noLine
    \UnaryInfC{$\vdots$}
    \noLine
    \UnaryInfC{$\Gamma\vdash (M^1 \, M^2):p$}
    \DisplayProof
    &
    \tagarray{\label{eq:Lbranching}}
\end{tabular}
\end{center}
where we have used \eqref{comm_weak_L} to move the weakenings in Lemma \ref{lemma:preimage_vars} below the $(L \imp)$ rule. We refer to the sequence of deduction rules connecting the root of the preproof to the displayed $(L \imp)$ rule as the \emph{porch} (note that the preproof may contain other branches that meet the displayed preproof as left hand branches at deduction rules within the porch). The porch may contain $(\operatorname{ctr})$, $(\operatorname{ex})$, $(\operatorname{weak})$ and $(L \imp)$ rules. Since $\pi$ is a $(L \imp)$-normal form none of these weakenings or contractions are relevant to the variables eliminated by $(L \imp)$ rules in the porch, so we may use commuting conversions to ensure that the $(L \imp)$ rules are all above any of these other rules.

We index the $(L \imp)$ rules on the porch, from top to bottom, by indices $\alpha$
\begin{prooftree}
\AxiomC{$\tau_\alpha$}
\noLine
\UnaryInfC{$\vdots$}
\noLine
\UnaryInfC{$\Gamma_\alpha \vdash p_\alpha$}
\AxiomC{}
\noLine
\UnaryInfC{$\vdots$}
\noLine
\UnaryInfC{$t_\alpha:q_\alpha, \Lambda_\alpha \vdash p$}
\RightLabel{$(r_\alpha)$}
\BinaryInfC{$y_\alpha: p_\alpha \imp q_\alpha, \Gamma_\alpha, \Lambda_\alpha \vdash p$}
\end{prooftree}
We now migrate $(L \imp)$ rules on the porch into $\zeta$ and the $\tau_\alpha$ branches. The variable $t_1$ either has a strong ancestor in $\zeta$ or its strong ancestor is the $y: s \imp p$ introduced by the $(L \imp)$ displayed in \eqref{eq:Lbranching}. In the former case, we can by \eqref{comm_L_L2} move the rule $(r_1)$ up into $\zeta$. In the latter case, we do nothing. Now assume that $\alpha > 1$ is given and that for all $\beta < \alpha$ the variable $t_\beta$ is introduced by one of the previous $(L \imp)$ rules on the porch. If $t_\alpha$ is introduced by one of the previous $(L \imp)$ rules on the porch then we do nothing, otherwise if $t_\alpha$ is introduced in $\zeta$ (resp. $\tau_\beta$ for $\beta < \alpha$) then we use \eqref{comm_L_L}, \eqref{comm_L_L2} to move $(r_\alpha)$ into $\zeta$ (resp. $\tau_\beta$). These applications of \eqref{comm_L_L}, \eqref{comm_L_L2} may introduce $(\operatorname{ex})$ rules onto the porch, which may either be absorbed into $(L \imp)$ rules by \eqref{tau_L_ex2} or moved to the bottom of the porch as above. Proceeding in this way through all the indices $\alpha \in \{1,\ldots,b\}$ in increasing order completes the migration.

This migration procedure shows that we may as well have assumed from the beginning that the only $(L \imp)$ rules on the porch are those in which $t_\alpha$ is introduced by the previous $(L \imp)$ rule on the porch. We now make this assumption. Using commuting conversions we can move any $(\operatorname{weak})$ rules in $\zeta$ (resp. any $\tau_j$) which introduce variables equivalent under $\approx_{str}$ to an occurrence in $\Delta$ (resp. $\Gamma_j$) down to the bottom of the porch. This shows that $\pi$ is equivalent under $\sim_o$ to a preproof of the form given in the statement of the lemma where $\zeta$ and all the $\tau_j$ are $(L \imp)$-normal forms satisfying the hypotheses of Lemma \ref{lemma:yeehoo} so that $[ \Delta ] = \operatorname{FV}(R)$ and $[ \Gamma_j ] = \operatorname{FV}(L_j)$. Using \eqref{co_ctr_assoc},\eqref{co_ctr_comm_alt} and commuting conversions we may also assume that the condition (iii) is satisfied by moving contractions up into the branches.

Now we use for the first time the hypothesis that $\Gamma$ is repetition-free. If any repetitions occurred in $\Delta$ or one of the $\Gamma_j$'s then this would have to be corrected by a contraction on the porch, which by (iii) is impossible. So $\Delta$ and all the $\Gamma_j$ are also repetition-free. Without loss of generality we may therefore assume, possibly inserting exchanges into $\zeta$ and $\tau_j$ that $\Delta = \operatorname{FV}^{seq}(R)$ and $\Gamma_j = \operatorname{FV}^{seq}(L_j)$ which is condition (ii). Condition (iv) can be arranged using \eqref{co_ctr_assoc}, \eqref{co_ctr_comm_alt}, \eqref{co_weak_ctr}. In the notation of (v) observe that for any $u$ the sequence $\Lambda(u)$ consists of variable occurrences which are eliminated in contraction rules within \eqref{eq:tiny_ctr} and so by $\alpha$-equivalence \eqref{alpha_ctr} we can rename them as we wish, provided the result is well-labelled. In particular we can rename them according to the specified rules with respect to a predetermined canonical series. This completes the proof of the existence of an application normal form and it only remains to prove uniqueness. 

Considering the translation of the normal form we see that $M = (M^1\, M^2)$ is obtained from
\be\label{eq:normal_form_app_term}
\big( (\cdots( (y_b \, L_b) \, L_{b-1} ) \, \cdots L_1) \, R\big)
\ee
by some number of contractions. Thus the index $b$ and the types $p_1,\ldots,p_b,s$ in the normal form can be read off from the term $M$. Suppose that
\be\label{eq:normal_form_app_term2}
(M^1\, M^2) = \big( (\cdots( (y_b \, L'_b) \, L'_{b-1} ) \, \cdots L'_1) \, R'\big)\,.
\ee
Now consider the sequence
\be
y_b,\operatorname{FV}^{seq}(L'_b),\ldots,\operatorname{FV}^{seq}(L'_1), \operatorname{FV}^{seq}(R')
\ee
and perform the following operation: if any variable $z:u$ is repeated in this sequence then replace all but the first occurrence by a special symbol $\bullet_u$ associated to $u$ but independent of $z$. This is done for every type $u$ and every variable of type $u$ before the next step. In the next step, for each type $u$ replace all the occurrences of $\bullet_u$ in order by variables taken from the canonical series \eqref{eq:canonical_series} for $u$. By conditions (ii),(iii),(iv),(v) the result of this operation is the sequence $y_b,\Gamma_b,\ldots,\Gamma_1,\Delta$ which is therefore determined by $M$ and is independent of any choices made above. Suppose that free variables $z_1:u_1,\ldots,z_k:u_k$ in $L'_j$ are replaced by this procedure with $\aleph^{u_1}_{t_1},\ldots,\aleph^{u_k}_{t_k}$. Then
\[
L_j = L'_j[ z_1 := \aleph^{u_1}_{t_1}, \ldots, z_k := \aleph^{u_k}_{t_k} ]
\]
and similarly for $R$, which completes the proof of the uniqueness statement.
\end{proof}

Actually the application normal form is unique in a much stronger sense, but we return to this in Section \ref{section:normal_form_seq}. We note that $b = 0$ is allowed in the definition of an application normal form, in which case there is a single $(L \imp)$ rule with left branch $\zeta$, followed by exchanges, contractions and weakenings as above.

\begin{lemma}\label{lemma:app_normal_wello} Let $\pi$ be an application normal form in which the rule series \eqref{eq:tiny_weak}, \eqref{eq:tiny_lad} are empty. Then $\pi$ is well-ordered.
\end{lemma}
\begin{proof}
Let $\pi$ be an application normal form as in the statement of Lemma \ref{lemma:preimage_app}. By Lemma \ref{lemma:yeehoo} we have $[\Gamma] = \operatorname{FV}(M)$. Suppose that $z:u, z':u'$ appear in this order within $\Gamma$ so that their strong ancestors appear in the same order within $y_b, \Gamma_b, \ldots, \Gamma_1, \Delta$. If $z = y_b$ then it is clear that the first free occurrence of $z':u'$ in $M$ appears after the first free occurrence of $z:u$. Otherwise there are two cases: in the first case $z:u, z':u'$ both appear within the same $\Gamma_j$ or both within $\Delta$, and in this case the variables appear in the same order within $\operatorname{FV}^{seq}(M)$ by condition (ii) of an application normal form. In the second case $z:u$ is in $\Gamma_j$ for some $j$ and $z':u'$ is in $\Gamma_{j'}$ for $j' < j$ or is in $\Delta$. In this case by inspection of \eqref{eq:normal_form_app_term}, \eqref{eq:normal_form_app_term2} it is clear that $z:u$ appears before $z':u'$ in $\operatorname{FV}^{seq}(M)$. 
\end{proof}

\begin{proposition}\label{prop:alpha_implies_p} If $\Gamma$ is repetition-free and $\pi_1,\pi_2$ are preproofs of $\Gamma \vdash p$ that are $(L \imp)$-normal forms then $f^\Gamma_p(\pi_1) = f^\Gamma_p(\pi_2)$ implies $\pi_1 \sim_o \pi_2$.
\end{proposition}
\begin{proof}
To be clear $f^\Gamma_p(\pi_1) = f^\Gamma_p(\pi_2)$ means equality of terms (that is, $\alpha$-equivalence of preterms). We set $M_i := f^\Gamma_p(\pi_i)$ for $i \in \{1,2\}$ so that by hypothesis $M_1 = M_2$ as terms. We proceed by induction on the length of the term $M = M_1 = M_2$. In the base case $M$ is a variable, and Lemma \ref{lemma:preimage_vars} shows that $\pi_i$ is equivalent under $\sim_o$ to
\begin{center}
 \AxiomC{}
 \RightLabel{$(\operatorname{ax})$}
 \UnaryInfC{$z:s \vdash s$}
 \doubleLine
 \RightLabel{$(\operatorname{weak})$}
 \UnaryInfC{$\Delta_i, z:s, \Theta_i \vdash s$}
 \DisplayProof
 \end{center}
for some decomposition $\Gamma = \Delta_i, z:s, \Theta_i$. Since $\Gamma$ is repetition-free there is only one occurrence of $z:s$ in $\Gamma$ so $\Delta_1 = \Delta_2, \Theta_1 = \Theta_2$ and this variable normal form is the same for both $\pi_1,\pi_2$. Hence $\pi_1 \sim_o \pi_2$ as required.

Next, suppose that $M = \lambda x . N$ is an abstraction where $p = q \imp r$. By Lemma \ref{lemma:preimage_abstractions} each $\pi_i$ is equivalent under $\sim_o$ to an abstraction normal form $\pi_i'$. Let $\psi_i$ denote the preproof obtained from $\pi'_i$ by deleting the final $(R \imp)$ rule, which we may assume eliminates a variable $x:q$ in both $\pi'_1$ and $\pi'_2$ which does not occur in $\Gamma$ and which is leftmost in the antecedent. Then
\[
f^{\Gamma,x:q}_r(\psi_1) = N = f^{\Gamma,x:q}_r(\psi_2)
\]
so by the inductive hypothesis $\psi_1 \sim_o \psi_2$ from which we deduce $\pi_1 \sim_o \pi_2$.

Finally suppose that $M$ is an application $(M^1 \, M^2):p$ with $M^1:r \imp p$ and $M^2:r$. By Lemma \ref{lemma:preimage_app} each $\pi_i$ is equivalent under $\sim_o$ to an application normal form $\pi_i'$. The proof of the lemma shows that the types $p_1,\ldots,p_b,s$, sequences $\Gamma_b,\ldots,\Gamma_1,\Delta,y_b$ and terms $L_b,\ldots,L_1,R$ may be read off from $M$ and therefore coincide in the normal forms for $\pi_1,\pi_2$. Let $\tau^i_j, \zeta^i$ denote the preproofs involved in the normal form for $\pi_i$. We deduce
\[
f^{\Gamma_j}_{p_j}(\tau^1_j) = f^{\Gamma_j}_{p_j}(\tau^2_j) \qquad 1 \le j \le b
\]
and $f^{\Delta}_s(\zeta^1) = f^{\Delta}_s(\zeta^2)$. Since $\Delta$ and $\Gamma_j$ for $1 \le j \le b$ are repetition-free it follows from the inductive hypothesis that $\tau^1_j \sim_o \tau^2_j$ for $1 \le j \le b$ and $\zeta^1 \sim_o \zeta^2$ and hence $\pi_1 \sim_o \pi_2$ which completes the proof of the inductive step.
\end{proof}

\begin{definition}\label{definition:eta_pattern} Let $\pi$ be a preproof of $\Gamma \vdash p$ which is a $(L \imp)$-normal form. A \emph{$\eta$-pattern} in $\pi$ is a configuration of rules within $\pi$ of the form
\begin{center}
\begin{tabular}{ >{\centering}m{10cm} >{\centering}m{0.5cm}}
    \AxiomC{$\zeta$}
    \noLine
    \UnaryInfC{$\vdots$}
    \noLine
    \UnaryInfC{$\Delta \vdash s$}
    \AxiomC{$\theta$}
    \noLine
    \UnaryInfC{$\vdots$}
    \noLine
    \UnaryInfC{$\Theta,z:p,\Theta' \vdash p$}
    \RightLabel{$(L \imp)$}
    \BinaryInfC{$y:s \imp p,\Delta,\Theta,\Theta' \vdash p$}
 \noLine
 \UnaryInfC{$\vdots$}
 \noLine
 \UnaryInfC{$\Gamma,x:s,\Gamma' \vdash p$}
 \RightLabel{$(R \imp)$}
 \UnaryInfC{$\Gamma,\Gamma' \vdash s \imp p$}
 \DisplayProof
 &
 \tagarray{\label{eq:eta_pattern}}
 \end{tabular}
\end{center}
with the following properties
\begin{itemize}
\item[(i)] The path from the displayed $(R \imp)$ rule to the displayed $(L \imp)$ rule takes only the right hand branch of any intermediate $(L \imp)$ rule and contains no $(R \imp)$ rules.
\item[(ii)] $f^{\Delta}_s(\zeta)$ and $f^{\Theta,z:p,\Theta'}_p(\theta)$ are both variables.
\item[(iii)] The contraction tree of the occurrence of $x:s$ eliminated by the $(R \imp)$ rule contains as leaves one occurrence introduced by an axiom in $\zeta$ and all other leaves are occurrences introduced by weakenings.
\end{itemize}
\end{definition}

\begin{example}\label{example:bad_eta} The prototypical example of an $\eta$-pattern is \eqref{eta_annotated}. However the reader should be aware that weakenings can complicate this picture:
\begin{center}
\begin{tabular}{ >{\centering}m{10cm} >{\centering}m{0.5cm}} 
 \AxiomC{}
 \RightLabel{$(\operatorname{ax})$}
 \UnaryInfC{$x': p \vdash \textcolor{blue}{x'}:p$}
 \RightLabel{$(\operatorname{weak})$}
 \UnaryInfC{$x:p,x':p \vdash \textcolor{blue}{x'}:p$}
 \RightLabel{$(\operatorname{ctr})$}
 \UnaryInfC{$x:p \vdash \textcolor{blue}{x}:p$}
 \AxiomC{}
 \RightLabel{$(\operatorname{ax})$}
 \UnaryInfC{$y: q \vdash \textcolor{blue}{y}:q$}
 \RightLabel{$(L\imp)$}
 \BinaryInfC{$z: p \imp q ,x:p\vdash \textcolor{blue}{(z \, x)}: q$}
 \RightLabel{$(R \imp)$}
 \UnaryInfC{$z: p\imp q \vdash \textcolor{blue}{\lambda x . (z \, x)}: p \imp q$}
 \DisplayProof
 &
 \tagarray{\label{eta_annotated_bad}}
 \end{tabular}
\end{center}
\end{example}

\begin{definition}\label{defn:special_l_normal} Let $\pi$ be a preproof of $\Gamma \vdash p$. We say that $\pi$ is a \emph{special $(L \imp)$-normal form} if it is a $(L \imp)$-normal form which contains no $\eta$-pattern.
\end{definition}

Recall that an \emph{$\eta$-redex} in a lambda term $M$ is a subterm of the form $\lambda x. (N \, x)$ in which $x$ does not occur as a free variable in $N$.

\begin{lemma}\label{lemma:etapattern_implies_redex} A preproof $\pi$ of $\Gamma \vdash p$ which is a $(L \imp)$-normal form contains an $\eta$-pattern if and only if $f^\Gamma_p(\pi)$ contains an $\eta$-redex.
\end{lemma}
\begin{proof}
Suppose that $\pi$ is an $(L \imp)$-normal form which contains an $\eta$-pattern \eqref{eq:eta_pattern}. Then Lemma \ref{lemma:var_occurs} shows that $z:p$ occurs as a free variable in $f^{\Theta,z:p,\Theta'}_p(\theta)$ which must therefore be equal to $z:p$. The translation of the part of the $\eta$-pattern ending at the $(L \imp)$ rule is therefore $(y \, x')$ where $x':s = f^{\Delta}_s(\zeta)$. Since $\pi$ is well-labelled there is precisely one occurrence of $x':s$ in $\Delta$ which is a weak ancestor of $x:s$ but not necessarily a strong ancestor. Since this occurrence cannot be a weak ancestor both of $x:s$ and of an occurrence eliminated in a $(L \imp)$ rule, we see that the translation of the $\eta$-pattern is of the form $\lambda x . (M \, x)$ for some term $M$. 

This term $M$ is constructed from $(L \imp)$ rules within the $\eta$-pattern starting with $y$ and the only way for $x$ to appear as a free variable in $M$ is for some weak ancestor of $x$ to appear in the antecedent of the left hand branch of one of these $(L \imp)$ rules. But by condition (iii) of a special $(L \imp)$-normal form such weak ancestors must all be introducing by weakenings, from which we conclude that $x$ is not free in $M$. This shows that the translation of the $\eta$-pattern is an $\eta$-redex, which survives in the translation of $\pi$.

Conversely, suppose that $f^\Gamma_p(\pi)$ contains an $\eta$-redex $\lambda x. (M \, x)$ where $x:s, M:s \imp p$. Then $\eta$ contains, since it is well-labelled, precisely one $(R \imp)$ rule that eliminates an occurrence of $x:s$ and we may assume it is as displayed in \eqref{eq:eta_pattern}. Follow the tree upwards from this rule taking the right hand branch at every $(L \imp)$ rule until an $(L \imp)$ rule is encountered for which the translation of the right hand branch $\theta$ is a variable $z:p$ and an occurrence of this variable is eliminated by the $(L \imp)$ rule. Since the translation of the tree above the $(R \imp)$ rule is $(M \, x)$ this walk encounters no $(R \imp)$ rule and is guaranteed to encounter an $(L \imp)$ rule of the specified kind. The left hand branch $\zeta$ of this $(L \imp)$ rule must similarly have for its translation a variable.

Now consider the contraction tree of $x:s$. It is clear that it contains one leaf corresponding to a weak ancestor introduced by $(\operatorname{ax})$ in $\zeta$. Suppose that there were another weak ancestor introduced by $(L \imp)$ or $(\operatorname{ax})$. By the proof of Lemma \ref{lemma:preimage_vars} we know that $\zeta, \theta$ contain no $(L \imp)$ rules so this other weak ancestor must be introduced between $(R \imp)$ and $(L \imp)$ in the $\eta$-pattern or in one of the left hand branches of one of the intermediate $(L \imp)$ rules and therefore occurs as a free variable in $M$, which is a contradiction. Hence $\pi$ contains an $\eta$-pattern.
\end{proof}

\begin{lemma}\label{lemma:simo_specialnormal} Suppose that $\pi_1,\pi_2$ are $(L \imp)$-normal forms with $\pi_1 \sim_o \pi_2$. If $\pi_1$ is a special $(L \imp)$-normal form then so is $\pi_2$.
\end{lemma}
\begin{proof}
Immediate from Lemma \ref{lemma:F_welldefined}(i) and Lemma \ref{lemma:etapattern_implies_redex}.
\end{proof}

\begin{lemma}\label{lemma:special_means_betaetanormal} If $\pi$ is a special $(L \imp)$-normal form then $f^\Gamma_p(\pi)$ is a $\beta\eta$-normal form.
\end{lemma}
\begin{proof}
Immediate from Lemma \ref{lemma:cutfree_means_betanormal} and Lemma \ref{lemma:etapattern_implies_redex}.
\end{proof}

\begin{lemma}\label{lemma:special_l_normal} Every preproof $\pi$ is equivalent under $\sim_p$ to a special $(L \imp)$-normal form.
\end{lemma}
\begin{proof}
We may by Lemma \ref{lemma:L_normal_form} assume $\pi$ is a $(L \imp)$-normal form. Consider an $\eta$-pattern \eqref{eq:eta_pattern} within $\pi$. By Lemma \ref{lemma:preimage_vars} there is a preproof equivalent under $\sim_p$ to $\pi$ in which the branch of the proof given by the $\eta$-pattern is replaced by
\begin{center}
\begin{tabular}{ >{\centering}m{10cm} >{\centering}m{0.5cm}}
    \AxiomC{}
    \RightLabel{$(\operatorname{ax})$}
    \UnaryInfC{$x':s \vdash x':s$}
    \AxiomC{}
    \RightLabel{$(\operatorname{ax})$}
    \UnaryInfC{$z:p \vdash z:p$}
    \RightLabel{$(L \imp)$}
    \BinaryInfC{$y:s \imp p,x':s \vdash (y \, x'):p$}
 \noLine
 \UnaryInfC{$\vdots$}
 \noLine
 \UnaryInfC{$\Gamma,x:s,\Gamma' \vdash M:p$}
 \RightLabel{$(R \imp)$}
 \UnaryInfC{$\Gamma,\Gamma' \vdash \lambda x.(M \, x):s \imp p$}
 \DisplayProof
 &
 \tagarray{\label{eq:eta_pattern2}}
 \end{tabular}
\end{center}
where we use commuting conversions to move any $(\operatorname{weak})$ rules below the $(L \imp)$. Using \eqref{co_weak_ctr} and Remark \ref{remark:old_co_ctr_comm2} we may eliminate all weak ancestors of $x:s$ in $\pi$ except for the displayed $x':s$, yielding a preproof in which the topmost occurrence of $x:s$ is the strong ancestor of occurrence eliminated in the $(R \imp)$ rule:
\begin{center}
\begin{tabular}{ >{\centering}m{10cm} >{\centering}m{0.5cm}}
    \AxiomC{}
    \RightLabel{$(\operatorname{ax})$}
    \UnaryInfC{$x:s \vdash x:s$}
    \AxiomC{}
    \RightLabel{$(\operatorname{ax})$}
    \UnaryInfC{$z:p \vdash z:p$}
    \RightLabel{$(L \imp)$}
    \BinaryInfC{$y:s \imp p, x:s \vdash (y \, x):p$}
 \noLine
 \UnaryInfC{$\vdots$}
 \noLine
 \UnaryInfC{$\Gamma,x:s,\Gamma' \vdash M:p$}
 \RightLabel{$(R \imp)$}
 \UnaryInfC{$\Gamma,\Gamma' \vdash \lambda x.(M \, x):s \imp p$}
 \DisplayProof
 &
 \tagarray{\label{eq:eta_pattern3}}
 \end{tabular}
\end{center}
The rules intermediate between the $(R \imp)$ and $(L \imp)$ in \eqref{eq:eta_pattern3} are either structural rules or $(L \imp)$ rules and by \eqref{comm_R_ex},\eqref{comm_R_weak},\eqref{comm_R_ctr} and \eqref{comm_R_L} we may commute the $(R \imp)$ with all of these rules, until we obtain a preproof equivalent to $\pi$ under $\sim_p$ with the original $\eta$-pattern branch replaced by
\begin{center}
    \AxiomC{}
    \RightLabel{$(\operatorname{ax})$}
    \UnaryInfC{$x:s \vdash x:s$}
    \AxiomC{}
    \RightLabel{$(\operatorname{ax})$}
    \UnaryInfC{$z:p \vdash z:p$}
    \RightLabel{$(L \imp)$}
    \BinaryInfC{$y:s \imp p,x:s \vdash (y \, x):p$}
    \RightLabel{$(R \imp)$}
    \UnaryInfC{$y:s \imp p \vdash \lambda x.(y \, x):s \imp p$}
 \noLine
 \UnaryInfC{$\vdots$}
 \noLine
 \DisplayProof
\end{center}
which is by $\eqref{eta}$ equivalent to
\begin{center}
    \AxiomC{}
    \RightLabel{$(\operatorname{ax})$}
    \UnaryInfC{$y:s \imp p \vdash y: s \imp p$}
 \noLine
 \UnaryInfC{$\vdots$}
 \noLine
 \DisplayProof
\end{center}
Applying the above reasoning to all $\eta$-patterns in $\pi$ from greatest to lowest height (measuring the height at the $(R \imp)$ rule) completes the proof.
\end{proof}

\begin{proof}[Proof of Proposition \ref{prop:curry_howard_actualwork}] Let $\Gamma$ be repetition-free and let $\mathbb{SL}\Sigma^\Gamma_p$ denote the set of preproofs of $\Gamma \vdash p$ which are special $(L \imp)$-normal forms. Let $\sim_p$ denote the induced relation on $\mathbb{SL}\Sigma^\Gamma_p$ noting that two elements may be equivalent via intermediate preproofs that are not special $(L\imp)$-normal forms. The inclusion $\mathbb{SL}\Sigma^\Gamma_p \subseteq \Sigma^\Gamma_p$ induces by Lemma \ref{lemma:special_l_normal} a bijection
\be
\xymatrix{
\mathbb{SL}\Sigma^\Gamma_p/\!\sim_p \ar[r]^-{\cong} & \Sigma^\Gamma_p/\!\sim_p
}
\ee
Recall that $Q = [\Gamma]$. Now consider the translation map $f^\Gamma_p$ restricted to special $(L \imp)$-normal forms and the induced map on the quotients
\[
\overline{f^\Gamma_p}: \mathbb{SL}\Sigma^\Gamma_p/\!\sim_p \lto \Lambda^Q_p/=_{\beta\eta}\,.
\]
We have a commutative diagram 
\be
\xymatrix@C+1pc{
\Sigma^\Gamma_p\ar[d] & \mathbb{SL}\Sigma^\Gamma_p\ar[d] \ar@{->}[l]_-{\operatorname{inc}}\ar[r]^-{f^\Gamma_p}\ar[d] & \Lambda^Q_p \ar[d]\\
\Sigma^\Gamma_p/\!\sim_p & \mathbb{SL}\Sigma^\Gamma_p/\!\sim_p \ar[l]^-{\cong}\ar[r]_-{\overline{f^\Gamma_p}} & \Lambda^Q_p / =_{\beta \eta}
}
\ee
in which the vertical arrows are the canonical maps to the quotient. It clearly suffices to prove that $\overline{f^\Gamma_p}$ is a bijection.

To prove it is injective, let $\pi_1,\pi_2 \in \mathbb{SL}\Sigma^\Gamma_p$ be such that $f^\Gamma_p(\pi_1) =_{\beta \eta} f^\Gamma_p(\pi_2)$. Since both of these terms are $\beta\eta$-normal forms by Lemma \ref{lemma:special_means_betaetanormal} it follows from a standard result in the theory of lambda calculus \cite[Corollary 4.3]{selinger} that $f^\Gamma_p(\pi_1) = f^\Gamma_p(\pi_2)$ in $\Lambda^Q_p$. Since $\Gamma$ is assumed to be repetition-free Proposition \ref{prop:alpha_implies_p} then implies $\pi_1 \sim_p \pi_2$ as required.

To prove surjectivity of $\overline{f^\Gamma_p}$ we prove surjectivity of the map
\be\label{eq:translate_betaeta}
f^\Gamma_p: \mathbb{SL}\Sigma^\Gamma_p \lto \mathbb{N} \Lambda^Q_p
\ee
where $\mathbb{N} \Lambda^Q_p$ denotes the set of $\beta\eta$-normal forms. The proof is by induction of the proposition $P(n)$ which says that for any repetition-free $\Gamma$ and formula $p$ any $\beta\eta$-normal lambda term $M$ of length $n$ is in the image of \eqref{eq:translate_betaeta} where $Q = [\Gamma]$. By appending exchanges and weakenings we may assume without loss of generality that $\Gamma$ is the set of distinct free variables of $M$, in order of appearance. The base case is clear by inspection of \eqref{eq:var_normal_form}. If $M = \lambda x . N \in \mathbb{N}\Lambda^Q_p$ is an abstraction with $p = q \imp r, x:q$ and $N:r$ and $x \notin Q$ then $N \in \mathbb{N}\Lambda^{Q \cup \{x:q\}}_r$ so by the inductive hypothesis there is a special $(L \imp)$-normal form $\psi$ with $f^{\Gamma,x:q}_r(\psi) = N$ and by appending a $(R \imp)$ rule to $\pi$ as in \eqref{eq:abstract_normal_form} we construct a special $(L \imp)$-normal form $\pi$ with $f^\Gamma_p(\pi) = M$. If $M \in \mathbb{N} \Lambda^Q_p$ is an application then since $M$ is $\beta\eta$-normal it must be of the form \eqref{eq:normal_form_app_term2} that is
\be
M = \big( (\cdots( (y_b \, L'_b) \, L'_{b-1} ) \, \cdots L'_1) \, R'\big)
\ee
for some formulas $p_1,\ldots,p_b,s$ and $\beta\eta$-normal terms $L'_j:p_j$ and $R':s$ and variable $y_b$. Possibly $b = 0$ in which case $M = (y R')$. As in the proof of Lemma \ref{lemma:preimage_app} we construct from this data a sequence of formulas $y_b,\Gamma_b,\ldots,\Gamma_1,\Delta$ and terms $R:s$ and $L_j:p_j$ for $1 \le j \le b$. By the inductive hypothesis we have special $(L \imp)$-normal forms $\tau_j$ and $\zeta$ such that $f^{\Gamma_j}_{p_j}(\tau_j) = L_j$ and $f^{\Delta}_s(\zeta) = R$. From these preproofs and the contraction pattern that produces $y_b,L'_1,\ldots,L'_b,R'$ from $y_b,L_1,\ldots,L_b,R$ we construct an application normal form $\pi$ as given in the statement of Lemma \ref{lemma:preimage_app} with $f^\Gamma_p(\pi) = M$. By construction $\pi$ is a special $(L \imp)$-normal form so the proof is complete.
\end{proof}

Let $\mathbb{N} \Lambda^Q_p$ denote the subset of $\beta\eta$-normal forms in $\Lambda^Q_p$. What the proof of Proposition \ref{prop:curry_howard_actualwork} actually shows is that there is a bijection
\be
\xymatrix{
\mathbb{SL}\Sigma^\Gamma_p / \!\sim_o \ar[r]^-{\cong} & \mathbb{N} \Lambda^Q_p\,.
}
\ee
This is still not satisfactory. For example, we cannot rule out \emph{a priori} that there are some special $(L \imp)$-normal forms $\pi_1,\pi_2$ that are related by $\sim_o$ but every chain of generating relations between them involves intermediate preproofs which are not special $(L \imp)$-normal forms. The methods already developed suffice to prove a much stronger statement, which we treat systematically in Section \ref{section:normal_form_seq}.


\subsection{Normal forms for sequent calculus proofs}\label{section:normal_form_seq}

The cut-elimination theorem of Gentzen \cite{gentzen} is the first step in the direction of establishing a normal form for sequent calculus proofs, but as there remain many cut-free proofs in sequent calculus that are ``the same'' this can hardly be called a normal form. The work of Mints \cite{mints} building on Kleene's work on permutative conversions \cite{kleene} is the first to establish a true normal form result for sequent calculus proofs, albeit in a system that is not quite standard LJ. In this section we revisit the topic of such normal forms.

The guiding principle behind our normal form for sequent calculus proofs is the concept of \emph{encapsulation}. Consider a preproof of the form
\begin{center}
    \AxiomC{$\zeta$}
    \noLine
    \UnaryInfC{$\vdots$}
    \noLine
    \UnaryInfC{$\Delta \vdash R : s$}
    \AxiomC{}
    \RightLabel{$(\operatorname{ax})$}
    \UnaryInfC{$x:p \vdash x:p$}
    \RightLabel{$(L \imp)$}
    \BinaryInfC{$y,\Delta \vdash (y \, R):p$}
    \noLine
    \UnaryInfC{$\vdots$}
    \DisplayProof
\end{center}
The left hand branch of the $(L \imp)$ rule supplies a term $R$ that may be viewed as either data or a subroutine. This subroutine is \emph{well encapsulated} if it is possible to apprehend its role in the broader proof entirely by inspecting the branch itself, that is, if the meeting point between $\zeta$ and the rest of the proof at this $(L \imp)$ rule serves as a \emph{boundary} across which there is minimal information flow. These are vague statements; to be more precise, we identify two kinds of boundary violation which break this principle of encapsulation. There are other kinds of boundary violations that one may imagine, but these are already impossible in special $(L \imp)$-normal forms so we do not elaborate them.

In the following $\pi$ denotes a preproof of $\Gamma \vdash p$ and we assume $\Gamma$ is repetition-free. If two variable occurrences are introduced above a boundary and contracted below it, then this creates a boundary violation of contraction type:

\begin{definition} A \emph{boundary violation of $(\operatorname{ctr})$ type} in $\pi$ is a pair consisting of a $(L \imp)$ rule and a $(\operatorname{ctr})$ rule, with the latter below the former as in
\begin{center}
\begin{tabular}{ >{\centering}m{10cm} >{\centering}m{0.5cm}}
    \AxiomC{$\vdots$}
    \noLine
    \UnaryInfC{$\Lambda \vdash s$}
    \AxiomC{$\vdots$}
    \noLine
    \UnaryInfC{$$}
    \RightLabel{$(L \imp)$}
    \BinaryInfC{$$}
 \noLine
 \UnaryInfC{$\vdots$}
 \noLine
 \UnaryInfC{$\Gamma,x:p,x':p,\Gamma' \vdash q$}
 \RightLabel{$(\operatorname{ctr})$}
 \UnaryInfC{$\Gamma,x:p,\Gamma' \vdash q$}
 \DisplayProof
 &
 \tagarray{\label{eq:bound_ctr_pattern}}
 \end{tabular}
\end{center}
where the final occurrence of $x:p$ has at least two distinct weak ancestors in $\Lambda$.
\end{definition}

If a variable occurrence is introduced above a boundary and eliminated by a $(L \imp)$ rule below it, this creates a boundary violation of $(L \imp)$-type:

\begin{definition} A \emph{boundary violation of $(L \imp)$ type} in $\pi$ is a pair consisting of two $(L \imp)$ rules as in
\begin{center}
\begin{tabular}{ >{\centering}m{10cm} >{\centering}m{0.5cm}}
    \AxiomC{$\vdots$}
    \noLine
    \UnaryInfC{$\Gamma \vdash p$}
    \AxiomC{$\vdots$}
    \noLine
    \UnaryInfC{$\Lambda \vdash s$}
    \AxiomC{$\vdots$}
    \noLine
    \UnaryInfC{$$}
    \RightLabel{$(L \imp)$}
    \BinaryInfC{$$}
    \noLine
    \UnaryInfC{$\vdots$}
    \noLine
    \UnaryInfC{$\Delta, x:q, \Theta \vdash r$}
    \RightLabel{$(L \imp)$}
    \BinaryInfC{$y: p \imp q, \Gamma, \Delta, \Theta \vdash r$}
 \DisplayProof
 &
 \tagarray{\label{eq:bound_L_pattern}}
 \end{tabular}
\end{center}
where the variable occurrence $x:q$ has a strong ancestor in $\Lambda$.
\end{definition}

Recall the notation $(\operatorname{dctr})^{i,j}$ for derived contractions from Definition \ref{definition:derived_ctr}. A \emph{rule pair} in $\pi$ is a pair of rules $(r),(r')$ adjacent in the underlying tree of $\pi$ with $(r')$ occurring immediately after $(r)$ on the path from $(r)$ to the root.

\begin{definition} A preproof $\pi$ of $\Gamma \vdash p$ is called \emph{well-structured} if it is a special $(L \imp)$-normal form and further satisfies the following conditions:
\begin{itemize}
\item[(a)] There are no boundary violations of $(\operatorname{ctr})$ type.
\item[(b)] There are no boundary violations of $(L \imp)$ type.
\item[(c)] The only $(\operatorname{weak})$ rules occur in pairs $(\operatorname{weak}), (R \imp)$ with the second rule eliminating the variable occurrence introduced by the first, which is leftmost in the antecedent.
\item[(d)] There is no rule pair $(r), (L \imp)$ with $(r)$ structural on the right branch.
\item[(e)] There is no rule pair $(R \imp), (r)$ where $(r)$ is structural.
\item[(f)] There is no rule pair $(R \imp), (L \imp)$ with the $(R \imp)$ on the right branch.
\item[(g)] There is no pair $(\operatorname{dctr})^{a,b}, (\operatorname{dctr})^{a',b'}$ of consecutive maximal derived contractions with $(a',b') < (a,b)$ in the lexicographic ordering.
\item[(h)] Every $(\operatorname{ex})$ rule occurs as part of a derived contraction.
\end{itemize}
\end{definition}

Recall from Definition \ref{definition:well_ordered} the notion of a well-ordered preproof.

\begin{definition}\label{defn:normal} A preproof $\pi$ is \emph{normal} if it is of the form
\begin{center}
\AxiomC{$\psi$}
\noLine
\UnaryInfC{$\vdots$}
\noLine
\UnaryInfC{$\Gamma'' \vdash p$}
\RightLabel{$(\operatorname{lad})$}
\doubleLine
\UnaryInfC{$\Gamma' \vdash p$}
\RightLabel{$(\operatorname{weak})$}
\doubleLine
\UnaryInfC{$\Gamma \vdash p$}
\DisplayProof
\end{center}
where $\psi$ is well-ordered and well-structured, and the ladders and weakening rules are
\begin{gather*}
(\operatorname{lad})^{c_1},(\operatorname{lad})^{c_2},\ldots,(\operatorname{lad})^{c_n}\\
(\operatorname{weak})^{d_1},(\operatorname{weak})^{d_2},\ldots,(\operatorname{weak})^{d_m}
\end{gather*}
with $c_1 < c_2 < \cdots < c_n$ and $d_1 < d_2 < \cdots < d_m$ (using the notation of Definition \ref{definition:ladder} and Definition \ref{definition:index_weak}). One or both of these series of rules may be empty.
\end{definition}

\begin{remark}\label{remark:on_normal}
Note that by (c), (h) no well-structured preproof can end with exchanges or weakenings, so that the subproof $\psi$ of Definition \ref{defn:normal} can be unambiguously recovered from the normal preproof $\pi$. The sequence $\Gamma''$ is by the hypothesis of well-ordering determined by the term $f^{\Gamma}_p(\pi) = f^{\Gamma''}_p(\psi)$ and so from this term and $\Gamma$ the ladders and weakening rules and their order are completely determined. 
\end{remark}


\begin{lemma}\label{lemma:subproof} If $\pi$ is well-structured then any subproof of $\pi$ not ending in $(\operatorname{weak})$ or $(\operatorname{ex})$ is also well-structured.
\end{lemma}
\begin{proof}
Left to the reader.
\end{proof}

\begin{proposition}\label{prop:normal_form} Let $\pi$ be a preproof of $\Gamma \vdash p$. Then
\begin{itemize}
\item[(I)] If $M$ is a variable then $\pi$ is well-structured if and only if it is an axiom rule.
\item[(II)] If $M$ is an abstraction then $\pi$ is well-structured if and only it is equivalent under $\sim_\alpha$ to an abstraction normal form \eqref{eq:abstract_normal_form} where there is no $\eta$-pattern involving the final $(R \imp)$ rule and the subproof $\psi$ of \eqref{eq:abstract_normal_form} is either well-structured, or is a well-structured proof followed by a single $(\operatorname{weak})$ rule with the introduced variable leftmost in the antecedent and eliminated by the final rule in $\pi$.
\item[(III)] If $M$ is an application then $\pi$ is well-structured if and only if it is equivalent under $\sim_\alpha$ to an application normal form as in Lemma \ref{lemma:preimage_app} in which $\zeta$ and $\tau_j$ for $1 \le j \le b$ are well-structured and the rule series \eqref{eq:tiny_weak}, \eqref{eq:tiny_lad} are empty. 
\end{itemize}
\end{proposition}
\begin{proof}
(I) If $M$ is a variable and $\pi$ is well-structured, consulting the proof of Lemma \ref{lemma:preimage_vars} we see that by condition (c) of the well-structured property there are no $(\operatorname{ctr})$ or $(\operatorname{weak})$ rules in $\pi$. There are no $(\operatorname{ex})$ rules by (h). So $\pi$ is an axiom rule. Conversely, it is clear that an axiom rule is well-structured. 

(II) Suppose $M$ is an abstraction and $\pi$ is well-structured. We must show the final rule in $\pi$ is $(R \imp)$. If we walk the tree from the root taking only right branches of $(L \imp)$ rules we eventually encounter a $(R \imp)$ rule. The only rules that may precede the first $(R \imp)$ on this walk are $(L \imp)$ and structural rules, and these are impossible by (e),(f) so $\pi$ must end in $(R \imp)$ and we are done. The reverse implication in (II) is also clear.

(III) For the reverse direction in (III) observe that if $\pi$ is a well-labelled application normal form satisfying the conditions then it is a special $(L \imp)$-normal form. Conditions (a), (b) follow respectively from condition (iii) of application normal form and the shape of the normal form proof tree, together with the assumption that the branches are well-structured. Any $(\operatorname{weak})$ rule in $\pi$ either occurs in the $\tau_j$ or $\zeta$ or at the bottom of $\pi$, and the latter is explicitly ruled out, so (c) is satisfied. Similarly for conditions (d)-(h), noting that (g) uses condition (iv) of an application normal form.

Finally suppose $M$ is an application and that $\pi$ is well-structured. Consulting the proof of Lemma \ref{lemma:preimage_app} the structural rules can only occur at the bottom of the porch by (d). Nothing needs to be done in the migration phase by (b). By (c) there are no $(\operatorname{weak})$ rules in $\zeta, \tau_j$ that need to be moved to the bottom of the porch, and nothing needs to be done to satisfy (iii) by (a). By (c),(h) the only structural rules on the porch are part of derived contractions which satisfy (iv) by condition (g). We are free to change $\pi$ up to $\alpha$-equivalence so we may assume (v) is satisfied, and (vi),(vii) are vacuous. So it only remains to prove (ii).

To do this we first prove Corollary \ref{corollary:well_struc_implies_well_ord} below, which requires only the part of (III) that we have already proven. Suppose for a contradiction that a well-structured preproof exists which is not well-ordered, and let $\rho$ be an example with $L = f^\Lambda_r(\rho)$ of minimal length. By (I) this term $L$ cannot be a variable, since an axiom rule is well-ordered. If $L$ were an application then by the part of (III) already proven $\rho$ is equivalent under $\sim_\alpha$ to a preproof which is an application normal form but for the possible failure of (ii); but if any of the branches failed to be well-ordered this would contradict minimality of $L$, and if they are all well-ordered then (ii) is satisfied and $\rho$ would therefore be well-ordered by Lemma \ref{lemma:app_normal_wello}, a contradiction. So the only possibility is that $L$ is an abstraction. By (II) then $\rho$ is equivalent under $\sim_\alpha$ to an abstraction normal form
\begin{center}
\begin{tabular}{ >{\centering}m{10cm}}
    \AxiomC{$\psi$}
    \noLine
    \UnaryInfC{$\vdots$}
    \noLine
    \UnaryInfC{$\Delta, x:q, \Delta' \vdash N:r$}
    \RightLabel{$(R \imp)$}
    \UnaryInfC{$\Delta, \Delta' \vdash \lambda x. N: q \imp r$}
 \DisplayProof
 \end{tabular}
\end{center}
By hypothesis $\Delta, \Delta' \neq \operatorname{FV}^{seq}( \lambda x . N )$. If $x:q$ is introduced by $(\operatorname{weak})$ then this contradicts minimality of $L$, and if not then by minimality $\Delta, x:q, \Delta' = \operatorname{FV}^{seq}(N)$ which contradicts $\Delta, \Delta' \neq \operatorname{FV}^{seq}( \lambda x . N )$. This completes the proof of the corollary.

Returning now to the proof of the theorem proper, the branches $\zeta, \tau_j$ cannot end in $(\operatorname{weak})$ by (c) and cannot end in $(\operatorname{ex})$ by (h) so by Lemma \ref{lemma:subproof} they are well-structured and hence by Corollary \ref{corollary:well_struc_implies_well_ord} they are well-ordered, which shows condition (ii) of an application normal form and completes the proof.
\end{proof}

\begin{corollary}\label{corollary:well_struc_implies_well_ord} If $\rho$ is a well-structured preproof then it is well-ordered.
\end{corollary}

In particular, a well-structured preproof is precisely a normal preproof in which the series of ladders and weakenings at the bottom are empty. We may now prove a strengthening of Proposition \ref{prop:alpha_implies_p}:

\begin{proposition}\label{prop:alpha_iff_alpha} If $\Gamma$ is repetition-free and $\pi_1, \pi_2$ are normal preproofs of $\Gamma \vdash p$ then $f^\Gamma_p(\pi_1) = f^\Gamma_p(\pi_2)$ implies $\pi_1 \sim_\alpha \pi_2$.
\end{proposition}
\begin{proof}
If $\pi_1,\pi_2$ are normal and $f^\Gamma_p(\pi_1) = f^\Gamma_p(\pi_2)$ then by Remark \ref{remark:on_normal} the ladders and weakenings at the bottom of $\pi_1,\pi_2$ agree and writing $\psi_1, \psi_2$ for the well-structured subproofs as in Definition \ref{defn:normal} we have
\begin{align}
f^\Gamma_p(\pi_1) = f^\Gamma_p(\pi_2) &\Longleftrightarrow f^{\Gamma''}_p(\psi_1) = f^{\Gamma''}_p(\psi_2)\label{eq:alpha_iff_alpha1}\\
\pi_1 \sim_\alpha \pi_2 &\Longleftrightarrow \psi_1 \sim_\alpha \psi_2 \label{eq:alpha_iff_alpha2}
\end{align}
The proof is similar to Proposition \ref{prop:alpha_implies_p} and is again by induction on the length of the term $M = f^\Gamma_p(\pi_1) = f^\Gamma_p(\pi_2)$. In the base case $M$ is a variable, and by what we have just said and Proposition \ref{prop:normal_form} (I) it is immediate that $\pi_1 \sim_{\alpha} \pi_2$. If $M$ is an abstraction $\lambda x. N$ then by Proposition \ref{prop:normal_form} (II) both $\psi_1, \psi_2$ end in $(R \imp)$ rules and we let $\psi'_1, \psi'_2$ denote the subproofs of $\Delta_1, x:q, \Delta'_1 \vdash q$ and $\Delta_2, x:q, \Delta'_2 \vdash q$ respectively obtained by deleting these final rules. These are either both well-structured (let us call this the first case) or both well-structured after deleting a final $(\operatorname{weak})$ rule which introduces $x:q$ in the leftmost position (call this the second case) hence from Corollary \ref{corollary:well_struc_implies_well_ord} we deduce that
\[
\Delta_1, x:q, \Delta'_1 = \Delta_2, x:q, \Delta'_2
\]
and this sequence is in the first case $\operatorname{FV}^{seq}(N)$ and in the second case $x:q,\operatorname{FV}^{seq}(N)$. In the first case let $\psi''_i = \psi'_i$ and in the second case let $\psi''_i$ be obtained from $\psi'_i$ by deleting the final $(\operatorname{weak})$, for $i \in \{1, 2\}$. Then by Corollary \ref{corollary:well_struc_implies_well_ord} the preproofs $\psi''_i$ are normal preproofs of the same sequent $\Theta \vdash q$ and 
\[
f^{\Theta}_q(\psi''_1) = N = f^{\Theta}_q(\psi''_2)
\]
so by the inductive hypothesis $\psi''_1 \sim_\alpha \psi''_2$ from which it follows that $\psi'_1 \sim_\alpha \psi'_2$ and hence $\pi_1 \sim_\alpha \pi_2$. If $M$ is an application then by Proposition \ref{prop:normal_form} (III) both $\psi_1, \psi_2$ are equivalent under $\sim_\alpha$ to application normal forms, in which the $\tau_j$ and $\zeta$ are well-structured, and so by the inductive hypothesis equivalent under $\sim_\alpha$, hence $\pi_1 \sim_\alpha \pi_2$.
\end{proof}

\begin{lemma}\label{lemma:all_normal_eq} If $\Gamma$ is repetition-free then every preproof $\pi$ is equivalent under $\sim_p$ to a normal preproof.
\end{lemma}
\begin{proof}
We may by Lemma \ref{lemma:special_l_normal} prove the lemma for special $(L \imp)$-normal forms $\pi$, in which case the proof is by induction on the length of $M = f^\Gamma_p(\pi)$.
By Lemma \ref{lemma:preimage_vars}, Lemma \ref{lemma:preimage_abstractions} and Lemma \ref{lemma:preimage_app} $\pi$ is equivalent under $\sim_o$ to one of the three types of normal forms $\pi'$. In the base case $M$ is a variable, and the claim follows from Proposition \ref{prop:normal_form} (I). 

For the inductive step, if $\pi'$ is an abstraction normal form, we may assume by the inductive hypothesis that the subproof $\psi$ obtained by deleting the final $(R \imp)$ rule is normal, and after moving exchanges and weakenings below the $(R \imp)$ rule we may assume $\psi$ satisfies the hypotheses of Proposition \ref{prop:normal_form} (II), so that $\pi'$ is normal. If $\pi'$ is an application normal form then by the inductive hypothesis we may assume $\tau_j$ for $1 \le j \le b$ and $\zeta$ are normal. By condition (ii) of an application normal form these branches cannot end in $(\operatorname{weak})$ rules. Let $\kappa$ denote one of the $\tau_j$ or $\zeta$ and suppose that $\kappa$ ends in a series of $(\operatorname{ex})$ rules. The well-structured subproof $\kappa'$ of $\kappa$ obtained by deleting these rules is well-ordered and has the same translation as $\kappa$, which is also well-ordered, so the series of $(\operatorname{ex})$ rules implement the identity permutation and may be deleted using \eqref{tau_ex_ex}. Hence we may assume without loss of generality that $\tau_j$ for $1 \le j \le b$ and $\zeta$ are not just normal, but well-structured. Hence by Proposition \ref{prop:normal_form} (III) the preproof $\pi'$ is normal.
\end{proof}

Recall that $\mathbb{N} \Lambda^Q_p$ denotes the subset of $\beta\eta$-normal forms in $\Lambda^Q_p$. We let $\mathbb{N}\Sigma^\Gamma_p$ denote the set of normal preproofs of $\Gamma \vdash p$ in the sense of Definition \ref{defn:normal}.

\begin{theorem}\label{theorem:normal_form_prop} If $\Gamma$ is repetition-free there is a commutative diagram
\be\label{eq:most_important_map_2}
\xymatrix@C+1pc@R+1pc{
\Sigma^\Gamma_p/ \sim_p\ar^-{\cong}[r] & \Lambda^Q_p/ \! =_{\beta \eta}\\
\mathbb{N}\Sigma^\Gamma_p / \sim_\alpha \ar[u]^-{\cong}\ar_-{\cong}[r] & \mathbb{N}\Lambda^Q_p \ar_-{\cong}[u]
}
\ee
in which the rows are bijections induced by the function $f^\Gamma_p$ and the columns are bijections induced by the inclusions $\mathbb{N} \Sigma^\Gamma_p \subseteq \Sigma^\Gamma_p$ and $\mathbb{N} \Lambda^Q_p \subseteq \Lambda^Q_p$.
\end{theorem}
\begin{proof}
There is clearly a commutative diagram of this form, and the first row is a bijection by Proposition \ref{prop:curry_howard_actualwork}. The second column is a bijection by the existence and uniqueness of $\beta\eta$-normal forms \cite[Corollary 4.3]{selinger}. Surjectivity of the first column is Lemma \ref{lemma:all_normal_eq} so it suffices to prove the second row is injective, which is Proposition \ref{prop:alpha_iff_alpha}.
\end{proof}

\begin{example}\label{example:church_2_normalise} The well-labelled Church numeral $\underline{2}$ of Example \ref{example:translation_2} is not normal, since it contains a boundary violation of $(L \imp)$ type, highlighted below:
\begin{prooftree}
        \AxiomC{}
        \RightLabel{$({\operatorname{ax}})$}
        \UnaryInfC{$x:p \vdash x:p$}
        \AxiomC{}
        \RightLabel{$({\operatorname{ax}})$}
        \UnaryInfC{$\textcolor{blue}{x'}:p \vdash x':p$}
        \AxiomC{}
        \RightLabel{$({\operatorname{ax}})$}
        \UnaryInfC{$x'':p \vdash x'':p$}
        \RightLabel{$(L \imp)$}
        \BinaryInfC{$y': p \imp p, \textcolor{blue}{x'}:p \vdash (y'\, x') : p$}
        \RightLabel{$(L \imp)$}
        \BinaryInfC{$y: p \imp p, y': p \imp p, x:p \vdash (y' \, (y \, x)) : p$}
        \RightLabel{$(\operatorname{ctr})$}
        \UnaryInfC{$y: p \imp p , x:p\vdash (y \, (y \, x)):p$}
        \RightLabel{$(R \imp)$}
        \UnaryInfC{$y:p \imp p \vdash \lambda x . (y \, (y \, x)) : p \imp p$}
\end{prooftree}
The algorithm of the proof of Lemma \ref{lemma:preimage_app} eliminates this boundary violation as part of the ``migration'' phase, which consists in this case of an application of \eqref{comm_L_L2} resulting in the $\sim_o$-equivalent preproof
\begin{prooftree}
        \AxiomC{}
        \RightLabel{$({\operatorname{ax}})$}
        \UnaryInfC{$x:p \vdash x:p$}
        \AxiomC{}
        \RightLabel{$({\operatorname{ax}})$}
        \UnaryInfC{$\textcolor{blue}{x'}:p \vdash x':p$}
        \RightLabel{$(L \imp)$}
        \BinaryInfC{$y: p \imp p, x:p \vdash (y\, x) : p$}
        \AxiomC{}
        \RightLabel{$({\operatorname{ax}})$}
        \UnaryInfC{$x'':p \vdash x'':p$}
        \RightLabel{$(L \imp)$}
        \BinaryInfC{$y': p \imp p, y: p \imp p, x:p \vdash (y' \, (y \, x)) : p$}
        \RightLabel{$(\operatorname{ex})$}
        \UnaryInfC{$y: p \imp p, y': p \imp p, x:p \vdash (y' \, (y \, x)) : p$}
        \RightLabel{$(\operatorname{ctr})$}
        \UnaryInfC{$y: p \imp p , x:p\vdash (y \, (y \, x)):p$}
        \RightLabel{$(R \imp)$}
        \UnaryInfC{$y:p \imp p \vdash \lambda x . (y \, (y \, x)) : p \imp p$}
\end{prooftree}
This is still not normal, but applying \eqref{co_ctr_comm_alt} the above preproof is $\sim_o$-equivalent to
\begin{center}
\begin{tabular}{ >{\centering}m{12cm} >{\centering}m{0.5cm}}
        \AxiomC{}
        \RightLabel{$({\operatorname{ax}})$}
        \UnaryInfC{$x:p \vdash x:p$}
        \AxiomC{}
        \RightLabel{$({\operatorname{ax}})$}
        \UnaryInfC{$x':p \vdash x':p$}
        \RightLabel{$(L \imp)$}
        \BinaryInfC{$y': p \imp p, x:p \vdash (y'\, x) : p$}
        \AxiomC{}
        \RightLabel{$({\operatorname{ax}})$}
        \UnaryInfC{$x'':p \vdash x'':p$}
        \RightLabel{$(L \imp)$}
        \BinaryInfC{$y: p \imp p, y': p \imp p, x:p \vdash (y \, (y' \, x)) : p$}
        \RightLabel{$(\operatorname{ctr})$}
        \UnaryInfC{$y: p \imp p , x:p\vdash (y \, (y \, x)):p$}
        \RightLabel{$(R \imp)$}
        \UnaryInfC{$y:p \imp p \vdash \lambda x . (y \, (y \, x)) : p \imp p$}
        \DisplayProof
        &
        \tagarray{\label{eq:2_final_normal}}
    \end{tabular}
\end{center}
which is normal.
\end{example}

\begin{remark}\label{remark:inverse_map} The translation function $f^\Gamma_p$ induces by Theorem \ref{theorem:normal_form_prop} a bijection between $\alpha$-equivalence classes of normal preproofs and $\beta\eta$-normal lambda terms. The inverse map
\be\label{equation:inverse_translation}
g^\Gamma_p: \mathbb{N}\Lambda^Q_p \lto \mathbb{N} \Sigma^\Gamma_p / \sim_\alpha
\ee
is implicit in Proposition \ref{prop:normal_form} and we now make this explicit. Given a $\beta\eta$-normal lambda term $M$ with free variables contained in $Q$ and writing $\Delta = \operatorname{FV}^{seq}(M)$, the preproof $g^\Gamma_p(M)$ is the preproof $g^\Delta_p(M)$ followed by the ladders and weakenings uniquely determined by the pair $\Delta, \Gamma$ as explained in Remark \ref{remark:on_normal}. It therefore suffices to define a well-structured preproof $g^\Gamma_p(M)$ in the special case where $\Gamma = \operatorname{FV}^{seq}(M)$, which we denote by $g(M)$. 

The preproof $g(M)$ has the following inductive definition:
\begin{itemize}
\item If $M = x:p$ is a variable then $g(M)$ is an axiom rule.
\item If $M = \lambda x. N$ is an abstraction with $x:q$ and $N:r$ then $g(M)$ is $g(N)$ followed by, if $x \notin \operatorname{FV}(N)$, a rule pair $(\operatorname{weak}),(R \imp)$ respectively introducing and then eliminating an occurrence of $x:q$, or if $x \in \operatorname{FV}(N)$, a $(R \imp)$ rule eliminating $x:q$. 
\item If $M = (M^1 \, M^2)$ is an application then it is of the form
\be
\big( (\cdots( (y_b \, L_b) \, L_{b-1} ) \, \cdots L_1) \, R\big)
\ee
where $y_b$ is a variable and $R$ and $L_j$ for $1 \le j \le b$ are $\beta\eta$-normal forms. Then $g(M)$ is the application normal form with branches $g(R)$ and $g(L_j)$ for $1 \le j \le b$ ending in the uniquely determined derived contractions.
\end{itemize}
For example, with $M = \lambda x.(y \, (y \, x))$ the normal preproof $g(M)$ is \eqref{eq:2_final_normal}.
\end{remark}


\begin{remark}\label{opencurryhoward}
It is convenient to treat the relationship between our sequent calculus and Zucker's via the Curry-Howard correspondence. Zucker has defined a surjective function which maps from the set of derivations in his sequent calculus $\scr{S}$ \cite[\S 2.2]{zucker} to derivations in natural deduction $\scr{N}$ \cite[\S 2.3]{zucker} denoted $\varphi: \scr{S} \to \scr{N}$ \cite[\S 2.4]{zucker}. Moreover, an equivalence relation $\sim$ on $\scr{S}$ is defined \cite[\S 4.1.2]{zucker} so that the induced map $\scr{S}/\!\sim \lto \scr{N}$ is a bijection. 
The sequent calculus $\scr{S}$ differs from the one considered in this paper in that it omits the weakening and exchange rules; the absence of exchange is compensated by the system $\scr{S}$ having a set of formulas as the antecedent of a sequent, and the absence of weakening is compensated by a special form of the $(R \imp)$ rule.

Let $\Gamma$ be a repetition-free sequence of variables and $\gamma$ a set of indexed formulas in the sense of Zucker \cite[\S 2.2.2]{zucker} such that multiple occurrences of formulas in $\Gamma$ are represented by formulas with distinct indices. With $Q = [\Gamma]$, by Theorem \ref{gentzen_mints_zucker}, \cite[Theorem 1]{zucker} and the Curry-Howard correspondence \cite[\S 6.5]{selinger}, we have a sequence of bijections
\be
\xymatrix@C+2pc{
\mathbb{N} \Sigma^\Gamma_p/\!\sim_\alpha \ar[r]_-{\cong} & \mathbb{N} \Lambda^Q_p \ar[r]_-{\cong} & \scr{N}^{\gamma}_p \ar[r]^-{\varphi^{-1}}_-{\cong} & \scr{S}^{\gamma}_p/\!\sim
}
\ee
where $\scr{N}^{\gamma}_p$ is the set of natural deduction derivations of $p$ from $\gamma$ (a set of assumption classes) and $\scr{S}^{\gamma}_p$ is the set of proofs in $\scr{S}$ of $\gamma \vdash p$, following the convention of \cite[\S 2.4.1]{zucker}.

Informally this bijection takes a normal preproof $\pi \in \bb{N}\Sigma_p^\Gamma$, erases the final weakenings and replaces all $(\operatorname{weak}),(R\imp)$ pairs by Zucker's special $(R \imp)$ rule, erases all exchanges and replaces the antecedent $\Gamma$ in every sequent by an appropriate set of indexed formulas.
\end{remark}

Next we compare the Mints normal form of \cite{mints} with ours; see also \cite{troelstra_margin}. For clarity we refer to the sequent calculus proofs of Mints as \emph{derivations}.

\begin{remark}\label{remark:mints_normal_2} In the sequent calculus system GJ of Mints a normal derivation (in the sense of \cite[Definition 4]{mints}) whose translation is the Church numeral $\lambda x . (y \, (y \, x))$ is
\begin{center}
\begin{tabular}{ >{\centering}m{10cm} >{\centering}m{0.5cm}}
        \AxiomC{}
        \RightLabel{$({\operatorname{ax}})$}
        \UnaryInfC{$p \vdash p$}
        \AxiomC{}
        \RightLabel{$({\operatorname{ax}})$}
        \UnaryInfC{$p \vdash p$}
        \RightLabel{$(\operatorname{weak})$}
        \UnaryInfC{$p, \textcolor{blue}{p} \vdash p$}
        \RightLabel{$(L \imp)$}
        \BinaryInfC{$p \imp p, p \vdash p$}
        \AxiomC{}
        \RightLabel{$({\operatorname{ax}})$}
        \UnaryInfC{$p \vdash p$}
        \RightLabel{$(\operatorname{weak})$}
        \doubleLine
        \UnaryInfC{$p, \textcolor{blue}{p \imp p}, \textcolor{blue}{p} \vdash p$}
        \RightLabel{$(L \imp)$}
        \BinaryInfC{$p \imp p, p \imp p, p \vdash p$}
        \RightLabel{$(R \imp)$}
        \UnaryInfC{$p \imp p, p \imp p \vdash p \imp p$}
        \RightLabel{$(\operatorname{ctr})$}
        \UnaryInfC{$p \imp p \vdash p$}
        \DisplayProof
        &
        \tagarray{\label{eq:2_final_normal_mints}}
    \end{tabular}
\end{center}
Note the weakenings (shown coloured) are forced by GJ's use of a synchronised antecedent $\Gamma$ in the $(L \imp)$ rule\footnote{A literal reading of \cite[Definition 4]{mints} would suggest the above derivation is not M-normal, but this seems to be due to a lack of precision in \emph{loc.cit.}, which should read ``a main formula of an inference rule or axiom, with only weakenings intervening'' see also \cite[Example 1]{mints}.} and introduce global structure into the tree (since the weakening on the right branch of $p \imp p$ reflects the appearance of this formula on the left branch). The only other difference to \eqref{eq:2_final_normal} is the order of the $(R \imp), (\operatorname{ctr})$ rules. 

In general, in a Mints normal form contraction rules take place as late as possible (that is, as close as possible to the bottom of the proof tree) whereas in our normal form these rules take place as early as possible. This encapsulation means that our normal forms are composable in a way that Mints normal forms are not. For example, denoting by $\underline{2}$ the normal preproof of \eqref{eq:2_final_normal}, the preproof
\begin{center}
\AxiomC{$\underline{2}$}
\noLine
\UnaryInfC{$\vdots$}
\noLine
\UnaryInfC{$y: p \imp p \vdash M:p \imp p$}
\AxiomC{}
\RightLabel{$(\operatorname{ax})$}
\UnaryInfC{$z:q \vdash q$}
\RightLabel{$(L \imp)$}
\BinaryInfC{$t: (p \imp p) \imp q, y: p \imp p \vdash (t \, M):q$}
\DisplayProof
\end{center}
is normal. However, appending a similar $(L \imp)$ rule (with attendant weakenings) to \eqref{eq:2_final_normal_mints} results in a derivation that is not in Mints normal form; to obtain the normal form the contractions must be brought down past the $(L \imp)$ rule.

A normal derivation is cut-free, W-normal, C-normal and M-normal. A normal preproof is W-normal and M-normal but not necessarily C-normal (as the above discussion shows). Hence, apart from the differences between LJ and GJ, the only difference between our notion of normality and that of Mints lies in the arrangement of contractions.
\end{remark}

\subsection{Internal BHK}\label{section:internal_bhk}

What is the intuitionist logical reading of the $(L \imp)$ rule in sequent calculus? Let us first recall that the Brouwer-Heyting-Kolmogorov (BHK) interpretation of intuitionistic propositional logic, as given by Heyting in \cite[\S 7.1.1]{heyting} and Troelstra-van Dalen    in \cite[Chapter 1, \S 3.1, \S 5.3]{troelstra}, gives the following interpretation of the logical sign $\imp$. The following quote is from \cite[\S 7.1.1]{heyting}:
\begin{displayquote}
``The implication $\mathfrak{p} \rightarrow \mathfrak{q}$ can be asserted, if and only if we possess a construction $\mathfrak{r}$, which, joined to any construction proving $\mathfrak{p}$ (supposing the latter be effected), would automatically effect a construction proving $\mathfrak{q}$. In other words, a proof of $\mathfrak{p}$, together with $\mathfrak{r}$, would form a proof of $\mathfrak{q}$.''
\end{displayquote} 
The justification of the deduction rules of natural deduction by the BHK-interpretation is given for example in \cite[\S 1.2, \S 1.4]{troelstra}. Let us briefly provide an analogue for sequent calculus of the discussion in \cite[\S 1.2]{troelstra}, using the same language. Suppose we have established $q$ repeatedly appealing to assumption $p$. This means that we have shown how to construct a proof of $q$ from hypothetical proof of $p$; thus on the BHK-interpretation this means that we have established the implication $p \imp q$ and this justifies the $(R \imp)$ rule of sequent calculus by the same argument justifying introduction for $\imp$ in natural deduction. Consider now the following simplified form of the $(L \imp)$ rule in sequent calculus
\begin{center}
\begin{tabular}{ >{\centering}m{10cm} >{\centering}m{0.5cm}}
        \AxiomC{$\vdash p$}
        \AxiomC{$x:q \vdash r$}
        \RightLabel{$(L \imp)$}
        \BinaryInfC{$y: p \imp q \vdash r$}
        \DisplayProof
        &
        \tagarray{\label{eq:bhk_imp}}
\end{tabular}
\end{center}
Suppose that we have shown how to construct a proof of $p$, and a proof of $r$ from a hypothetical proof of $q$. Then we can we may construct a proof of $r$ from a hypothetical proof of $p \imp q$ according to the following recipe. Given the earlier justification of the $(R \imp)$ rule, to prove $p \imp q$ we must possess a construction of a proof of $q$ from a hypothetical proof of $p$. Enact this construction on the given proof of $p$, and enact on the resulting proof of $q$ the construction which produces from such an object a proof of $r$.
\\

With this intuitionist reading of $(L \imp)$ in hand let us now consider the logical status of the following simplified forms of the rules \eqref{lambda_L_L_ctr} and \eqref{lambda_weak_L}:
    \begin{center}
    \begin{tabular}{ >{\centering}m{5cm} >{\centering}m{0.5cm} >{\centering}m{6cm} >{\centering}m{0.5cm}}
        \AxiomC{$\vdash p$}
        \AxiomC{$x:q, x':q \vdash r$}
        \RightLabel{$(\operatorname{ctr})$}
        \UnaryInfC{$x:q \vdash r$}
        \RightLabel{$(L \imp)$}
        \BinaryInfC{$y: p \imp q \vdash r$}
        \DisplayProof
        
        &
        $\sim_\lambda$
        &

        \AxiomC{$\vdash p$}
        \AxiomC{$\vdash p$}
        \AxiomC{$x:q, x': q \vdash r$}
        \RightLabel{$(L \imp)$}
        \BinaryInfC{$x:q, y': p \imp q \vdash r$}
        \RightLabel{$(L \imp)$}
        \BinaryInfC{$y: p\imp q, y': p\imp q \vdash r$}
        \RightLabel{$(\operatorname{ctr})$}
        \UnaryInfC{$y: p \imp q\vdash r$}
        \DisplayProof        
        &
        \tagarray{\label{lambda_L_L_ctr_bhk}}
    \end{tabular}
    \end{center}
    
\begin{center}
\begin{tabular}{ >{\centering}m{5cm} >{\centering}m{0.5cm} >{\centering}m{6cm} >{\centering}m{0.5cm}} 
     \AxiomC{$\vdash p$}
     \AxiomC{$\vdash r$}
     \RightLabel{$(\operatorname{weak})$}
     \UnaryInfC{$x: q \vdash r$}
     \RightLabel{$(L \imp)$}
     \BinaryInfC{$y: p \imp q \vdash r$}
     \DisplayProof &$\sim_\lambda$&
     \AxiomC{$\vdash r$}
     \RightLabel{$(\operatorname{weak})$}
     \UnaryInfC{$y: p \imp q\vdash r$}
     \DisplayProof
     &
     \tagarray{\label{lambda_weak_L_bhk}}
\end{tabular}
\end{center}
What is the intuitionist logical reading of the $(\operatorname{ctr})$ rule in sequent calculus? There are at least two. Suppose we have shown how to construct a proof of $r$ from two hypothetical proofs of formulas $q,q'$ that just happen to be the same, that is $q = q'$. Joining this with a construction of a proof of $q$ and a construction of a proof of $q'$ certainly effects a construction of a proof of $r$. The question is: does simply stating $q = q'$ and showing a single construction of a proof of $q$ suffice as a construction of a proof of $r$? One possible answer is ``yes it suffices, because we can simply run the construction of a proof of $q$ and copy the result'' and another is ``yes it suffices, because however many copies are required, we can repeat the construction that number of times (necessarily entailing the repetition of earlier constructions that feed into this one)''. Principle \eqref{lambda_L_L_ctr_bhk} and the corresponding cut-elimination rule \eqref{cut:log_vs_ctr} correspond to the endorsement of the second possible reading: there is no fundamental operation of ``copying'' when it comes to constructions of proofs. This is another logical principle (the first being coalgebraic structure, see Definition \ref{co_equivalence}) that is emphasised by linear logic.

The intuitionist logical reading of $(\operatorname{weak})$ is that any construction of a proof of $r$ is also a construction of a proof of $r$ from a hypothetical proof of $q$, which is ``ignored'' during the construction. The question here: is ignoring a hypothetical proof of $q$, constructed from a proof of $p$ by a hypothetical proof of $p \imp q$, the same as ignoring the hypothetical proof of $p \imp q$? One possible answer is ``no, because in the former case more information is discarded than in the latter'' and another is ``yes it is the same, I do not believe in a logical distinction between ignoring a machine and ignoring all of its outputs''. Principle \eqref{lambda_weak_L_bhk} and the corresponding cut-elimination rule \eqref{cut:log_vs_weak} endorse the second reading.

\begin{remark}[(Internal vs external composition)]\label{remark:limp_internal_cut}
Composition in the category $\cat{S}$ (recall that this means $\cat{S}_\Gamma$ with $\Gamma$ empty) which we henceforth refer to as \emph{external} composition, is effected via the $(\operatorname{cut})$ rule. The \emph{internal} composition, in the sense of the theory of Cartesian closed categories, is a morphism $\kappa \in \cat{S}(p, (p \imp q) \imp q)$ given by
\begin{center}
\begin{tabular}{ >{\centering}m{10cm} >{\centering}m{0.5cm}}
        \AxiomC{}
        \UnaryInfC{$x:p \vdash p$}
        \AxiomC{}
        \UnaryInfC{$y:q \vdash q$}
        \RightLabel{$(L \imp)$}
        \BinaryInfC{$x:p, z: p \imp q \vdash q$}
        \RightLabel{$(R \imp)$}
        \UnaryInfC{$x:p \vdash (p \imp q) \imp q$}
        \RightLabel{$(R \imp)$}
        \UnaryInfC{$\vdash p \imp ( (p \imp q) \imp q)$}
        \DisplayProof
        &
        \tagarray{\label{eq:bhk_imp2}}
\end{tabular}
\end{center}
In this sense the $(L \imp)$ rule determines structure on the category $\cat{S}$ which internalises composition, and thus the $(\operatorname{cut})$ rule. To make this connection fully precise, let us compare the rules for $(\operatorname{cut})$ with those for $(L \imp)$ in our sequent calculus system:
\begin{itemize}
\item \eqref{cut:struc_vs_any} for $(\operatorname{cut})$ vs \eqref{comm_weak_L},\eqref{comm_L_ctr},\eqref{comm_L_ex} for $(L \imp)$.
\item \eqref{cut:L_vs_any} for $(\operatorname{cut})$ vs \eqref{comm_L_L2} for $(L \imp)$.
\item \eqref{cut:log_vs_struc_np} for $(\operatorname{cut})$ vs \eqref{comm_weak_L},\eqref{comm_L_ctr2},\eqref{comm_L_ex2} for $(L \imp)$.
\item \eqref{cut:log_vs_ctr} for $(\operatorname{cut})$ vs \eqref{lambda_L_L_ctr} for $(L \imp)$.
\item \eqref{cut:log_vs_weak} for $(\operatorname{cut})$ vs \eqref{lambda_weak_L} for $(L \imp)$.
\item \eqref{cut:R_vs_R} for $(\operatorname{cut})$ vs \eqref{comm_R_L} for $(L \imp)$.
\item \eqref{cut:R_vs_L} for $(\operatorname{cut})$ has no analogue for $(L \imp)$.
\item \eqref{cut:R_vs_L_nonp} for $(\operatorname{cut})$ vs \eqref{comm_L_L} for $(L \imp)$.
\item \eqref{cut:R_vs_L_nonp2} for $(\operatorname{cut})$ vs \eqref{comm_L_L2} for $(L \imp)$.
\end{itemize}
\end{remark}

\subsection{Local vs global}\label{section:comparison}

As elaborated in the introduction, Gentzen-Mints-Zucker duality is interesting because sequent calculus proofs and lambda terms are different. The principal difference is that the structure of sequent calculus is \emph{local} while that of lambda calculus is \emph{global}.

Let us first collect some preliminary comments. Theorem \ref{gentzen_mints_zucker} can be read as saying that the ``true'' proof objects are $\beta\eta$-equivalence classes of lambda terms (or via the Curry-Howard correspondence, natural deduction proofs) since there is, up to $\alpha$-equivalence, a unique such object representing every morphism in $\cat{S}_\Gamma$. From this point of view sequent calculus is a system that enables us to work on these objects \cite[p.39]{girard_blind} and a proof in sequent calculus ``can be looked upon as an instruction on how to construct a corresponding natural deduction'' \cite[\S A.2]{prawitz} (although see \cite[\S 1.5.1]{zucker}). This raises a natural question: what advantages does this more complicated object, the sequent calculus proof, have over the lambda term that it constructs?
\\

This brings us to the issue of global structure in lambda calculus. A variable $x:p$ may occur multiple times as a free variable in a term $M$, and hence $\beta$-reduction involves global coordination: reducing $(\lambda x.M) N$ to $M[ x:= N ]$ may make arbitrarily many ``simultaneous'' substitutions. This global rewriting is the principal reason that time complexity is difficult to analyse directly in lambda calculus. If $\pi$ is a well-labelled preproof with $f^\Gamma_p(\pi) = M$ then $\pi$ contains, in the form of the contraction tree of the occurrence $x:p$ in $\Gamma$, a specification for \emph{how} any two occurrences of $x:p$ in $M$ are equal and it is by use of this information that cut-elimination is able to present a refinement of $\beta$-reduction which is local, in the precise sense that the relation \eqref{cut:log_vs_ctr} represents copying a term only once. The advantage of the sequent calculus proof is that it provides the ``missing'' structural rules $(\operatorname{ctr}), (\operatorname{weak})$ and $(\operatorname{ex})$ that allow global $\beta$-reduction steps to be replaced by more local transformations.

The generating relations of proof equivalence for sequent calculus preproofs also involve global changes to preproof trees, so this dichotomy between local and global needs to be understood in the proper sense. The minor examples are $\alpha$-equivalence and the strong ancestor substitution in \eqref{cut:ax_left}. The more important instances are the generating relations \eqref{lambda_L_L_ctr} and \eqref{cut:log_vs_ctr} which copy a branch and \eqref{lambda_weak_L} and \eqref{cut:log_vs_weak} which delete a branch; various other relations rearrange branches. Apart from $\alpha$-equivalence and this copying, deleting and rearranging of branches, the changes in the proof tree are localised to a small group of nearby vertices, edges and their labels and in this sense the generating relations of proof equivalence are local. For further discussion of ``locality'' in the context of differences between sequent calculus and natural deduction see \cite[\S 3]{negriplato} and \cite{negri}.


\subsection{Related work}\label{section:related_work}

We have already discussed in some detail the relation of our work to that of Zucker \cite{zucker} and Mints \cite{mints}, see Remark \ref{remark:zucker_1}, Remark \ref{remark:zucker_2}, Remark \ref{opencurryhoward} and Remark \ref{remark:mints_normal_2}. In this section we contrast our approach to that of  Dyckhoff-Pinto \cite{dyckhoffpinto} and Pottinger \cite{pottinger}.

The most important differences between sequent calculus and natural deduction are the explicit structural rules in the former, and the fact that sequent calculus has a left introduction rule for $\imp$ whereas natural deduction has an elimination rule
\begin{center}
\begin{tabular}{ >{\centering}m{10cm} >{\centering}m{0.5cm}}
        \AxiomC{$\Gamma \vdash p \imp q$}
        \AxiomC{$\Gamma \vdash p$}
        \RightLabel{$(\imp E)$}
        \BinaryInfC{$\Gamma \vdash q$}
        \DisplayProof
        &
        \tagarray{\label{eq:nd_elim}}
\end{tabular}
\end{center}
We refer to the fact that the antecedent $\Gamma$ is the same in all the sequents appearing in the elimination rule by saying that the antecedents are \emph{synchronised}. This allows for a form of contraction in natural deduction, as shown in the following example.

\begin{example} Compare the Church numeral $\underline{2}$ in sequent calculus (Example \ref{example:church_2}) to the natural deduction
\begin{center}
\begin{tabular}{ >{\centering}m{10cm} >{\centering}m{0.5cm}}
        \AxiomC{}
        \RightLabel{$({\operatorname{ax}})_f$}
        \UnaryInfC{$\Gamma \vdash p \imp p$}
        \AxiomC{}
        \RightLabel{$({\operatorname{ax}})_f$}
        \UnaryInfC{$\Gamma \vdash p \imp p$}
        \AxiomC{}
        \RightLabel{$({\operatorname{ax}})_x$}
        \UnaryInfC{$\Gamma \vdash p$}
        \RightLabel{$(\imp E)$}
        \BinaryInfC{$\Gamma \vdash p$}
        \RightLabel{$(\imp E)$}
        \BinaryInfC{$f: p \imp p, x: p \vdash p$}
        \RightLabel{$(\imp I)_x$}
        \UnaryInfC{$f: p \imp p \vdash p \imp p$}
        \RightLabel{$(\imp I)_f$}
        \UnaryInfC{$\vdash (p \imp p) \imp (p \imp p)$}
        \DisplayProof
        &
        \tagarray{\label{eq:nd_2}}
\end{tabular}
\end{center}
where $\Gamma = \{ f: p \imp p, x: p \}$. Here we follow natural deduction as presented in \cite[\S 6.4, \S 6.5]{selinger} noting that any presentation of natural deduction that arrives at a bijection between deductions and lambda terms must have a similar flavour. 

Observe that the synchronised antecedent in the $\imp$ elimination rule allows for a form of contraction on $f$ at the cost of introducing global structure (the axiom rules $(\operatorname{ax})_f$ must include $x$ in the antecedent, and the rule $(\operatorname{ax})_x$ must include $f$).
\end{example}

There are a variety of systems which are ``in between'' sequent calculus and natural deduction in the sense that they either modify the $(L \imp)$ rule of sequent calculus to be more like $\imp$ elimination, or they omit some or all of the structural rules; see \cite{negri} and \cite[\S 2.2, \S 2.3, \S 2.4]{negriplato}. For example, Mints uses a system very similar to Kleene's G \cite{kleene} which modifies the $(L \imp)$ rule to have synchronised antecedent and drops exchange but keeps weakening and contraction, Zucker has a standard $(L \imp)$ rule but omits weakening and exchange, Dyckhoff-Pinto \cite{dyckhoffpinto} consider a system like Kleene's G but without any structural rules, and Pottinger \cite{pottinger} revisits the work of Zucker for a sequent calculus system without structural rules but with a standard $(L \imp)$ rule; see \cite[p.331]{pottinger}.

The motivation for these modifications appears to be primarily technical: it is easier to analyse the relationship between ``sequent calculus'' and natural deduction (or lambda calculus) if the former is redefined to be more similar to the latter. There are also applications in logic programming and proof search \cite{dyckhoffpinto} where the simplified systems are sufficient. However these modifications come at the price of introducing global structure into proofs: the synchronised antecedent of the $(L \imp)$ rule of Kleene's G necessitates changes to the Church numeral $\underline{2}$ along the lines of the natural deduction version \eqref{eq:nd_2} (see Remark \ref{remark:mints_normal_2}) and omitting structural rules works against the local nature of cut-elimination as discussed in Section \ref{section:comparison}. Since in our view it is the duality between local and global that makes the comparison of sequent calculus and natural deduction interesting, it seems desirable to avoid these compromises.
\\

Another line of development relating sequent calculus to lambda calculus due to Herbelin \cite{herbelin} builds on the work of Zucker \cite{zucker} by considering a restricted set of cut-free proofs in sequent calculus and showing that this is isomorphic to a form of lambda calculus with explicit substitutions. This yields a close alignment between cut-elimination and $\beta$-reduction. It is not the purpose of the present paper to study such alignment.



\appendix

\section{Background on lambda calculus}\label{section:intro_lambda}

In the simply-typed lambda calculus \cite[Chapter 3]{sorensen} there is an infinite set of \emph{atomic types} and the set $\Phi_{\typearrow}$ of \emph{simple types} is built up from the atomic types using $\typearrow$. Let $\Lambda'$ denote the set of untyped lambda calculus preterms in these variables, as defined in \cite[Chapter 1]{sorensen}. We define a subset $\Lambda'_{wt} \subseteq \Lambda'$ of \emph{well-typed} preterms, together with a function $t: \Lambda'_{wt} \lto \Phi_{\typearrow}$ by induction:
\begin{itemize}
\item all variables $x : \sigma$ are well-typed and $t(x) = \sigma$,
\item if $M = (P \, Q)$ and $P,Q$ are well-typed with $t(P) = \sigma \typearrow \tau$ and $t(Q) = \sigma$ for some $\sigma, \tau$ then $M$ is well-typed and $t(M) = \tau$,
\item if $M = \lambda x\ldot N$ with $N$ well-typed, then $M$ is well-typed and $T(M) = t(x) \typearrow t(N)$.
\end{itemize}
We define $\Lambda'_\sigma = \{ M \in \Lambda'_{wt} \l t(M) = \sigma \}$ and call these \emph{preterms of type $\sigma$}. Next we observe that $\Lambda'_{wt} \subseteq \Lambda'$ is closed under the relation of $\alpha$-equivalence on $\Lambda'$, as long as we understand $\alpha$-equivalence type by type, that is, we take
\[
\lambda x \ldot M =_\alpha \lambda y\ldot M[x := y]
\]
as long as $t(x) = t(y)$. Denoting this relation by $=_\alpha$, we may therefore define the sets of \emph{well-typed lambda terms} and \emph{well-typed lambda terms of type $\sigma$}, respectively:
\begin{align}
\Lambda_{wt} &= \Lambda'_{wt} / =_\alpha\,\\
\Lambda_\sigma &= \Lambda'_\sigma / =_\alpha\,.
\end{align}
Note that $\Lambda_{wt}$ is the disjoint union over all $\sigma \in \Phi_{\typearrow}$ of $\Lambda_\sigma$. We write $M: \sigma$ as a synonym for $[M] \in \Lambda_\sigma$, and call these equivalence classes \emph{terms of type $\sigma$}. Since terms are, by definition, $\alpha$-equivalence classes, the expression $M = N$ henceforth means $M =_\alpha N$ unless indicated otherwise. We denote the set of free variables of a term $M$ by $\FV(M)$.


\begin{definition}\label{defn:subst} The substitution operation on lambda terms is a family of functions
 \[
\big\{ \operatorname{subst}_\sigma: Y_\sigma \times \Lambda_{\sigma} \times \Lambda_{wt} \longrightarrow \Lambda_{wt} \big\}_{\sigma \in \Phi_{\rightarrow}} \]
We write $M[ x:= N ]$ for $\operatorname{subst}_\sigma( x, N, M )$ and this term is defined inductively (on the structure of $M$) as follows:
 \begin{itemize}
  \item if $M$ is a variable then either $M = x$ in which case $M[x := N] = N$, or $M \neq x$ in which case $M[ x:= N ] = M$.
  \item if $M = (M_1 \, M_2)$ then $M[x := N] = \big(M_1[x:= N] \, M_2[x := N]\big)$.
  \item if $M = \lambda y. L$ we may assume by $\alpha$-equivalence that $y \neq x$ and that $y$ does not occur in $N$ and set $M[x := N] = \lambda y. L[ x := N ]$. 
 \end{itemize}
 Note that if $x \notin \operatorname{FV}(M)$ then $M[ x:= N] = M$.
\end{definition}

\bibliographystyle{amsalpha}
\providecommand{\bysame}{\leavevmode\hbox to3em{\hrulefill}\thinspace}
\providecommand{\href}[2]{#2}

\end{document}